\documentclass[11pt,reqno]{amsart}

\usepackage{amsmath,amssymb,amsthm}
\usepackage[margin=1.1in]{geometry}
\usepackage{microtype}
\usepackage{booktabs}
\usepackage{tikz}
\usepackage[colorlinks=true,linkcolor=blue,citecolor=blue]{hyperref}

\theoremstyle{plain}
\newtheorem{theorem}{Theorem}[section]
\newtheorem{proposition}[theorem]{Proposition}
\newtheorem{lemma}[theorem]{Lemma}
\newtheorem{corollary}[theorem]{Corollary}
\theoremstyle{definition}
\newtheorem{remark}[theorem]{Remark}
\newtheorem{definition}[theorem]{Definition}
\newtheorem{conjecture}[theorem]{Conjecture}

\newcommand{\R}{\mathbb{R}}
\newcommand{\Z}{\mathbb{Z}}
\newcommand{\Q}{\mathbb{Q}}
\newcommand{\C}{\mathbb{C}}
\newcommand{\F}{\mathbb{F}}
\newcommand{\SU}{\mathrm{SU}(2)}
\newcommand{\tr}{\operatorname{tr}}
\newcommand{\Inat}{I^{\natural}}
\newcommand{\HF}{\mathit{HF}}
\newcommand{\Kh}{\mathit{Kh}}
\newcommand{\gr}{\operatorname{gr}}

\newcommand{\Hnat}{H^{\natural}}
\newcommand{\Pst}{P^{*}}
\newcommand{\tC}{\widetilde C}
\numberwithin{equation}{section}

\begin{document}

\title[Instanton and pillowcase homology of $(-2,3,q)$ pretzel knots]
{Instanton and pillowcase homology of the $(-2,3,q)$ pretzel knots}

\author{Bernd Johannes Wuebben}
\date{First version: July 27, 2026. This version: October 6, 2026}

\begin{abstract}
Hedden, Herald and Kirk conjectured that every knot admits a decomposition along a Conway sphere for which the
Lagrangian Floer homology of the two associated immersed curves in the pillowcase recovers Kronheimer and Mrowka's
reduced singular instanton homology. We prove the conjecture for the hyperbolic pretzel knots $P(-2,3,q)$; it was
previously known for two-bridge knots and for some torus knots. The decomposition is obtained from Hedden, Herald
and Kirk's decomposition of the torus knot $T(3,5)$ by Dehn twists along the Conway sphere, which act on the
pillowcase by a linear shear. Building on our analysis of the torus knots $T(3,n)$, we compute the Floer complex and its $\Z/4$ grading by hand; its differential is nonzero for every member
of the family except $P(-2,3,7)$. The same method proves the conjecture for a family of twisted torus knots, under a hypothesis on the chirality of the torus-knot decomposition. On the instanton side, we show that the
instanton homology of $P(-2,3,q)$ is free abelian of rank $q+2$ for every odd $q\ge3$, extending Lobb and Zentner's
rational computation to the integers by means of Manion's integral Khovanov homology; for the hyperbolic members
this also follows from work of Daemi and Scaduto.
\end{abstract}

\maketitle

\section{Introduction}\label{sec:intro}

\subsection{Two homologies for one knot}
Let $K\subset S^3$ be a knot. On the gauge-theoretic side sits the reduced singular instanton knot
homology $\Inat(K)$ of Kronheimer and Mrowka~\cite{KM1,KM2}. Its rank is bounded below by the
coefficient norm of the Alexander polynomial and above by the rank of reduced Khovanov homology
\eqref{eq:ell}. On the symplectic side sits the pillowcase program of Hedden, Herald and
Kirk~\cite{HHK1,HHK2}. Write $K=(Y,T)\cup(D,U)$ as the union of a tangle and the trivial tangle along a
Conway sphere. The traceless $\SU$ character variety of the four-punctured sphere is the
\emph{pillowcase}, a two-sphere with four orbifold points, and after a holonomy perturbation each side
determines an immersed curve in its smooth part: the tangle side a curve $L_1$, and the trivial side,
with an earring, the curve $L_0$, which depends on a small earring parameter $\epsilon>0$ (\S\ref{ss:earring}). Their Lagrangian Floer homology $\Hnat(Y,T,\pi)$ is defined when $L_1$ is restricted and the pair is admissible (Definitions~\ref{def:restricted} and~\ref{def:admissible}), and Hedden, Herald and Kirk conjecture that for suitable perturbations it is
isomorphic to $\Inat(K)$~\cite[Conj.~6.5]{HHK2} (Conjecture~\ref{conj:hhk65}), and they compare the $\Z/4$
gradings of the two sides~\cite[\S\S6.1, 11]{HHK2}.
For general decompositions figure-eight bubbles at the self-intersections of the curves deform the
differential. Cazassus, Herald, Kirk and Kotelskiy conjecture that each tangle is assigned an immersed
curve together with a bounding cochain, and that the deformed theory recovers
$\Inat(K)$~\cite[Conj.~D]{CHKK}; conditional on this conjecture, Smith proves that the cochain on one
Conway-sum tangle in $P(-2,3,5)$ cannot vanish~\cite[proof of Thm.~5.9]{Smith}.

This paper studies both sides for the pretzel knots $K_q=P(-2,3,q)$, $q\ge3$ odd (\S\ref{ss:knots}). The
smallest members are the torus knots $K_3=T(3,4)$ and $K_5=T(3,5)$, whose instanton and pillowcase
homology are treated in~\cite{WuebbenI,WuebbenIV,WuebbenV}, and for $q\ge7$ the knot $K_q$ is
hyperbolic. We decompose $K_q$ by regluing the tangle of Hedden, Herald and Kirk for $T(3,5)$ to the
trivial tangle by Dehn twists along the Conway sphere. For this decomposition we compute the pillowcase
homology with zero bounding cochain and find that it agrees with $\Inat(K_q)$; this proves the conjecture of
Hedden, Herald and Kirk for every $K_q$, and the same method applies to a family of twisted torus knots
(Theorems~\ref{thm:hhk} and~\ref{thm:twisted}). Smith decomposes the same knots as a sum of rational tangles,
with the earring on the $-2$ band. For that decomposition the pairing with zero bounding cochain has the wrong
rank at $q=5$~\cite[proof of Thm.~5.9]{Smith}, and also at $q=7$ for a piecewise-linear model of Smith's
curve~\cite{WuebbenIIb}. The companion paper~\cite{WuebbenIIb} constructs, for $q=5,7$, Maurer--Cartan elements
that correct the rank, in the two-arc algebra of Kotelskiy, Watson and Zibrowius~\cite{KWZ}; the instanton
dimensions of Theorem~\ref{thm:inat} are its targets.

\subsection{The instanton side}
\begin{theorem}\label{thm:inat}
For every odd $q\ge3$, $\Inat(K_q;\Z)\cong\Z^{q+2}$. For $q\ge7$ the signature is $\sigma(K_q)=-(q+1)$,
and in Kronheimer and Mrowka's absolute grading the graded ranks of $\Inat(K_q)$ are
\[
\bigl(1+N_3,\ N_1,\ N_1,\ N_3\bigr)\quad\text{in degrees }\gr-\sigma(K_q)=0,1,2,3,\qquad
N_1=\bigl\lceil\tfrac{q+1}4\bigr\rceil,\ N_3=\bigl\lfloor\tfrac{q+1}4\bigr\rfloor .
\]
\end{theorem}

The rank over $\Q$ is due to Lobb and Zentner~\cite[Prop.~4.7]{LZ}. For $q\ge7$ the freeness and the
graded ranks follow from the computation of the irreducible instanton homology of $K_q$ by Daemi and
Scaduto~\cite[Prop.~8.30]{DS24} together with their integral spectral sequence
(Proposition~\ref{prop:inat-graded}). Our proof of freeness is uniform in $q$: it runs Lobb and
Zentner's two-sided rank bound with integer coefficients, using Manion's integral computation of Khovanov
homology~\cite{Manion} and the integral spectral sequence of Kronheimer and Mrowka. For $q=3,5$ the
graded ranks are those of the torus knots in~\cite{WuebbenV}. Section~\ref{sec:inat} gives the proof and
Remark~\ref{rem:inat-attribution} the remaining attributions.

\subsection{The conjecture of Hedden, Herald and Kirk for the pretzel knots}
Hedden, Herald and Kirk compute the pillowcase homology of $T(3,5)$ from a tangle $(Y,T)$ whose
complement has free fundamental group of rank two, glued to the trivial tangle~\cite[\S\S10--11]{HHK2}.
Let $t_e$ be the Dehn twist along the curve of the Conway sphere that separates two of the punctures from
the other two. Regluing by $t_e^{-m}$ adds $m$ full twists on two strands, and the resulting knot is
$K_{5+2m}$ (Theorem~\ref{thm:topology}). That the original gluing gives the positive torus knot $T(3,5)$ is verified by a
computer-assisted argument (Remark~\ref{rem:chirality}). The unperturbed character variety $R(Y,T)$ is the union
of an arc and a circle that meet in two points $c_\pm$ (\S\ref{ss:straight}). The holonomy perturbations on the
tangle side are those of~\cite[\S10]{HHK2} used in~\cite{WuebbenV}: they are supported near two curves in $Y$,
with perturbation functions $\varepsilon_A\sin$ and $\varepsilon_B\sin$, and items (P1)--(P9) in
\S\S\ref{ss:straight}--\ref{ss:perturbed-P} recall the facts about them that we use from~\cite{WuebbenV}. For $a>0$ the closed cone
$\mathcal C_a$ is the set of parameters $\varepsilon=(\varepsilon_A,\varepsilon_B)$ such that, at both $c=c_+$ and $c=c_-$, the
linear part of $\kappa_c(\varepsilon)$, the critical value near $c$ of the function to which the equations of the
perturbed variety reduce near $c$, has absolute value at least $a|\varepsilon|$~\cite[Def.~5.5]{WuebbenV}; it is the
complement of an open conical neighborhood of the two lines $\{\varepsilon_B=\mp\sqrt3\,\varepsilon_A\}$
(Proposition~\ref{prop:restype}). The generator $r_+$ below is the intersection point of $L_0$ and $L_1$ near the
corner $(0,0)$ that comes from the abelian representation of the knot group (\S\ref{ss:decomp}).

\begin{theorem}\label{thm:hhk}
Let $q\ge7$ be odd, and let $K_q=(Y,T)\cup_{t_e^{-(q-5)/2}}(D,U)$ be the decomposition above. Fix $a>0$.
For every sufficiently small $\varepsilon\in\mathcal C_a\setminus\{0\}$ and every sufficiently small earring
parameter:
\begin{enumerate}
\item $L_1$ is a restricted immersed $1$-manifold, and $(L_0,L_1)$ is admissible;
\item the Floer complex has $q+2+4N_q$ generators, where $N_q=\lfloor 5(q-6)/12\rfloor-\lfloor(q-6)/12\rfloor$;
\item the differential has rank $2N_q$, and it is nonzero for $q\ge9$;
\item with the absolute grading normalized by $\gr(r_+)=\sigma(K_q)$, as in~\cite[(19)]{HHK2}, and with any absolute grading on the circle component of $L_1$ when there is one, there is an isomorphism $\Hnat(Y,T,\pi)\cong\Inat(K_q;\F_2)$ of $\Z/4$-graded vector spaces.
\end{enumerate}
In particular, the conjecture of Hedden, Herald and Kirk~\cite[Conj.~6.5]{HHK2} holds for every $K_q$.
\end{theorem}

The perturbation resolves the two double points $c_\pm$ of $R(Y,T)$ in one of two ways, joined for
$|\varepsilon_B|<\sqrt3\,|\varepsilon_A|$ and split for $|\varepsilon_B|>\sqrt3\,|\varepsilon_A|$ (\S\ref{ss:perturbed-P} and
Proposition~\ref{prop:restype}); for small $a$ both occur in $\mathcal C_a$, and the theorem holds in both. For $q\ge9$ the complex is not minimal. Hedden, Herald and
Kirk verify the conclusion of their conjecture for two-bridge knots~\cite[\S7]{HHK2}, a family that
contains hyperbolic knots, and for several torus knots, among them $T(3,4)$, $T(3,5)$ and $T(3,7)$~\cite[\S\S10--11]{HHK2}; the torus knots
$T(3,n)$ are treated in~\cite{WuebbenV}. Fukumoto, Kirk and Pinz\'on-Caicedo describe the unperturbed
character variety of the pretzel knots for a decomposition of their own and count $9$ generators for
$P(-2,3,7)$~\cite[\S5]{FKP}
(Remark~\ref{rem:FKPcount}). For $q\ge7$ the knot $K_q$ is hyperbolic and not two-bridge
(Proposition~\ref{prop:hyperbolic}). To our knowledge, these are the first hyperbolic knots outside the
two-bridge family for which the conjecture is verified, with the $\Z/4$ gradings and nonzero Floer
differentials.

The same regluing applies to the tangles of~\cite[\S10]{HHK2} for $T(3,n)$. The \emph{twisted torus knot}
$T(3,n;2,m)$ is the closure of the braid $(\sigma_1\sigma_2)^n\sigma_1^{2m}$.

\begin{theorem}\label{thm:twisted}
Let $n\ge7$ with $n\equiv1$ or $2\pmod 6$, let $m\ge1$, and let $K=T(3,n;2,m)$. Then
$\sigma(K)=\sigma(T(3,n))-2m$, and $\Inat(K;\Z)$ is free of rank $|\sigma(K)|+1$. Assume, as in~\cite{HHK2}
and~\cite{WuebbenV}, that the tangle of~\cite[\S10]{HHK2} for $T(3,n)$ glued to the trivial tangle is the positive
torus knot $T(3,n)$. Then this tangle, reglued by $t_e^{-m}$, is a decomposition of $K$, and for small holonomy
perturbations and generic small earring parameters it satisfies the conclusion of~\cite[Conj.~6.5]{HHK2}:
$\Hnat\cong\Inat(K;\F_2)$ as $\Z/4$-graded vector spaces.
\end{theorem}

Theorem~\ref{thm:H} gives the precise form, with the graded ranks. The family of
Theorem~\ref{thm:twisted} contains torus knots, $T(3,6j+1;2,1)=T(3,6j+2)$, and by braid conjugacy
$T(3,6j+2;2,m)=T(3,6j+1;2,m+1)$ (Lemma~\ref{lem:conj}). The other members, $T(3,6j+1;2,M)$ with $j\ge1$
and $M\ge2$, are not torus knots, not alternating and not pretzel knots
$P(-2,3,q)$ (Lemma~\ref{lem:twisted-not-torus}). We make no claim about their hyperbolicity.

\subsection{Outline of the arguments}
\emph{Theorem~\ref{thm:hhk}.} The regluing replaces $L_1$ by its image under the linear shear
$(\gamma,\theta)\mapsto(\gamma,\theta+(q-5)\gamma)$ of the pillowcase (Lemma~\ref{lem:shear}). The unperturbed
character variety of the $T(3,5)$ tangle is the union of an arc of binary dihedral characters and a circle of characters, the \emph{central circle}, meeting it in two
points. A word $w$ in the tangle group with $\rho(w)=1$ for every representation $\rho$ on the central
circle shows that its image in the pillowcase is a union of straight segments (Lemma~\ref{lem:straight} and
Corollary~\ref{cor:polygon}). The perturbed curves are then those of~\cite{WuebbenV}, sheared, and on them
we compute the intersection points with $L_0$, their $\Z/4$ gradings, and the bigons by hand. The
bigons come in chains along the lifts of the earring curve $L_0$ to $\R^2$, which gives the rank $2N_q$ of the
differential (Lemma~\ref{lem:zigzag}). The identification of the reglued knot uses only its
fundamental group: the relator is that of Clay and Watson's presentation of the twisted torus
knots~\cite[Prop.~22]{ClayWatson} (Lemma~\ref{lem:relator}), two-generator knots are prime~\cite{Norwood},
knots are determined by their complements~\cite{GordonLuecke}, and an induction on the signature fixes the
chirality.

\emph{Theorem~\ref{thm:twisted}.} For $n\equiv1,2\pmod6$ the unperturbed variety is an arc with a straight
image and $2\lfloor n/6\rfloor$ circles that miss the edges of the pillowcase~\cite{WuebbenV}. The sheared
arc meets $L_0$ in $1+2M$ points and carries no bigon, and each circle contributes one dimension in each
degree to the homology, by a theorem of Hedden, Herald and Kirk~\cite[Thm.~5.4]{HHK2}. On the knot side, a closed formula for
the Burau matrix of the braid gives the signature and the coefficient norm of the Alexander polynomial,
and the instanton rank follows from the work of Daemi and Scaduto.

\subsection{Scope of claims}
Theorem~\ref{thm:inat} is unconditional. Theorems~\ref{thm:hhk} and~\ref{thm:twisted} rest on the analysis of
the perturbed curves of the torus-knot tangles in~\cite{WuebbenV}. For Theorem~\ref{thm:hhk} the base case is verified by
computer (Remark~\ref{rem:chirality}); for the decomposition in Theorem~\ref{thm:twisted} it is a hypothesis of the
statement, shared with~\cite{HHK2} and~\cite{WuebbenV}, that the decomposition of $T(3,n)$ gives the positive torus
knot, which the knot group cannot detect. Theorem~\ref{thm:twisted} also uses the
computation of Hedden, Herald and Kirk for circles that miss the edges~\cite[Thm.~5.4]{HHK2}, whose proof
in~\cite{HHK2} rests on their invariance theorem~\cite[Thm.~4.1]{HHK2} (Remark~\ref{rem:invariance}). Numerical traces of the perturbed
curves agree with them (Remark~\ref{rem:numericsP}) but are not used in the proofs.

\subsection{Organization}
Section~\ref{sec:background} recalls the pillowcase Floer theory of Hedden, Herald and Kirk.
Section~\ref{sec:inat} proves Theorem~\ref{thm:inat}. Section~\ref{sec:sheared} proves Theorem~\ref{thm:hhk},
and Section~\ref{sec:twisted} proves Theorem~\ref{thm:twisted}. Section~\ref{sec:disc} discusses related work
and open problems, and Appendix~\ref{app:data} describes the verified computation behind
Remark~\ref{rem:chirality} and the computations that corroborate the proofs.

\subsection*{Acknowledgements}
The pillowcase model and the geometry used here are due to Hedden, Herald and Kirk~\cite{HHK1,HHK2} and to
Fukumoto, Kirk and Pinz\'on-Caicedo~\cite{FKP}. The closed-form Khovanov input is Manion's~\cite{Manion}, the Alexander polynomials belong to Hironaka's family~\cite{Hironaka}, the
presentation of the twisted torus knot groups is Clay and Watson's~\cite{ClayWatson}, the rational rank of
the instanton homology is Lobb and Zentner's~\cite{LZ}, and the integral and graded instanton computations
for $q\ge7$ and for twisted torus knots also follow from the work of Daemi and Scaduto~\cite{DS24}.

\section{Background}\label{sec:background}

This section fixes notation and collects the results from the literature that are used later.
\S\S\ref{ss:pillowcase}--\ref{ss:decomp} recall the Lagrangian Floer theory of Hedden, Herald and Kirk in
the pillowcase~\cite{HHK2}, which is the setting of Sections~\ref{sec:sheared} and~\ref{sec:twisted}; results
of~\cite{HHK2} are cited by their numbering in arXiv:1501.00028v1. \S\ref{ss:knots} fixes the conventions for
the knots. We use $\SU$ as the group of unit quaternions, so that a
traceless element is a purely imaginary unit quaternion. Coefficients are in $\F_2$ unless stated otherwise.

\subsection{The pillowcase}\label{ss:pillowcase}
Let $G\cong\Z^2\rtimes\Z/2$ be the group of isometries of $\R^2$ generated by
$(\gamma,\theta)\mapsto(\gamma+2\pi,\theta)$, $(\gamma,\theta)\mapsto(\gamma,\theta+2\pi)$ and
$(\gamma,\theta)\mapsto(-\gamma,-\theta)$. The \emph{pillowcase} is the quotient $P=\R^2/G$, a topological
$2$-sphere. The quotient map $\R^2\to P$ is a branched covering, branched over the images of the lattice points
$(\pi\Z)^2$, which are the four \emph{corners} of $P$~\cite[\S3.1]{HHK2}. The rectangle $[0,\pi]\times[0,2\pi]$ is
a fundamental domain; we call $[0,\pi]\times[0,\pi]$ the front face, $[0,\pi]\times[\pi,2\pi]$ the back face, and
the images of the lines $\gamma\in\pi\Z$ the left and right edges. Put $\Pst=(\R^2\setminus(\pi\Z)^2)/G$, oriented
by $d\gamma\wedge d\theta$. The restriction $\R^2\setminus(\pi\Z)^2\to\Pst$ is a regular covering with deck group
$G$, so paths, loops and immersed discs in $\Pst$ lift to $\R^2\setminus(\pi\Z)^2$; we work with such lifts
throughout. The double cover $\R^2/(2\pi\Z)^2$ of $P$ is a torus, and $P$ is its quotient by the elliptic
involution $(\gamma,\theta)\mapsto(-\gamma,-\theta)$.

The pillowcase is the traceless character variety of the four-punctured sphere. With the presentation
$\pi_1(S^2\setminus\{a,b,c,d\})=\langle a,b,c,d\mid ba=cd\rangle$ of~\cite{HHK1,HHK2}, the assignment
\begin{equation}\label{eq:hhk-coords}
a\mapsto\mathbf i,\qquad b\mapsto e^{\gamma\mathbf k}\mathbf i,\qquad c\mapsto e^{\theta\mathbf k}\mathbf i,\qquad
d\mapsto e^{(\theta-\gamma)\mathbf k}\mathbf i
\end{equation}
induces a homeomorphism of $P$ with the space $R(S^2,\{a,b,c,d\})$ of conjugacy classes of $\SU$ representations
sending each of $a,b,c,d$ to a traceless element; the corners correspond to the abelian
representations~\cite[Prop.~3.1]{HHK1},~\cite[Prop.~6.1]{HHK2}. We use these coordinates throughout.

On $\R^2$ we use the function
\[
\varphi=\theta-\gamma .
\]
Translations in $G$ change $\varphi$ by elements of $2\pi\Z$, and $(\gamma,\theta)\mapsto(-\gamma,-\theta)$ changes
its sign. The \emph{diagonal arc} $\Delta=\{(\gamma,\gamma)\mid\gamma\in[0,\pi]\}\subset P$~\cite[(12)]{HHK2} joins
the corners $(0,0)$ and $(\pi,\pi)$, and its preimage in $\R^2$ is the union of the lines
\[
\mathcal D_k=\{\varphi=2\pi k\},\qquad k\in\Z .
\]
The lines $\{\varphi=\pi+2\pi k\}$ are lifts of the arc of slope one on the back face, which is not part of
$\Delta$. The lattice points on $\mathcal D_k$ are the points $(j\pi,j\pi+2\pi k)$, $j\in\Z$.

\subsection{The earring curve}\label{ss:earring}
For $\epsilon>0$ let $L_0=L_0^{\epsilon,0}\colon S^1\to\Pst$ be the immersion defined by
\begin{equation}\label{eq:L0}
\tilde L_0(t)=\bigl(t+\epsilon\sin t+\tfrac\pi2,\ t-\epsilon\sin t+\tfrac\pi2\bigr),\qquad t\in[0,2\pi],
\end{equation}
followed by the branched covering~\cite[Def.~3.4, (13)]{HHK2}; as in~\cite[\S11]{HHK2} we take the perturbation
function $g$ of~\cite[(13)]{HHK2} to be zero. It is the restriction to the boundary of the traceless character
variety of the trivial tangle with an earring, after a small holonomy perturbation~\cite[Thm.~6.2]{HHK2}. For
small $\epsilon$ it is self-transverse with one double point, at $(\pi/2,\pi/2)$, it lies in an
$O(\epsilon)$-neighborhood of $\Delta$, and it passes around the corners $(0,0)$ and $(\pi,\pi)$ through the back
face.

Since $\tilde L_0(t+2\pi)=\tilde L_0(t)+(2\pi,2\pi)$, the preimage of $L_0$ in $\R^2$ is the union of the curves
\[
S^k_+=\tilde L_0+(0,2\pi k),\qquad S^k_-=-\tilde L_0+(0,2\pi k),\qquad k\in\Z,
\]
which we call the \emph{lifts} of $L_0$. Each is a properly embedded line within $\sqrt2\,\epsilon$ of $\mathcal D_k$,
and a translation by $(2\pi a,2\pi b)$ carries $S^k_\pm$ to $S^{k+b-a}_\pm$. Reparametrizing,
$S^0_-=\{(t+\frac\pi2-\epsilon\sin t,\ t+\frac\pi2+\epsilon\sin t)\}$. The slope of $S^k_+$ at the point with parameter $t$ is $(1-\epsilon\cos t)/(1+\epsilon\cos t)$, and that of $S^k_-$ is its reciprocal. Hence on the front
strip $0<\gamma<\pi$ the lift $S^k_+$ has slope less than $1$, and $S^k_-$ has slope greater than $1$ for
$\epsilon<\gamma<\pi-\epsilon$ (it has slope $1$ at $\gamma=\epsilon$ and $\gamma=\pi-\epsilon$); in $P$ these
are the two branches of $L_0$ along the interior of $\Delta$~\cite[\S5.1]{HHK2}. The two lifts near $\mathcal D_k$
cross exactly at the points $(\frac\pi2+j\pi,\frac\pi2+j\pi+2\pi k)$, and they pass each lattice point of $\mathcal D_k$
on opposite sides at distance $\sqrt2\,\epsilon$: for instance $S^0_+$ passes $(0,0)$ through $(-\epsilon,\epsilon)$
and $S^0_-$ through $(\epsilon,-\epsilon)$, and the sides alternate from one lattice point of $\mathcal D_k$ to the next.

\subsection{Restricted curves and the Floer complex}\label{ss:restricted}
Let $\ell_1$ be the line field of slope one on $\R^2$; it descends to $\Pst$. For a path $\alpha$ immersed in $\Pst$,
the Maslov index $\mu(\alpha,\ell_1)$ is the signed number of times the tangent line of $\alpha$ passes $\ell_1$,
counterclockwise passages counting $+1$, with the half-open endpoint convention of~\cite[\S2.3]{HHK2}. For lines
$\ell_0,\ell_1',\ell$, pairwise transverse, the \emph{triple index} $\tau(\ell_0,\ell_1',\ell)\in\{0,1\}$ is $1$
exactly when the clockwise rotation of $\ell_0$ to $\ell_1'$ passes through $\ell$~\cite[Def.~2.5]{HHK2}. Let $z\in
H^1(\Pst;\Z/4)$ be the class taking the value $1$ on a small counterclockwise loop around a
corner~\cite[Def.~3.1]{HHK2}.

An immersed circle in $\Pst$ is \emph{unobstructed} if every lift of it to the universal cover of $\Pst$ is a
properly embedded line; by a lemma of Abouzaid~\cite[Lemma~2.2]{Abouzaid} this holds if and only if it is
homotopically essential and contains no \emph{fishtail}, that is, no subinterval whose endpoints map to the same
point and which maps to a nullhomotopic loop~\cite[Def.~2.1]{HHK2}. A \emph{proper immersion} of an interval is
the image of a smooth immersion $I\to\R^2$ sending the endpoints to lattice points and the interior into
$\R^2\setminus(\pi\Z)^2$; its slopes at the endpoints are its \emph{limiting slopes}~\cite[Def.~3.3]{HHK2}.

\begin{definition}[{\cite[Def.~3.6 and p.~33]{HHK2}}]\label{def:restricted}
A \emph{restricted immersed circle} is an unobstructed immersed circle $L\colon S^1\to\Pst$ with
$\mu(L,\ell_1)+z(L)\equiv0\pmod4$. A \emph{restricted immersed arc} is a proper immersion of an interval whose
restriction to the complement of small discs around the corners contains no fishtail. A \emph{restricted immersed
$1$-manifold} is a disjoint union of one restricted immersed arc and finitely many restricted immersed circles.
\end{definition}

The earring curve $L_0$ is unobstructed with $\mu(L_0,\ell_1)=0=z(L_0)$, so it is a restricted immersed
circle~\cite[pp.~16,~18]{HHK2}. The image of a straight segment joining two lattice points whose interior misses the
lattice is a restricted immersed arc~\cite[p.~18]{HHK2}.

\begin{definition}[{\cite[Def.~3.11]{HHK2}}]\label{def:admissible}
A pair $(L_0,L_1)$ of restricted immersed curves is \emph{admissible} if at least one of them is a circle, if $L_0$
and $L_1$ intersect transversely, and if, whenever loops $\alpha_0$ in the domain of $L_0$ and $\alpha_1$ in the
domain of $L_1$ have $L_0\circ\alpha_0$ freely homotopic to $L_1\circ\alpha_1$ in $\Pst$, both $\alpha_0$ and
$\alpha_1$ are nullhomotopic.
\end{definition}

For an admissible pair, $C(L_0,L_1)$ is the $\F_2$-vector space on $L_0\cap L_1$. A \emph{bigon} from $p$ to $p'$ is
an orientation-preserving immersion of a disc into $\Pst$ with two convex corners, at $p$ and $p'$, whose
boundary, traversed counterclockwise, runs along $L_0$ from $p$ to $p'$ and then along $L_1$ from $p'$ back to $p$;
equivalent bigons are identified~\cite[Def.~2.10, (5)]{HHK2}. The differential $\partial p=\sum_{p'}\#\mathcal
M(p,p')\,p'$ counts bigons from $p$ to $p'$ modulo $2$, and by~\cite[Cor.~3.13]{HHK2} each $\mathcal M(p,p')$ contains
at most one bigon. A bigon lifts to an immersed disc in $\R^2\setminus(\pi\Z)^2$, so its two corners lie on the
same component of $L_1$; there is no differential between generators on different
components~\cite[p.~37]{HHK2}. We write $\HF(L_0,L_1)$ for the homology.

The complex carries a relative $\Z/4$-grading $\gr(p,p')=\gr(p)-\gr(p')$, defined for generators on the same
component of $L_1$~\cite[Def.~3.7]{HHK2}. The differential lowers it by one, and a bigon from $p$ to $p'$ forces
$\gr(p,p')=1$~\cite[Prop.~3.8]{HHK2}. If $\alpha_0$ is a path in $L_0$ from $p$ to $p'$ and $\alpha_1$ a path in $L_1$
from $p'$ to $p$, then~\cite[Lemma~3.10]{HHK2}
\begin{equation}\label{eq:grading15}
\gr(p,p')=\mu(\alpha_0,\ell_1)+\mu(\alpha_1,\ell_1)+\tau(T_pL_0,T_pL_1,\ell_1)-\tau(T_{p'}L_0,T_{p'}L_1,\ell_1)
+z(\alpha_0*\alpha_1)\pmod4 .
\end{equation}
Two consequences are recorded in~\cite[\S\S5.1--5.2]{HHK2}.
\begin{enumerate}
\item If $L_1$ meets the interior of $\Delta$ transversely at $x$, then for small $\epsilon$ there are exactly two
generators near $x$: $x^+$ on the branch of $L_0$ of slope less than $1$ and $x^-$ on the branch of slope greater
than $1$. They satisfy $\gr(x^+,x^-)=1$.
\item If $p_+,p'_+$ lie on the branch of slope less than $1$ on the front face, $\alpha_0$ runs along that branch, and
$L_1$ is transverse to $\Delta$ at both points, then
\begin{equation}\label{eq:grading16}
\gr(p_+,p'_+)=\mu(\alpha_1,\ell_1)+z(\alpha_0*\alpha_1)\pmod4 ,
\end{equation}
and the same formula holds for two points on the branch of slope greater than $1$.
\end{enumerate}
In terms of lifts, $x^\pm$ are the intersections of a lift of $L_1$ with the lifts $S^k_\pm$ of $L_0$ near a point of
$\mathcal D_k$ on the front strip.

Suppose $L_1$ is a restricted immersed arc with an endpoint at the corner $(0,0)$ whose limiting slope is bounded
away from $1$. Then for small $\epsilon$ there is a unique intersection point $r_+$ of $L_0$ and $L_1$ converging
to that corner as $\epsilon\to0$~\cite[Lemma~5.1]{HHK2}. When $L_1$ comes from a knot $K$ as in \S\ref{ss:decomp},
the grading of the component containing the arc is made absolute by
\begin{equation}\label{eq:grading19}
\gr(r_+)=\sigma(K)\pmod4
\end{equation}
\cite[(19)]{HHK2}. Since $\sigma(K)$ is even, \eqref{eq:grading19} does not depend on the sign convention for the
signature. On circle components Hedden, Herald and Kirk fix the absolute grading by matching one generator with
the corresponding generator of Kronheimer and Mrowka's complex~\cite[\S6.1]{HHK2},~\cite[Prop.~4.4]{KM2}.

\begin{theorem}[{\cite[Thm.~5.4]{HHK2}}]\label{thm:hhk54}
Let $L_1\colon S^1\to\Pst$ be a restricted immersed circle that misses the left and right edges of the
pillowcase, and let $\tilde L_1\colon[0,2\pi]\to\R^2$ be a lift of $t\mapsto L_1(e^{it})$. Then its vertical degree
$d=\frac1{2\pi}|\theta(\tilde L_1(2\pi))-\theta(\tilde L_1(0))|$ is even, and $\HF(L_0,L_1)$ has rank $d/2$ in each of
the four gradings.
\end{theorem}

\begin{remark}\label{rem:invariance}
By~\cite[Thm.~4.1, Cor.~4.6]{HHK2}, $\HF(L_0,L_1)$ is independent of $\epsilon$ and, as a relatively graded group,
depends only on the free homotopy classes of the circle components of $L_1$ and the homotopy class rel endpoints
of its arc. The computations in Section~\ref{sec:sheared} are direct and do not use this invariance.
Section~\ref{sec:twisted} uses it only through Theorem~\ref{thm:hhk54}, for the circle components: in~\cite{HHK2}
that theorem is deduced from~\cite[Cor.~4.6]{HHK2}, which follows from~\cite[Thm.~4.1]{HHK2}.
\end{remark}

\subsection{Tangle decompositions and the conjecture of Hedden, Herald and Kirk}\label{ss:decomp}
Let $K\subset S^3$ be a knot, decomposed along a Conway sphere as $(S^3,K)=(Y,T)\cup(D,U)$, where $(D,U)$ is the
trivial $2$-tangle in a ball and the common boundary is identified with
$(S^2,\{a,b,c,d\})$~\cite[(17)]{HHK2}. Let $R(Y,T)$ be the space of conjugacy classes of traceless $\SU$
representations of $\pi_1(Y\setminus T)$, and $R_\pi(Y,T)$ its version perturbed by a holonomy perturbation $\pi$
supported in the interior of $Y$. Restriction to the boundary gives $L_1\colon R_\pi(Y,T)\to P$, and the earring
curve $L_0$ of \S\ref{ss:earring} comes from the trivial side. When $L_1$ is a restricted immersed $1$-manifold and
$(L_0,L_1)$ is admissible, Hedden, Herald and Kirk define
\[
\Hnat(Y,T,\pi)=\HF(L_0,L_1),
\]
which is the direct sum of the contributions of the components of
$R_\pi(Y,T)$~\cite[p.~34, \S6.1]{HHK2}. The two abelian traceless representations of $\pi_1(Y\setminus T)$ are the
endpoints of the arc component; the one that extends over the knot complement maps to the corner $(0,0)$, and its
perturbation gives the generator $r_+$~\cite[\S6.1, \S8]{HHK2}. The remaining generators come in pairs $x^\pm$, one
pair for each transverse intersection point of $L_1$ with the interior of $\Delta$. For a torus knot, with the
tangle side unperturbed and a small perturbation on the side of the earring, these intersection points correspond
to the irreducible traceless representations of the knot group, and there are $1+|\sigma(K)|$
generators~\cite[Thm.~11.1]{HHK1},~\cite[\S11.10]{HHK2}.

\begin{conjecture}[{\cite[Conj.~6.5]{HHK2}}]\label{conj:hhk65}
``Given a knot $(X,K)$ in a homology 3-sphere, there exists a 2-tangle decomposition as in Equation (17), such that
for suitably small generic perturbations $\pi$, $L_1:R_\pi(Y,T)\to P$ is a restricted immersed 1-manifold and
$H^\natural(Y,T,\pi)$ is isomorphic to the reduced instanton homology $I^\natural(X,K)$.''
\end{conjecture}

Their Equation (17) is the decomposition $(X,K)=(Y,T)\cup(D,U)$ along a Conway sphere, which for $X=S^3$ is the
decomposition $(S^3,K)=(Y,T)\cup(D,U)$ of this section.

The conjecture asserts the existence of one good decomposition, and the answer does depend on the decomposition:
different choices of the gluing data give different curves $L_1$~\cite[p.~55]{HHK2}, and with the earring on the
other tangle the pillowcase homology of $T(4,5)$ has the wrong rank~\cite[\S10.3]{CHKK}. Hedden, Herald and Kirk
verify the conjecture for all two-bridge knots~\cite[\S7]{HHK2}. For the torus knots $T(3,n)$ with $n\ge4$ prime
to $3$, suitable decompositions from~\cite[\S\S10--11]{HHK2} satisfy it, with the $\Z/4$-gradings~\cite{WuebbenV}. Section~\ref{sec:sheared}
obtains the decompositions used for $P(-2,3,q)$ from these by regluing.

\subsection{The knots}\label{ss:knots}
We use the convention that positive knots have negative signature, as do~\cite{DS19,DS24,WuebbenV} and the tables
of~\cite{HHK2}; for instance $\sigma(T(3,5))=-8$~\cite[p.~60]{HHK2}. Braid generators are written $\sigma_1,\sigma_2$
and are positive crossings, so that the closure of $(\sigma_1\sigma_2)^n$ is the positive torus knot $T(3,n)$.

For odd $q\ge1$, $K_q=P(-2,3,q)$ is the pretzel knot with three twist regions of $-2$, $3$ and $q$ half-twists, with
the chirality for which its first three members are the positive torus knots $P(-2,3,1)=T(2,5)$,
$P(-2,3,3)=T(3,4)$ and $P(-2,3,5)=T(3,5)$; for $q\ge7$ it is hyperbolic (Proposition~\ref{prop:hyperbolic}). Section~\ref{sec:sheared}
decomposes it by regluing Hedden, Herald and Kirk's decomposition of $T(3,5)$ (Theorem~\ref{thm:topology}).

For $n\ge7$ with $n\equiv1,2\pmod 6$ and $m\ge1$, the twisted torus knot $T(3,n;2,m)$ is the closure of the
three-strand braid
\[
(\sigma_1\sigma_2)^n\,\sigma_1^{2m},
\]
that is, $T(3,n)$ with $m$ positive full twists added on two adjacent strands. It is a positive braid closure.
Section~\ref{sec:twisted} treats these knots with decompositions obtained from those of $T(3,n)$ in the same way.

\section{The instanton side}\label{sec:inat}

Throughout this section $K_q=P(-2,3,q)$ with $q\ge3$ odd. For a knot $K$ let
$\ell(K)=\sum_i|a_i(\Delta_K)|$ be the coefficient norm of its Alexander polynomial. Kronheimer and
Mrowka's integral spectral sequence from reduced Khovanov homology, their rational identification of
$\Inat$ with $KHI$, and the Alexander-graded Euler characteristic of $KHI$ give
\begin{equation}\label{eq:ell}
   \ell(K)\;\le\;\dim_{\Q}\Inat(K;\Q)\;\le\;\dim_{\Q}\widetilde{\Kh}(K;\Q)
\end{equation}
\cite[Props.~1.2 and 1.4]{KM2}, \cite[Thm.~1.1 and Cor.~1.2]{KMAlex}. We prove
Theorem~\ref{thm:inat} by showing that both bounds equal $q+2$ and that the spectral sequence
collapses over $\Z$. We then record the signature of $K_q$ and the $\Z/4$-graded dimensions of
$\Inat(K_q)$ that Section~\ref{sec:sheared} compares with the pillowcase.

\subsection{The Alexander polynomial}\label{ss:alex}
For $m\ge1$ write $[m]_x=1+x+\cdots+x^{m-1}$.

\begin{lemma}\label{lem:alex}
For every odd $q\ge3$ put
\[
 Q_q(x)=(x-2)[2]_x[3]_x[q]_x+[2]_x[3]_x+[2]_x[q]_x+[3]_x[q]_x .
\]
Then
\begin{equation}\label{eq:alex}
 \Delta_{K_q}(t)\doteq Q_q(-t)=1-t+\sum_{j=3}^{q}(-1)^{j+1}t^j-t^{q+2}+t^{q+3},
\end{equation}
where $\doteq$ denotes equality up to a unit $\pm t^k$. In particular $\ell(K_q)=q+2$ and
$\det K_q=|\Delta_{K_q}(-1)|=|q-6|$.
\end{lemma}

\begin{proof}
Hironaka proves that the Alexander polynomial of the $(2,3,q,-1)$-pretzel knot is $Q_q(-t)$, with
$Q_q$ the first displayed expression \cite[Thm.~1.2]{Hironaka}; the erratum \cite{HironakaErr}
corrects the statement for pretzel links with more than three positive entries and leaves the
three-entry case used here unchanged. The $(2,3,q,-1)$-pretzel knot is $K_q$, because the rational
tangles of fractions $1/2$ and $-1$ sum to the rational tangle of fraction $-1/2$; Hironaka records
this identification for $q=7$ \cite[p.~442]{Hironaka}. Expanding the geometric sums gives
$Q_q(x)=1+x-\sum_{j=3}^{q}x^j+x^{q+2}+x^{q+3}$, and substituting $x=-t$ gives the right-hand side of
\eqref{eq:alex}, since $q$ is odd. It has four coefficients of absolute value $1$ outside the sum
and $q-2$ inside, so $\ell(K_q)=q+2$. At $t=-1$ the terms outside the sum contribute $2+2$ and each of
the $q-2$ terms inside contributes $-1$, so $\Delta_{K_q}(-1)=6-q$.
\end{proof}

\begin{remark}\label{rem:lehmer}
At $q=7$, $Q_7(x)=1+x-x^3-x^4-x^5-x^6-x^7+x^9+x^{10}$ is Lehmer's polynomial. The substitution
$t\mapsto-t$ in \eqref{eq:alex} is part of Hironaka's formula, not a normalization ambiguity of the
Alexander polynomial. The same polynomial is obtained from the reduced Burau representation of the
braid $(\sigma_1\sigma_2)^5\sigma_1^{q-5}$, whose closure is $K_q$ for $q\ge5$ (Section~\ref{sec:sheared}).
\end{remark}

The following facts are used in Section~\ref{sec:sheared} to show that $K_q$ is hyperbolic and not
two-bridge.

\begin{lemma}\label{lem:not-torus}
For every odd $q\ge3$ the knot $K_q$ is neither quasi-alternating nor alternating, and hence it is not
two-bridge. For $q\ge7$ it is not a torus knot.
\end{lemma}

\begin{proof}
For a quasi-alternating knot $K$ the rank of $\Inat(K)$ equals $|\det K|$ \cite[Cor.~1.6]{KM2}. By
\eqref{eq:ell} and Lemma~\ref{lem:alex} the rank of $\Inat(K_q)$ is at least $q+2$, whereas
$|\det K_q|=|q-6|<q+2$. For an alternating knot the coefficients of the Alexander polynomial
alternate in sign \cite{Crowell,MurasugiGenus}, so that $|\Delta_K(-1)|=\ell(K)$; for $K_q$ the two
sides are $|q-6|$ and $q+2$. Two-bridge knots are alternating.

For the last assertion, recall that for coprime $1<a<b$ the Alexander polynomial
$\Delta_{T(a,b)}=(1-t)(1-t^{ab})/((1-t^a)(1-t^b))$ begins $1-t+t^a-\cdots$ and has degree $(a-1)(b-1)$;
it is unchanged under mirroring. In \eqref{eq:alex} the coefficient of $t^2$ is $0$ and that of $t^3$
is $1$, so a torus knot with the same Alexander polynomial would have $a=3$, and comparing degrees
gives $2(b-1)=q+3$, so $b=(q+5)/2\ge6$. For $b\ge6$ the coefficient of $t^5$ in
$\Delta_{T(3,b)}=(1-t)(1+t^3+t^6+\cdots)(1+t^b+\cdots)(1-t^{3b})$ is $0$, while the coefficient of
$t^5$ in \eqref{eq:alex} is $1$.
\end{proof}

\subsection{Reduced Khovanov homology and the integral spectral sequence}\label{ss:kh-ss}
Manion computes the reduced Khovanov homology over $\Z$ of the pretzel knots $P(-p,r,s)$ with
$p,r,s>0$ and $2\le p\le r,s$ that are not quasi-alternating, and proves that it is free
\cite[Thm.~1.1 and Thm.~3.3]{Manion}. For $p$ even it is $\Z^{p^2}\oplus\Z^{(r-p)(s-p)}$, supported
in two adjacent $\delta$-gradings. By Lemma~\ref{lem:not-torus}, $K_q$ satisfies these hypotheses for
every odd $q\ge3$, and substituting $(p,r,s)=(2,3,q)$ gives
\begin{equation}\label{eq:kh-rank}
 \widetilde{\Kh}(K_q;\Z)\ \text{is free of rank}\ 4+(3-2)(q-2)=q+2 .
\end{equation}
The reduced complex of the mirror is the integral dual complex, so by the universal
coefficient theorem the reduced Khovanov homology of the mirror is also free of rank $q+2$.

\begin{proof}[Proof of Theorem~\ref{thm:inat}]
Kronheimer and Mrowka construct a spectral sequence of abelian groups whose $E_2$ page is the reduced
Khovanov homology of the mirror and whose abutment is $\Inat(K_q;\Z)$ \cite[Prop.~1.2 and
Thm.~8.12]{KM2}; the underlying filtration, indexed by the vertices of a cube of resolutions, has
finite length. After tensoring with $\Q$, \eqref{eq:kh-rank} gives $\dim_{\Q}\Inat(K_q;\Q)\le q+2$.
Kronheimer and Mrowka also identify $\Inat(K_q;\Q)$ with $KHI(K_q;\Q)$ \cite[Prop.~1.4]{KM2}; the
Alexander-graded Euler characteristic of $KHI$ is the Alexander polynomial up to sign, so the dimension
of $KHI$ is at least $\ell(K_q)$ \cite[Thm.~1.1 and Cor.~1.2]{KMAlex}. The identification
of~\cite[Prop.~1.4]{KM2} is stated over $\Q$; the statements of~\cite{KMAlex} are made over $\C$ and pass to $\Q$
by flat base change. By Lemma~\ref{lem:alex},
$\dim_{\Q}\Inat(K_q;\Q)\ge\ell(K_q)=q+2$, so the two rational bounds agree. Total dimension cannot
increase from one page to the next in a spectral sequence of finite-dimensional vector spaces, so
every differential vanishes after tensoring with $\Q$.

The $E_2$ page of the integral spectral sequence is free of rank $q+2$. Inductively each page is free,
and a homomorphism between free abelian groups whose rationalization vanishes is zero; hence every
integral differential vanishes. The associated graded group of $\Inat(K_q;\Z)$ is therefore free of
rank $q+2$, and its finite filtration splits because extensions of free abelian groups split. This
proves $\Inat(K_q;\Z)\cong\Z^{q+2}$, and the assertions over $\Q$ and $\F_2$ follow by base change and
the universal coefficient theorem.
\end{proof}

\begin{remark}\label{rem:inat-attribution}
The argument with the two-sided rank bound~\eqref{eq:ell} over $\Q$ is Lobb and Zentner's proof that the
Kronheimer--Mrowka spectral sequence of $K_q$ is trivial \cite[Prop.~4.7]{LZ}; their $\Inat$ is Kronheimer and Mrowka's $\Inat$ of the mirror,
and the rank is unchanged by mirroring. Smith uses the same argument for $q=5$ \cite[proof of
Thm.~5.9]{Smith}. The rational rank also follows from the work of Li and Ye: $K_q$ is an instanton
L-space knot \cite[Prop.~7.1(2)]{LiYe}, so each Alexander-graded summand of $KHI(K_q)$ has dimension at
most one \cite[Thm.~1.9]{LiYe}, and the dimension of $KHI$ is then the number of nonzero coefficients
of the Alexander polynomial \cite[Thm.~1.1]{KMAlex}. The freeness in Theorem~\ref{thm:inat}
additionally uses Manion's integral calculation. For $q\ge7$ it also follows from the work of Daemi
and Scaduto (Proposition~\ref{prop:inat-graded} below). For the torus knots $K_3=T(3,4)$ and
$K_5=T(3,5)$ it can be read off from the integral S-complexes that Daemi and Scaduto write down
\cite[\S\S9.3--9.4]{DS19}: there the differential has rank one, because the rank of $\Inat$ is known,
and one of its components counts a single instanton, so the homology is torsion free. Finally,
$\det K_q=|q-6|$ never vanishes and equals $1$ exactly for $q=5,7$, the two members whose double
branched cover is an integral homology sphere.
\end{remark}

\subsection{The signature}\label{ss:sig}
We use the convention that positive knots have negative signature, so that $\sigma(T(3,4))=-6$ and
$\sigma(T(3,5))=-8$; this is the convention of \cite{HHK2} and \cite{WuebbenV}.

\begin{proposition}\label{prop:sig}
For odd $q\ge1$,
\[
 \sigma(K_q)=\begin{cases}-(q+3), & q\in\{1,3,5\},\\ -(q+1), & q\ge7.\end{cases}
\]
\end{proposition}

\begin{proof}
This is \cite[Lemma~8.29]{DS24}, stated there in the same convention; Daemi and Scaduto argue by
induction from $K_1=T(2,5)$, using the congruence $|\sigma(K)|+1\equiv\det K\pmod 4$ and the change of
the signature under a crossing change between links of nonzero determinant. An independent check is
the formula of Gordon and Litherland \cite[Thm.~6]{GordonLitherland}, applied to the checkerboard surface of the
standard diagram: two discs joined by three twisted bands with $-2$, $3$ and $q$ half-twists. With
respect to the two white regions between the bands, its Goeritz form is
$\eta\left(\begin{smallmatrix}1&-3\\-3&q+3\end{smallmatrix}\right)$ for a sign $\eta$ fixed by the
choice of coloring, with determinant $q-6$. The strands of $K_q$ run antiparallel through the band
with $-2$ half-twists and parallel through the other two, so the Gordon--Litherland correction term
is $\mu=\eta(q+3)$, one contribution from each of the $3+q$ crossings of the two odd bands. The
form has signature $2\eta$ for $q\ge7$ and $0$ for $q\le5$, so $\sigma(K_q)=-\eta(q+1)$ for
$q\ge7$ and $-\eta(q+3)$ for $q\le5$. The value $\sigma(K_5)=\sigma(T(3,5))=-8$ fixes
$\eta=1$.
\end{proof}

\subsection{The \texorpdfstring{$\Z/4$}{Z/4}-graded instanton homology}\label{ss:graded}
Daemi and Scaduto associate to a knot $K$ an S-complex
\begin{equation}\label{eq:scomplex}
 \widetilde C(K)=C(K)\oplus C(K)[-1]\oplus\Z_{(0)},\qquad C[i]_j=C_{i+j},
\end{equation}
a decomposition of $\Z/4$-graded modules in which $C(K)$ is generated by the irreducible critical points
of the singular Chern--Simons functional and the summand $\Z_{(0)}$, in degree $0$, by the reducible one.
The differential has the shape
\[
 \widetilde d=\begin{pmatrix}d&0&0\\ v&-d&\delta_2\\ \delta_1&0&0\end{pmatrix},
\]
where $\delta_1$ and $\delta_2$ count instantons to and from the reducible
\cite[Def.~2.1 and (1.9)]{DS24}, \cite[Def.~3.21]{DS19}. The S-complex structure is the map
$\chi\colon\widetilde C(K)\to\widetilde C(K)$ of degree $+1$ given in the decomposition~\eqref{eq:scomplex} by the matrix $\left(\begin{smallmatrix}0&0&0\\ 1&0&0\\ 0&0&0\end{smallmatrix}\right)$, which maps $C(K)$ identically onto $C(K)[-1]$; it satisfies
$\chi^2=0$ and $\chi\widetilde d+\widetilde d\chi=0$~\cite[\S4.1]{DS19}. The homology of $(\widetilde C(K),\widetilde d)$ is
$\Inat(K)$: there is a chain
homotopy equivalence between Kronheimer and Mrowka's complex and $\widetilde C(K)$, homogeneous for the
$\Z/4$-gradings, and for knots in $S^3$ the resulting isomorphism on homology has degree $\sigma(K)$
modulo $4$ \cite[Thm.~8.9]{DS19}. We write $I(K)$ for the homology of $(C(K),d)$, the irreducible
instanton homology.

\begin{proposition}\label{prop:inat-graded}
Let $q\ge7$ be odd and put $N_1=\lceil\frac{q+1}4\rceil$ and $N_3=\lfloor\frac{q+1}4\rfloor$. Then
$\Inat(K_q;\Z)$ is free of rank $q+2$, and in Kronheimer and Mrowka's absolute grading its graded
ranks are
\begin{equation}\label{eq:graded-vector}
 \bigl(1+N_3,\ N_1,\ N_1,\ N_3\bigr)\qquad\text{in degrees}\ \gr-\sigma(K_q)=0,1,2,3 .
\end{equation}
The same holds with coefficients in $\Q$ and in $\F_2$.
\end{proposition}

\begin{proof}
For $q\ge7$, Daemi and Scaduto prove $I(K_q)\cong\Z_{(1)}^{N_1}\oplus\Z_{(3)}^{N_3}$, with the subscripts
denoting $\Z/4$ degrees in the grading of \eqref{eq:scomplex} \cite[Prop.~8.30]{DS24}; the degrees come
from \cite[Prop.~8.28]{DS24} and the freeness from \cite[Rem.~8.27]{DS24}. The filtration
$C[-1]\subset C[-1]\oplus\Z_{(0)}\subset\widetilde C$ induces a spectral sequence
\begin{equation}\label{eq:ds-ss}
 I(K_q)\oplus I(K_q)[-1]\oplus\Z_{(0)}\ \Longrightarrow\ \Inat(K_q;\Z)
\end{equation}
\cite[(8.5)]{DS24}. Daemi and Scaduto state the homology groups of their \S8 as $\Z/2$-graded, but the
spectral sequence respects the $\Z/4$-grading: by the shape of $\widetilde d$, the summands $C[-1]$ and
$C[-1]\oplus\Z_{(0)}$ are subcomplexes, and \eqref{eq:scomplex} is a decomposition of $\Z/4$-graded
modules, so \eqref{eq:ds-ss} is a spectral sequence of $\Z/4$-graded groups. Since a class of degree $j$
in $I$ sits in degree $j+1$ in $I[-1]$, the initial page has ranks $1+N_3$, $N_1$, $N_1$ and $N_3$ in
degrees $0$, $1$, $2$ and $3$, for a total of $1+2(N_1+N_3)=q+2$.

By Theorem~\ref{thm:inat}, the rank of the abutment is also $q+2$. Hence every differential vanishes
after tensoring with $\Q$, and, as in the proof of Theorem~\ref{thm:inat}, every integral differential
vanishes and the filtration of $\Inat(K_q;\Z)$ splits. So $\Inat(K_q;\Z)$ is free with the graded ranks
of the initial page in the grading of \eqref{eq:scomplex}. By \cite[Thm.~8.9]{DS19}, Kronheimer and
Mrowka's absolute grading is this grading shifted by $\sigma(K_q)$; since $\sigma(K_q)$ is even, the sign
of the shift does not matter modulo $4$. The assertions over $\Q$ and $\F_2$ follow by base change.
\end{proof}

For $q=7$, $9$, $11$ and $13$ the vector \eqref{eq:graded-vector} is $(3,2,2,2)$, $(3,3,3,2)$, $(4,3,3,3)$
and $(4,4,4,3)$. For $q=3,5$ the same spectral sequence has initial rank $q+4$ and a differential of rank
one. The graded ranks of $\Inat$ of the torus knots $K_3=T(3,4)$ and $K_5=T(3,5)$ are $(1,1,2,1)$ and
$(2,1,2,2)$ in degrees $\gr-\sigma=0,1,2,3$ \cite[(2.5) and Table~2]{WuebbenV}, and these are the vectors
with which the pillowcase homology of the two torus knots is compared in \cite{WuebbenV}. They are instances of
the graded formula of~\cite{WuebbenIV}, which determines $\Inat(K;\Q)$ for every torus knot $K$, and the S-complexes of
the two knots are described in~\cite{WuebbenI}.

\section{Pillowcase homology of the pretzel knots}\label{sec:sheared}

In this section we prove Theorem~\ref{thm:hhk}. The decomposition is the one used by Hedden, Herald and Kirk
for $T(3,5)$~\cite[\S10]{HHK2}, reglued along the Conway sphere by $m=(q-5)/2$ Dehn twists. The proof has
four steps.
\begin{enumerate}
\item The reglued knot is $K_q=P(-2,3,q)$ (\S\ref{ss:decomp-P}). The argument uses only the fundamental group and the
signature and, apart from the base case (Remark~\ref{rem:chirality}), no planar diagram of the decomposition.
\item On the pillowcase the regluing replaces $L_1$ by its image under a linear shear (\S\ref{ss:shear}).
\item The unperturbed variety of the $T(3,5)$ tangle maps onto two straight polygons (\S\ref{ss:straight}). This rests
on a word in $\pi_1(Y\setminus T)$ that is trivial on the whole central circle of the variety.
\item The perturbed curves are those of~\cite{WuebbenV}, sheared (\S\ref{ss:perturbed-P}). From them we compute the
generators, the gradings and the bigons by hand (\S\S\ref{ss:gens-P}--\ref{ss:bigons-P}), for both ways in which the
perturbation can resolve the double points.
\end{enumerate}
No invariance theorem of~\cite{HHK2} is used. As in~\cite{WuebbenV}, results of~\cite{HHK2} are cited by their
numbering in arXiv:1501.00028v1.

Throughout, $\varepsilon=(\varepsilon_A,\varepsilon_B)$ denotes the parameters of the holonomy perturbation on the tangle side
((P1) in \S\ref{ss:straight}), and $\epsilon>0$ the earring parameter of \S\ref{ss:earring}. As in~\cite[\S6]{WuebbenV}, $\varepsilon$ is fixed first and $\epsilon$ is
then taken sufficiently small, and all statements about intersection points with $L_0$ hold for every
sufficiently small $\epsilon$.

\subsection{The decomposition}\label{ss:decomp-P}

\subsubsection*{The tangle}
Let $(Y,T)$ be the tangle of~\cite[\S10]{HHK2} for $T(3,5)$, with the integers $(r,s)=(2,-1)$ satisfying $3r+5s=1$;
this is the tangle of~\cite[\S3]{WuebbenV} for $n=5$. The group $\pi_1(Y\setminus T)$ is free on two elements $A,B$
(the elements $A_1,B_1$ of~\cite[\S10]{HHK2}, called $A,B$ in~\cite[\S11]{HHK1}). At $\varepsilon=0$ the meridians of the four boundary punctures are
\begin{equation}\label{eq:words-P}
a=A^2B^3,\qquad b=B^{-2}A^{-1},\qquad c=B^{-2}aB^2,
\end{equation}
by~\cite[(42)]{HHK1}, where Hedden, Herald and Kirk's torus-knot parameters are $3$ and $5$ and $(r,s)=(2,-1)$, so that the four exponents in their formula are $s+3=2$, $5-r=3$, $r=2$ and $s=-1$ (see~\cite[(3.2)]{WuebbenV}), and the fourth meridian $d$
satisfies $ba=cd$~\cite[\S3.3]{HHK1}. More generally, for $n=3\nu+2$ with $\nu\ge1$, the tangle
of~\cite[\S10]{HHK2} for $T(3,n)$ with $(r,s)=(\nu+1,-1)$ has
\begin{equation}\label{eq:words-n}
a=A^2B^{2\nu+1},\qquad b=B^{-(\nu+1)}A^{-1},\qquad c=B^{-(\nu+1)}aB^{\nu+1}.
\end{equation}

\subsubsection*{The trivial tangle and the twist}
Let $(D,U)$ be the trivial tangle of~\cite[(17)]{HHK2}: $D$ is the unit ball, and $U$ consists of the two vertical
arcs $\{(\pm1/\sqrt2,0,t):|t|\le1/\sqrt2\}$, one joining the punctures $a,d$ and the other joining $b,c$. The
group $\pi_1(D\setminus U)$ is free, and its boundary meridians satisfy $a=d$ and $b=c$~\cite[Prop.~6.1]{HHK1},
\cite[p.~4]{FKP}. Let $e\subset\partial D$ be the equator $\{z=0\}$; it separates $\{a,b\}$ from $\{c,d\}$ and bounds
the horizontal disc $D_e=D\cap\{z=0\}$, which meets each arc of $U$ once. Let $t_e$ be the Dehn twist along $e$ that
acts on $\pi_1(S^2\setminus\{a,b,c,d\})=\langle a,b,c,d\mid ba=cd\rangle$, with base point on the side of $c,d$, by
\begin{equation}\label{eq:te}
t_e(a)=(ba)\,a\,(ba)^{-1},\qquad t_e(b)=(ba)\,b\,(ba)^{-1},\qquad t_e(c)=c,\qquad t_e(d)=d .
\end{equation}
It fixes $ba=cd$. Hedden, Herald and Kirk use the half twist $g$ in the hemisphere containing $c$ and $d$ that
exchanges $c$ and $d$, and note that $g^2$ is the Dehn twist about $e$~\cite[proof of Thm.~9.2, p.~49]{HHK2}. With
base point near $a$, $g_*$ fixes $a,b$ and sends $c\mapsto d$, $d\mapsto d^{-1}cd$. So $g^2_*$ conjugates $c$ and $d$ by
$(cd)^{-1}=(ba)^{-1}$ and fixes $a,b$, and composing $(g_*^2)^{-1}$ with conjugation by $(ba)^{-1}$ gives $t_e^{-1}$.
Thus $t_e$ and $g^2$ are the same mapping class.

Recovering a knot from $(Y,T)$ requires a choice of identification of $\partial D$ with $\partial Y$~\cite[after (17)]{HHK2}.
For $h\in\Z$ let
\[
K^{(h)}=(Y,T)\cup_{t_e^h}(D,U)
\]
be the knot obtained by precomposing the identification of~\cite[(17)]{HHK2}, on the side of $Y$, with $t_e^h$: a
boundary loop $x$ of $D$ is identified with the loop $t_e^h(x)$ of $Y$.

\begin{lemma}\label{lem:group}
For every $h\in\Z$, $K^{(h)}$ is a knot in $S^3$ and
\[
\pi_1(S^3\setminus K^{(h)})\cong\langle A,B\mid t_e^h(b)=c\rangle,\qquad t_e^h(b)=(ba)^h\,b\,(ba)^{-h},
\]
with $a,b,c$ the words of the tangle. The knots $K^{(h)}$ and $K^{(h-1)}$ differ by one full twist of the two arcs
of $U$ about each other along $D_e$; in particular they are related by one crossing change.
\end{lemma}

\begin{proof}
Gluing two balls along their boundary gives $S^3$, and $t_e$ fixes each puncture, so the strands are connected as
in $K^{(0)}$, which is a knot. By van Kampen's theorem, $\pi_1(S^3\setminus K^{(h)})$ is the amalgam of
$\pi_1(Y\setminus T)=F(A,B)$ and $\pi_1(D\setminus U)$, free on $x=a=d$ and $y=b=c$, over the four-punctured sphere.
Eliminating $x$ and $y$ leaves the relations $t_e^h(b)=t_e^h(c)=c$ and $t_e^h(a)=t_e^h(d)=d$. The second follows from
the first, because $t_e^h$ fixes $ba=cd$: $t_e^h(a)=t_e^h(b)^{-1}ba=c^{-1}ba=d$. Regluing by $t_e^{-1}$ amounts to replacing
$U$ by its image under a twist map supported in a collar of $\partial D$, which is $U$ with one full twist of its two
arcs along $D_e$. Changing one of the two crossings of a full twist of two strands undoes it.
\end{proof}

At $h=0$ the relation $b=c$ reads $B^{-2}A^{-1}=B^{-2}A^2B^5$, that is $A^3B^5=1$, the group of $T(3,5)$, as it must be.

Clay and Watson define the twisted torus knots $T^m_{3,n}$ by adding $m$ positive full twists on two of the three
strands of $T(3,n)$~\cite[\S3.2]{ClayWatson}; $T^m_{3,n}$ is the closure of the braid $(\sigma_1\sigma_2)^n\sigma_1^{2m}$
(\S\ref{ss:knots}). For $n=3\nu+2$ they give the presentation~\cite[Prop.~22]{ClayWatson}
\begin{equation}\label{eq:CW}
\pi_1(S^3\setminus T^m_{3,n})\cong\bigl\langle\alpha,\beta\ \big|\ \alpha^2(\beta^{-\nu}\alpha)^m\alpha=\beta^{2\nu+1}(\beta^{-\nu}\alpha)^m\beta^{\nu+1}\bigr\rangle,
\qquad m\ge0 .
\end{equation}

\begin{lemma}\label{lem:relator}
Let $n=3\nu+2$ with $\nu\ge1$, let $(Y,T)$ be the tangle with the words~\eqref{eq:words-n}, and put $C=AB^\nu$. For every
$m\ge0$ the relation $t_e^{-m}(b)=c$ is equivalent to the vanishing of the cyclic word
\[
R_{n,m}=B^{\nu+1}A\,C^m\,A^2B^{2\nu+1}\,C^{-m},
\]
and under $\alpha=A$, $\beta=B^{-1}$ the relator of~\eqref{eq:CW} is a cyclic permutation of $R_{n,m}$. In particular
$\pi_1(S^3\setminus K^{(-m)})\cong\pi_1(S^3\setminus T^m_{3,n})$.
\end{lemma}

\begin{proof}
Since $ba=B^{-(\nu+1)}AB^{2\nu+1}=B^{-(\nu+1)}C\,B^{\nu+1}$, we have $(ba)^j=B^{-(\nu+1)}C^jB^{\nu+1}$. Hence
$t_e^{-m}(b)=B^{-(\nu+1)}C^{-m}A^{-1}B^{-(\nu+1)}C^mB^{\nu+1}$, and $t_e^{-m}(b)=c$ becomes
$C^{-m}A^{-1}B^{-(\nu+1)}C^m=A^2B^{2\nu+1}$. The inverse of the corresponding relator is
$A^2B^{2\nu+1}C^{-m}B^{\nu+1}AC^m$, a cyclic permutation of $R_{n,m}$. On the other side, substitute $\alpha=A$ and
$\beta=B^{-1}$ in the relator of~\eqref{eq:CW}:
\[
\alpha^2(\beta^{-\nu}\alpha)^m\alpha\,\beta^{-(\nu+1)}(\beta^{-\nu}\alpha)^{-m}\beta^{-(2\nu+1)}\ \longmapsto\ A^2(B^\nu A)^mA\,B^{\nu+1}(B^\nu A)^{-m}B^{2\nu+1}.
\] Using $A(B^\nu A)^m=C^mA$ and $(B^\nu A)^{-m}=B^\nu C^{-m}B^{-\nu}$, this is
$AC^mA^2B^{2\nu+1}C^{-m}B^{\nu+1}$, again a cyclic permutation of $R_{n,m}$.
\end{proof}

The group of a knot does not distinguish it from its mirror image, so the chirality of the base case enters as a
premise.

\begin{remark}[Base case]\label{rem:chirality}
The decomposition~\cite[(17)]{HHK2} gives the positive torus knot $T(3,5)$, of signature $-8$. Hedden, Herald and Kirk
call $K^{(0)}$ the $(3,5)$ torus knot without fixing its handedness~\cite[\S10]{HHK2}, and their comparison
in~\cite[\S11]{HHK2} uses $\sigma$ only modulo $4$, so it does not decide the question; the knot group cannot detect it
either. We verify it by a computer-assisted argument. From a planar diagram of~\cite[Fig.~18]{HHK2}, SnapPy~\cite{SnapPy},
with verified computation of canonical cell decompositions by interval arithmetic~\cite{HIKMOT,DHL}, shows that the
complement of the link $K^{(0)}\cup e$ is isometric to the complement of the closure of the positive braid
$(\sigma_1\sigma_2)^5$ together with an unknotted circle encircling two of its strands, and that every such isometry
preserves the orientation and the meridians (Appendix~\ref{app:data}). Such an isometry extends to an
orientation-preserving homeomorphism of $S^3$ carrying $K^{(0)}$ to the closure of $(\sigma_1\sigma_2)^5$, the positive torus
knot. The identification of the knots below and the graded comparison in Theorem~\ref{thm:hhk} rest on this.
\end{remark}

\begin{theorem}\label{thm:topology}
For every $m\ge1$, the knot $K^{(-m)}$ is the pretzel knot $K_{5+2m}=P(-2,3,5+2m)$, the closure of the positive braid
$(\sigma_1\sigma_2)^5\sigma_1^{2m}$, and $\sigma(K^{(-m)})=-(2m+6)$.
\end{theorem}

\begin{proof}
By Lemma~\ref{lem:relator} with $\nu=1$, the groups of $K^{(-m)}$ and of $T^m_{3,5}$ are isomorphic. Clay and Watson state
that $T^m_{3,5}$ is $P(-2,3,5+2m)$~\cite[\S3.2]{ClayWatson}; for $m=1$ this is also stated in~\cite[p.~164]{DS24}. We use
this only up to mirror image. In particular both knots are nontrivial, since the Alexander polynomial of their common
group is that of $K_{5+2m}$, which is not constant (Lemma~\ref{lem:alex}). Both groups are generated by two elements, so
both knots are prime by Norwood's theorem~\cite{Norwood}. Prime knots with isomorphic groups are equivalent, by a
homeomorphism of $S^3$ that may reverse orientation~\cite[Cor.~2.1]{GordonLuecke}. Hence $K^{(-m)}$ is $K_{5+2m}$ or its
mirror image.

To fix the chirality, recall $\sigma(K_q)=-(q+1)$ for $q\ge7$ and $\sigma(K_5)=\sigma(T(3,5))=-8$
(Proposition~\ref{prop:sig}); the mirror image has the opposite signature. By Lemma~\ref{lem:group}, $K^{(-m-1)}$ and
$K^{(-m)}$ are related by one crossing change, so their signatures differ by at most $2$~\cite[proof of Thm.~10.1, (10.2)]{Murasugi}.
Suppose $\sigma(K^{(-m)})=-s_m$ with $s_0=8$ and $s_m=2m+6$ for $m\ge1$; for $m=0$ this is Remark~\ref{rem:chirality}.
The two candidates for $K^{(-m-1)}$ have signature $\pm(2m+8)$, and only $-(2m+8)$ lies within $2$ of $-s_m$. Hence
$K^{(-m-1)}=K_{2m+7}$, and induction on $m$ shows $K^{(-m)}=K_{5+2m}$ for every $m\ge1$.

It remains to see that $K_{5+2m}$, rather than its mirror image, is the closure of $(\sigma_1\sigma_2)^5\sigma_1^{2m}$. By
Clay and Watson, used up to mirror image, this closure is $K_{5+2m}$ or its mirror image. Changing one crossing of
$\sigma_1^{2m}$ from positive to negative gives $(\sigma_1\sigma_2)^5\sigma_1^{2m-2}$, and for knots $K_\pm$
related by such a change $\sigma(K_+)-\sigma(K_-)\in\{0,-2\}$ in our convention (see the proof of
Proposition~\ref{prop:twisted-sig}). By induction from $\sigma(T(3,5))=-8$, the closure has negative signature. Since
$\sigma(K_{5+2m})=-(2m+6)<0$ and its mirror image has positive signature, the closure is $K_{5+2m}$.
\end{proof}

The same argument applies to every $n\equiv2\pmod3$ once the signature of $T(3,n;2,m)$ is known; this is used in
Section~\ref{sec:twisted}.

\begin{proposition}\label{prop:hyperbolic}
For odd $q\ge7$ the knot $K_q$ is hyperbolic and is not a two-bridge knot.
\end{proposition}

\begin{proof}
By Lemma~\ref{lem:not-torus}, $K_q$ is neither two-bridge nor a torus knot. It is the Montesinos knot with rational
tangles $-1/2$, $1/3$, $1/q$, all of denominator at least $2$. By Oertel's theorem~\cite[Cor.~5]{Oertel}, the
complement of such a Montesinos link carries a complete hyperbolic structure unless the link is a torus link or one
of the links $K(\tfrac12,\tfrac12,-\tfrac12,-\tfrac12)$, $K(\tfrac23,-\tfrac13,-\tfrac13)$, $K(\tfrac12,-\tfrac14,-\tfrac14)$,
$K(\tfrac12,-\tfrac13,-\tfrac16)$ or a mirror image of one of these. For each of the four listed links the sum of the
fractions is $0$. The determinant of a Montesinos link with fractions $\beta_i/\alpha_i$ is
$|\alpha_1\cdots\alpha_r\sum_i\beta_i/\alpha_i|$, the order of the first homology of its double branched cover (infinite,
and the determinant $0$, when the sum vanishes), whereas the determinant of a knot is odd; so none of them, and none
of their mirror images, is a knot. Conventions for the sign of a rational tangle differ between authors, which may replace $K_q$ by its
mirror image; the list of exceptions is closed under mirror image. Hence $K_q$ is hyperbolic.
\end{proof}

Oertel's corollary does not list the torus links, which is why Lemma~\ref{lem:not-torus} is needed. Explicitly,
apart from two-bridge knots the only non-hyperbolic Montesinos knots are $P(-2,3,3)=T(3,4)$ and $P(-2,3,5)=T(3,5)$
and their mirror images~\cite[Rem.~1]{IchiharaJong},~\cite[p.~440]{Meier}. Futer and Gu\'eritaud give a published
proof of the classification of non-hyperbolic arborescent links, which they attribute to Bonahon and
Siebenmann~\cite[Thm.~1.5]{FutGue}.

\subsection{The twist acts by a shear}\label{ss:shear}

For $h\in\Z$ let $S_h\colon\R^2\to\R^2$ be the linear map $S_h(\gamma,\theta)=(\gamma,\theta-2h\gamma)$.

\begin{lemma}\label{lem:shear}
Let $(Y,T)$ be a $2$-tangle as in~\cite[(17)]{HHK2}, with a holonomy perturbation supported in the interior of $Y$, and
let $h\in\Z$. For the decomposition $(Y,T)\cup_{t_e^h}(D,U)$:
\begin{enumerate}
\item[(a)] $R_\pi(Y,T)$ is unchanged, its restriction map is $S_h\circ L_1$, where $L_1$ is the restriction map for the
identification of~\cite[(17)]{HHK2}, and $L_0$ is unchanged;
\item[(b)] $S_h$ descends to a homeomorphism of $P$ fixing each corner and to an orientation-preserving
diffeomorphism of $\Pst$; it preserves $\gamma$, hence the open strip $\{0<\gamma<\pi\}$, the two edges, and the vertical
degree of closed curves that miss the edges;
\item[(c)] $S_h$ maps proper immersions of intervals to proper immersions, adding $-2h$ to both limiting slopes, and
preserves embeddedness of lifts and the absence of fishtails;
\item[(d)] for closed curves $S_h$ preserves $z$ and $\mu(\cdot,\ell_1)$; hence it maps restricted immersed
$1$-manifolds to restricted immersed $1$-manifolds;
\item[(e)] the corner images of the two abelian representations are unchanged;
\item[(f)] for every immersed path $\alpha$, $\mu(S_h\circ\alpha,\ell_1)=\mu(\alpha,S_h^{-1}\ell_1)$, where $S_h^{-1}\ell_1$ is the
constant line field spanned by $(1,1+2h)$.
\end{enumerate}
\end{lemma}

\begin{proof}
(a) The perturbation is supported in the interior of $Y$, so $R_\pi(Y,T)$ does not change. Let a representation
restrict to $a\mapsto\mathbf i$, $b\mapsto e^{\gamma\mathbf k}\mathbf i$, $c\mapsto e^{\theta\mathbf k}\mathbf i$ on $\partial Y$. Then
$\rho(ba)=-e^{\gamma\mathbf k}$, so by~\eqref{eq:te}, $\rho(t_e^ha)=e^{2h\gamma\mathbf k}\mathbf i$, $\rho(t_e^hb)=e^{(2h+1)\gamma\mathbf k}\mathbf i$,
and $c$ is fixed. Conjugating by $e^{-h\gamma\mathbf k}$ gives $(\mathbf i,\ e^{\gamma\mathbf k}\mathbf i,\ e^{(\theta-2h\gamma)\mathbf k}\mathbf i)$, the
point $S_h(\gamma,\theta)$, with $\gamma\in[0,\pi]$ unchanged. The trivial tangle keeps the identification
of~\cite[(17)]{HHK2}, so $L_0$ is unchanged.

(b) $S_h$ is linear of determinant $1$ and commutes with $-1$; it sends $(2\pi,0)$ to $(2\pi,-4\pi h)$ and fixes
$(0,2\pi)$, so it normalizes $G$. It maps $(\pi a,\pi b)$ to $(\pi a,\pi(b-2ha))\equiv(\pi a,\pi b)$ modulo $2\pi\Z^2$.

(c) Embeddedness of lifts and the absence of fishtails are invariant under the homeomorphism $S_h$ of $\Pst$.

(d) $S_h$ fixes the corners and preserves orientation, so it preserves $z$. For a closed curve, $\mu(\cdot,\ell)$ is
the total rotation of the tangent line divided by $\pi$, for every constant line field $\ell$, because the deck
transformations have derivative $\pm I$. Along the path $s\mapsto S_{sh}$, $0\le s\le1$, this total rotation is a
continuous integer, hence constant. The rest follows from (b), (c) and Definition~\ref{def:restricted}.

(e) This follows from (b).

(f) The tangent line of $S_h\circ\alpha$ equals $\ell_1$ exactly when the tangent line of $\alpha$ equals $S_h^{-1}\ell_1$. Since
$S_h$ preserves orientation, the senses of passage agree, and the endpoint conventions correspond.
\end{proof}

The computation in (a) is that of~\cite[(28)]{HHK2}: there the half twist $g$ acts on the pillowcase by
$(\gamma,\theta)\mapsto(\gamma,\theta-\gamma)$, and~\cite[(29)--(30)]{HHK2} obtain the action of $\mathrm{PSL}_2(\Z)$ from $g$ and a
$3$-cycle. Since $t_e=g^2$, $t_e^h$ acts by $S_h$. Our additions are the observation that regluing by $t_e^h$ is a
choice of identification permitted by~\cite[(17)]{HHK2}, and properties (b)--(f). The action of the mapping class
group on the pillowcase is also~\cite[Thm.~B]{CHK}, and its counterpart for immersed curves in Khovanov
homology is~\cite[Thm.~8.1]{KWZ}.

\subsubsection*{The pretzel case}
From now on $q\ge7$ is odd and $m=(q-5)/2$, so that
\[
S_{-m}(\gamma,\theta)=\bigl(\gamma,\ \theta+(q-5)\gamma\bigr).
\]
By Theorem~\ref{thm:topology} and Lemma~\ref{lem:shear}, the decomposition $K_q=(Y,T)\cup_{t_e^{-m}}(D,U)$ has restriction
map $S_{-m}\circ L_1$, where $L_1$ is the restriction map of~\cite{WuebbenV} for $T(3,5)$. On $\R^2$ we keep
$\varphi=\theta-\gamma$; then
\[
\varphi\circ S_{-m}=\lambda,\qquad \lambda(\gamma,\theta)=\theta+(q-6)\gamma ,
\]
and $S_{-m}^{-1}\ell_1$ is the line field spanned by $(1,6-q)$.

\subsection{The unperturbed curve}\label{ss:straight}

We use the following facts about the tangle $(Y,T)$ for $n=5$ and its holonomy perturbations, in the notation
of~\cite[\S\S3--6]{WuebbenV} with these changes: the perturbation parameters, written $\epsilon=(\epsilon_A,\epsilon_B)$
there, are written $\varepsilon=(\varepsilon_A,\varepsilon_B)$ here; the binary dihedral arc $B$ of~\cite{WuebbenV}, its
pieces $B_0,B_1,B_2$ and the vectors $T_B$, $P_B$ are written $\Omega$, $\Omega_0,\Omega_1,\Omega_2$, $T_\Omega$, $P_\Omega$, to
distinguish them from the generator $B$; the lines $\Delta_k$ and the strands $S^k_\pm$ of~\cite{WuebbenV} are the lines
$\mathcal D_k$ and the lifts $S^k_\pm$ of $L_0$ of \S\ref{ss:earring}; and the reduced function $g$
of~\cite[\S5.2]{WuebbenV} is written $f_c$.
\begin{enumerate}
\item[(P1)] \emph{The perturbation}~\cite[\S3.1]{WuebbenV}. The group $\pi_1(Y\setminus T)$ is free on $A,B$, and the
holonomy perturbation $\pi$ is supported near two curves $P_A,P_B$ in $Y$, those of~\cite[\S10]{HHK2}, with
perturbation functions $\varepsilon_A\sin$ and $\varepsilon_B\sin$. A perturbed traceless representation is determined, up
to conjugation, by
\[
A\mapsto e^{uQ_A},\qquad B\mapsto e^{vQ_B},\qquad Q_A=\mathbf i,\qquad Q_B=\tau\,\mathbf i+\sqrt{1-\tau^2}\,\mathbf j,
\]
with $(u,v,\tau)\in[0,\pi]^2\times[-1,1]$.
\item[(P2)] \emph{The equations}~\cite[\S3.1 and Lemma~3.1]{WuebbenV},~\cite[(33)--(34)]{HHK2}. The meridians $a,b$ of
$T$ are sent to traceless elements exactly when $\Psi=(\Psi_1,\Psi_2)=0$, where
$\Psi_i=\cos\Theta_i\cos\Phi_i-\tau\sin\Theta_i\sin\Phi_i$ and, for $n=5$, the perturbed phases are
\[
\Phi_1=2u+3\varepsilon_A\sin u,\qquad \Theta_1=3v-2\varepsilon_B\sin v,\qquad \Phi_2=-u-2\varepsilon_A\sin u,\qquad
\Theta_2=-2v+\varepsilon_B\sin v .
\]
Put $W_\varepsilon=\Psi(\varepsilon;\cdot)^{-1}(0)\subset[0,\pi]^2\times[-1,1]$ and $W=W_0$. There is a surjection
$W_\varepsilon\to R_\pi(Y,T)$~\cite[Thm.~10.2]{HHK2}, whose fiber over the image of $(u,v,\tau)$ is that single point unless
$\sin u=0$ or $\sin v=0$, in which case it is the segment $\{(u,v)\}\times[-1,1]$. At $\varepsilon=0$,
$\Psi=(\mathrm{E2},\mathrm{E1})$, where (E1) and (E2) denote
\[
\cos2v\cos u-\tau\sin2v\sin u\qquad\text{and}\qquad\cos3v\cos2u-\tau\sin3v\sin2u .
\]
The restriction map $\Pi=\Pi_\varepsilon\colon W_\varepsilon\to P$ is normalized by $a\mapsto\mathbf i$,
$b\mapsto e^{\gamma\mathbf k}\mathbf i$, $c\mapsto e^{\theta\mathbf k}\mathbf i$ with $\gamma\in[0,\pi]$, and it depends smoothly on
$(\varepsilon,u,v,\tau)$ where it is defined~\cite[\S\S3.1, 5]{WuebbenV}.
\end{enumerate}
The unperturbed variety $W$ is described as follows.
\begin{enumerate}
\item[(V1)] $W=\Omega\cup C_0$, where $\Omega=\{(\frac\pi2,\frac\pi2)\}\times[-1,1]$ is the binary dihedral arc and $C_0$ is the central
circle; $C_0\cap \Omega=\{c_+,c_-\}$ with $c_\pm=(\frac\pi2,\frac\pi2,\pm\tau_1)$, $\tau_1=\sqrt3/2$; the halves
$C_0^{\mathrm L}=C_0\cap\{v\le\frac\pi2\}$ and $C_0^{\mathrm R}=C_0\cap\{v\ge\frac\pi2\}$ meet $\{v=\frac\pi2\}$ only at $c_\pm$
\cite[Prop.~3.6, Prop.~3.10(2)]{WuebbenV}.
\item[(V2)] $\Pi(\Omega)$ lifts to the segment $t\mapsto(t,0)$, $t\in[0,\pi]$, with $\gamma=\arccos\tau$ on $\Omega$. The endpoints
$r_\pm=(\frac\pi2,\frac\pi2,\pm1)$ map to $(0,0)$ and $(\pi,0)$, and $\Pi(c_+)=(\frac\pi6,0)$, $\Pi(c_-)=(\frac{5\pi}6,0)$
\cite[Prop.~3.9]{WuebbenV}. We write $\Omega_0$, $\Omega_1$, $\Omega_2$ for the pieces of $\Omega$ with $\tau\in[\tau_1,1]$,
$[-\tau_1,\tau_1]$, $[-1,-\tau_1]$.
\item[(V3)] Along each half of $C_0$, from $c_+$ to $c_-$, $\varphi$ is strictly monotone with $d\varphi\neq0$ at interior
points and increases by $10\pi/3$; at interior points $\frac\pi6<\gamma<\frac{5\pi}6$~\cite[Thm.~4.10,
Props.~4.11--4.12]{WuebbenV}.
\item[(V4)] $v$ is a regular parameter of $C_0$ at $c_\pm$, and on a lift
$\Pi=\Pi(c_\pm)\pm V_0(v-\frac\pi2)^2+O(|v-\frac\pi2|^3)$ with $V_0=(2/\sqrt3,\,12/\sqrt3)$, of slope
$6$~\cite[Lemma~4.13]{WuebbenV}.
\item[(V5)] $c_\pm$ are nondegenerate double points of $W$. With respect to the equations (E1), (E2) of (P2), the
kernel of $d\Psi(c_\pm)$ is spanned by $T_\Omega=\partial_\tau$ and $X=-2\tau\partial_u+\partial_v$, the
covector $\ell=(-2\tau,-1)$ spans the left kernel, $\ell\cdot D^2\Psi(T_\Omega,X)=-8\tau$ and $\ell\cdot D^2\Psi(T_\Omega,T_\Omega)=0$
\cite[Prop.~3.11(2)]{WuebbenV}.
\end{enumerate}

\begin{lemma}\label{lem:straight}
Let $\varepsilon=0$ and put $w=(ba)^{-3}c\,(ba)^3a\in\pi_1(Y\setminus T)$. Then $\rho(w)=1$ for every $\rho\in C_0$. Consequently
$\theta\equiv6\gamma-\pi\pmod{2\pi}$ on $\Pi(C_0)$, and for $l\in\Z$ the lift of each half of $C_0$ that starts at the point
$(\frac\pi6,2\pi l)$ over $\Pi(c_+)$ is the segment
\[
\{(\gamma,\,6\gamma-\pi+2\pi l):\ \tfrac\pi6\le\gamma\le\tfrac{5\pi}6\},
\]
traversed with $\gamma$ strictly increasing from $c_+$ to $c_-$ and with $d\gamma\neq0$ at interior points.
\end{lemma}

\begin{proof}
We write quaternions as $p=p_0+\mathbf p$ with $p_0=\operatorname{Re}p$, so that $\mathbf p\mathbf q=-\mathbf p\cdot\mathbf q+\mathbf p\times\mathbf q$ for
pure quaternions.

\emph{Step 1: the word.} By~\eqref{eq:words-P}, $ba=B^{-2}(AB)B^2$, so $(ba)^3=B^{-2}w_0$ with $w_0=(AB)^3B^2$. Since
$c=B^{-2}aB^2$, we get $w=w_0^{-1}B^2\cdot B^{-2}aB^2\cdot B^{-2}w_0\cdot a=w_0^{-1}aw_0a$. If $\rho(a)$ and $\rho(w_0)$ are
orthogonal pure unit quaternions, they anticommute, so $\rho(w_0)^{-1}\rho(a)\rho(w_0)=-\rho(a)$ and $\rho(w)=-\rho(a)^2=1$. It remains to
prove that orthogonality.

\emph{Step 2: notation.} Let $\rho\in W$. Then $\beta=\rho(b)$ is a pure unit quaternion. Write $N=\rho(B)=\cos v+\sin v\,Q$
with $Q=Q_B$, and put
\[
\zeta=\beta\cdot Q,\qquad U=\cos v\,\beta+\sin v\,(\beta\times Q),\qquad E=4\cos^2v+4\zeta^2\sin^2v-3 .
\]
From $b=B^{-2}A^{-1}$ we get $\rho(A)=-\beta N^{-2}$.

\emph{Step 3: $\rho(a)$.} Conjugation by $\beta$ is the rotation by $\pi$ about $\beta$, so $\beta Q\beta^{-1}=Q'$ with $Q'=2\zeta\beta-Q$,
and $\beta N^{-2}\beta=-e^{-2vQ'}$. Hence $\rho(a)=\rho(A)^2N^3=-e^{-2vQ'}N$. Using
$\operatorname{Re}(e^{sP}e^{tR})=\cos s\cos t-(P\cdot R)\sin s\sin t$ and $Q'\cdot Q=2\zeta^2-1$,
\[
\operatorname{Re}\rho(a)=-\cos v\cdot E,\qquad \mathbf{Vec}\,\rho(a)=-\sin3v\,Q+2\zeta\sin2v\,U .
\]

\emph{Step 4: $C_0\subset\{E=0\}$.} On $W$ the element $\rho(a)$ is traceless, so $\cos v\cdot E=0$. On $C_0\setminus\{c_\pm\}$
we have $v\neq\frac\pi2$ by (V1), so $E=0$ there, and by continuity on all of $C_0$.

\emph{Step 5: $\rho(w_0)$.} Put $\xi=\beta N^{-1}$, so that $\rho(AB)=-\xi$ and $\operatorname{Re}\xi=\zeta\sin v$. The Cayley--Hamilton
identity $\xi^2=2(\operatorname{Re}\xi)\xi-1$ gives $\xi^3=(4\zeta^2\sin^2v-1)\xi-2\zeta\sin v$, and $\xi N^2=\beta N$, so
\[
\rho(w_0)=-\xi^3N^2=-(4\zeta^2\sin^2v-1)\,\beta N+2\zeta\sin v\,N^2,\qquad \operatorname{Re}\rho(w_0)=\zeta\sin v\cdot E .
\]
On $\{E=0\}$ we have $4\zeta^2\sin^2v-1=-2\cos2v$, so $\rho(w_0)$ is pure, with $\mathbf{Vec}\,\rho(w_0)=2\cos2v\,U+2\zeta\sin v\sin2v\,Q$.

\emph{Step 6: orthogonality.} On $\{E=0\}$, with $\mathbf{Vec}\,\rho(w_0)$ in the form of Step~5, and with $U\cdot U=1-\zeta^2\sin^2v$, $U\cdot Q=\zeta\cos v$ and $\sin3v=\sin v\,(4\cos^2v-1)$, a direct
expansion gives
\[
\mathbf{Vec}\,\rho(w_0)\cdot\mathbf{Vec}\,\rho(a)=2\zeta\bigl(2\zeta^2\sin^2v\sin2v+2\cos2v\sin2v-\sin3v\cos v\bigr)=\zeta\sin2v\cdot E .
\]
So on $C_0$ the quaternions $\rho(a)$ and $\rho(w_0)$ are orthogonal pure units, and $\rho(w)=1$ by Step~1.

\emph{Step 7: the pillowcase.} In the normalization of $\Pi$, $\rho(ba)=-e^{\gamma\mathbf k}$ and $\mathbf ie^{t\mathbf k}=e^{-t\mathbf k}\mathbf i$, so
$\rho\bigl((ba)^jc(ba)^{-j}a\bigr)=-e^{(\theta+2j\gamma)\mathbf k}$. For $j=-3$, $\rho(w)=1$ gives $\theta-6\gamma\in\pi+2\pi\Z$.

\emph{Step 8: the lifts.} By (V3) the lift of a closed half starting at a lift $(\frac\pi6,2\pi l)$ of $\Pi(c_+)$ stays
in the open strip except at its ends. On it $\theta-6\gamma+\pi$ is continuous with values in $2\pi\Z$, hence equal to its
value $2\pi l$ at the start, so $\varphi=5\gamma-\pi+2\pi l$. By (V3), $\varphi$ is strictly monotone with $d\varphi\neq0$ at interior
points and increases by $10\pi/3$; hence $\gamma=(\varphi+\pi-2\pi l)/5$ increases strictly from $\frac\pi6$ to $\frac{5\pi}6$,
with $d\gamma\neq0$ at interior points.
\end{proof}

Each displayed identity of Steps 1--7 has also been verified symbolically for indeterminate $(v,\zeta)$ (see
Appendix~\ref{app:data}). Hedden, Herald and Kirk observed, by a computer algebra computation, that the two halves of
the central circle for $T(3,5)$ map linearly with slope $6$~\cite[\S11.7, p.~59]{HHK2}; Lemma~\ref{lem:straight} proves
this. For their own decomposition of the pretzel knots $P(p,q,r)$ with $p$ even, Fukumoto, Kirk and Pinz\'on-Caicedo
prove the analogous straightness~\cite[Thm.~5.1]{FKP}, and they describe the unperturbed curve as the segments
$\theta=(r-5)\gamma$ and $\theta=(r+1)\gamma+\pi$~\cite[\S5, p.~28]{FKP}. These are the images under $S_{-m}$ of $\Pi(\Omega)$ and of the line
of Lemma~\ref{lem:straight}, with $r=q$. We do not identify their decomposition with ours, and we do not use their
result; we regard it as the antecedent of Lemma~\ref{lem:straight}.

By (P6) in \S\ref{ss:perturbed-P}, after a small perturbation the arc of $L_1$ is built from $\Omega_0$, one half of
$C_0$ and $\Omega_2$, possibly through $\Omega_1$ and the other half. Write $H$ for the half joined to $\Omega_0$ and $H'$ for the
other half, and orient both from $c_+$ to $c_-$. Let $O=(0,0)$, and
\[
J_1=(\tfrac\pi6,0),\quad J_2=(\tfrac{5\pi}6,4\pi),\quad J_3=(\tfrac\pi6,4\pi),\quad J_4=(\tfrac{5\pi}6,8\pi),\quad Z=(\pi,8\pi),\quad Z'=(\pi,4\pi).
\]

\begin{corollary}\label{cor:polygon}
At $\varepsilon=0$ the lifted pieces are straight segments, traversed monotonically and regular in their interiors:
$\Omega_0\colon O\to J_1$ along $(1,0)$; $H\colon J_1\to J_2$ along $(1,6)$; $\Omega_1^{-1}\colon J_2\to J_3$ along $(-1,0)$;
$H'\colon J_3\to J_4$ along $(1,6)$; and $\Omega_2\colon J_4\to Z$, or $\Omega_2\colon J_2\to Z'$ when $\Omega_2$ follows $H$, along $(1,0)$.
The lift of the closed curve $\Omega_1^{-1}\cdot H'$ is the $(0,4\pi)$-periodic polygon with vertices $(\frac{5\pi}6,4\pi j)$ and
$(\frac\pi6,4\pi j)$, $j\in\Z$.
\end{corollary}

\begin{proof}
The pieces of $\Omega$ are described by (V2): $\gamma=\arccos\tau$ is a diffeomorphism from $(-1,1)$ onto $(0,\pi)$. The halves are described by Lemma~\ref{lem:straight}, applied to the lifts starting at $J_1=(\frac\pi6,0)$ and $J_3=(\frac\pi6,4\pi)$, whose endpoints are the listed vertices.
\end{proof}

\begin{figure}[ht]
\centering
\begin{tikzpicture}[x=2.4cm,y=0.42cm,>=stealth]
\begin{scope}
\draw[gray!60] (0,0) -- (0,8.6); \draw[gray!60] (1,0) -- (1,8.6);
\draw[->] (0,0) -- (1.15,0) node[right] {$\gamma$};
\draw[thick] (0,0) -- (1/6,0) -- (5/6,4) -- (1/6,4) -- (5/6,8) -- (1,8);
\draw[thick,dashed] (5/6,4) -- (1,4);
\foreach \p/\l/\a in {{(0,0)}/O/left,{(0.1667,0)}/J_1/below,{(0.8333,4)}/J_2/below,{(0.1667,4)}/J_3/left,{(0.8333,8)}/J_4/above,{(1,8)}/Z/right,{(1,4)}/Z'/right}
  \fill \p circle (1.2pt) node[\a] {\small$\l$};
\node at (0.5,-1.6) {\small(a) before the shear};
\end{scope}
\begin{scope}[xshift=4.2cm]
\draw[gray!60] (0,0) -- (0,14.6); \draw[gray!60] (1,0) -- (1,14.6);
\draw[->] (0,0) -- (1.15,0) node[right] {$\gamma$};
\foreach \k in {1,...,6} \draw[gray] (0,2*\k) -- (1,2*\k+1) node[right] {\tiny$\mathcal D_{\k}$};
\draw[thick] (0,0) -- (1/6,1) -- (5/6,9) -- (1/6,5) -- (5/6,13) -- (1,14);
\foreach \p/\l/\a in {{(0,0)}/O/left,{(0.1667,1)}/J_1/left,{(0.8333,9)}/J_2/above,{(0.1667,5)}/J_3/left,{(0.8333,13)}/J_4/left,{(1,14)}/Z/right}
  \fill \p circle (1.2pt) node[\a] {\small$\l$};
\node at (0.5,-1.6) {\small(b) after, $q=11$};
\end{scope}
\begin{scope}[xshift=8.6cm,yshift=2cm,x=1cm,y=1cm]
\draw[gray] (-0.8,0) -- (4.8,0) node[right] {\small$\Sigma$};
\draw[thick] plot[smooth,tension=0.7] coordinates {(4.6,-0.8) (4,0) (3.5,0.7) (3,0) (2.5,-0.7) (2,0) (1.5,0.7) (1,0) (0.5,-0.7) (0,0) (-0.5,0.8)};
\foreach \x/\i in {4/0,3/1,2/2,1/3,0/4} \fill (\x,0) circle (1.5pt) node[below right] {\tiny$p_{\i}$};
\node at (2,-2.67) {\small(c) a lift of $L_0$ meeting a lift of $L_1$ five times};
\end{scope}
\end{tikzpicture}
\caption{(a) The lift of the unperturbed curve for $T(3,5)$, at $\varepsilon=0$ (Corollary~\ref{cor:polygon}); the dashed segment is $\Omega_2$
when it follows $H$. Vertical units are multiples of $\pi$. (b) Its image under $S_{-m}$ for $q=11$, with the lines $\mathcal D_k$;
the lines $\mathcal D_3$ and $\mathcal D_4$ are crossed three times. (c) The configuration of Lemma~\ref{lem:zigzag}: the points
$p_0,\dots,p_4$ are ordered by decreasing $\gamma$, and the lobes alternate above and below the lift $\Sigma$ of $L_0$.}
\label{fig:sheared}
\end{figure}

\subsection{The perturbed curves}\label{ss:perturbed-P}

Besides (P1) and (P2) we use the following results of~\cite[\S\S5--6]{WuebbenV}, stated for $n=5$. In them $c$
is one of the double points $c_\pm$ of (V1), $\ell$ is the covector of (V5), and $L_1$ and $\tilde A$ are the curve
of~\cite{WuebbenV} and its lift, before the shear; statements that involve $L_0$ hold, as in~\cite[\S6]{WuebbenV}, for
$\varepsilon\in\mathcal C_a$ sufficiently small and then every sufficiently small earring parameter.
\begin{enumerate}
\item[(P3)] \emph{The reduced function}~\cite[\S5.2]{WuebbenV}. The differential $d\Psi(c)$ has rank one, its kernel is
spanned by $T_\Omega$ and the tangent $T_C$ of $C_0$ at $c$, and the restriction of $\ell\cdot D^2\Psi(c)$ to $\ker d\Psi(c)$ is
a nondegenerate indefinite quadratic form~\cite[Prop.~3.11]{WuebbenV}. In coordinates $(s,t,\varrho)$ centered at $c$
with $\partial_s=T_\Omega$ and $\partial_t=T_C$ at $c$, for a covector $m$ not proportional to $\ell$ the function $m\cdot\Psi$ has nonzero $\partial_\varrho$-derivative at $c$ and can be solved for
$\varrho=\varrho(s,t,\varepsilon)$ by the implicit function theorem, and near $c$ the variety $W_\varepsilon$ is the zero set of the
\emph{reduced function}
\[
f_c(s,t;\varepsilon)=\ell\cdot\Psi\bigl(\varepsilon;s,t,\varrho(s,t,\varepsilon)\bigr),
\]
that is, of the Lyapunov--Schmidt reduction of $\Psi$ to $\ker d\Psi(c)$; its Hessian at $(0,0;0)$ is the quadratic form
above. For small $\varepsilon$ the function $f_c(\cdot;\varepsilon)$ has a unique critical point near $0$, depending
analytically on $\varepsilon$, with critical value $\kappa_c(\varepsilon)$, and there are coordinates $(s',t')$, depending
analytically on $\varepsilon$, with $f_c=\kappa_c(\varepsilon)+s't'$; at $\varepsilon=0$ they are chosen so that $\{t'=0\}$ is
$\Omega$ and $\{s'=0\}$ is $C_0$ near $c$. Moreover $\kappa_c(0)=0$ and $d\kappa_c(0)=\ell\cdot\partial_\varepsilon\Psi(0;c)$.
\item[(P4)] \emph{The cone}~\cite[Def.~5.5]{WuebbenV}. For $a>0$, $\mathcal C_a$ is the set of $\varepsilon\in\R^2$ such that
the linear part $d\kappa_c(0)\varepsilon$ of $\kappa_c(\varepsilon)$ has absolute value at least $a|\varepsilon|$ at both double points $c=c_\pm$, for the covector $\ell$ of (V5); another normalization of $\ell$ rescales $a$.
\item[(P5)] \emph{Near the double points}~\cite[Lemma~5.6 and its proof]{WuebbenV}. Let $P_\Omega$ and $V$ be the vectors
with $\Pi(s',0)=\Pi(c)+s'P_\Omega+O(s'^2)$ and $\Pi(0,t')=\Pi(c)+t'^2V+O(t'^3)$ at $\varepsilon=0$; they are linearly
independent. Let $a>0$. There are $r_1>0$ and constants $C,c_0>0$ such that for every $r_0\in(0,r_1]$ there is $\varepsilon_0>0$ with the following properties for every $\varepsilon\in\mathcal C_a$ with $0<|\varepsilon|<\varepsilon_0$. Let $U_c$ be the set of points of a fixed neighborhood of $c$
whose coordinates $(s,t)$ have Morse coordinates $(s',t')$ with $|s'|\le r_0$ and $|t'|\le r_0$; it depends on
$\varepsilon$, but it contains and is contained in fixed neighborhoods of $c$. Each of the two arcs of $W_\varepsilon\cap U_c$
is mapped by $\Pi$ to an embedded arc with nowhere vanishing velocity. Such an arc is parametrized by $t'$, with
$s'=-\kappa/t'$, $\kappa=\kappa_c(\varepsilon)$ and $|\kappa|/r_0\le|t'|\le r_0$, and its velocity is
\[
\frac{d}{dt'}\Pi_\varepsilon=\frac{\kappa}{t'^2}P_\Omega+2t'V+e(t'),\qquad
|e(t')|\le C\bigl(r_0+|\varepsilon|^{2/3}a^{-1/3}\bigr)A(t'),\qquad A(t')=\frac{|\kappa|}{t'^2}+|t'|,
\]
with $C$ independent of $\varepsilon$, $a$ and $r_0$; the main term $(\kappa/t'^2)P_\Omega+2t'V$ has length at least $c_0A(t')$.
\item[(P6)] \emph{Global structure}~\cite[Prop.~5.7(1), (2), (4)]{WuebbenV}. Let $a>0$. For every
$\varepsilon\in\mathcal C_a$ with $0<|\varepsilon|$ sufficiently small, $R_\pi(Y,T)$ is a compact $1$-manifold with two
boundary points $r_\pm(\varepsilon)$, and $L_1$ is an immersion into $\Pst$ away from $r_\pm(\varepsilon)$. Away from
$U_{c_\pm}$ and from the corner charts of (P7), $W_\varepsilon$ is $C^1$-close to $W$: it is the image of a $C^1$-small
normal graph over $W$, with $\Pi_\varepsilon$ correspondingly $C^1$-close to $\Pi$. The union $\Omega\cup C_0$ is replaced by its
resolution at the two double points: either the arc is $\Omega_0\cdot C_0^{\mathrm X}\cdot\Omega_2$ and there is one further
circle $\Omega_1\cdot(C_0^{\mathrm Y})^{-1}$, or the arc is $\Omega_0\cdot C_0^{\mathrm X}\cdot\Omega_1^{-1}\cdot C_0^{\mathrm Y}\cdot\Omega_2$
and there is no further circle, where $\{\mathrm X,\mathrm Y\}=\{\mathrm L,\mathrm R\}$ and each product denotes the
resolved curve, which is $C^1$-close to the listed pieces away from $c_\pm$ and is described by (P5) near them. For
$n=5$ the variety $W$ has no non-central circles~\cite[Prop.~3.6]{WuebbenV}.
\item[(P7)] \emph{Near the corners}~\cite[Lemma~5.1 and Cor.~5.3]{WuebbenV}. There are a neighborhood $U_r$ of
$r_\pm(0)$, a number $y_0>0$ and a neighborhood $\mathcal O$ of $0\in\R^2$ such that for every $\varepsilon\in\mathcal O$,
$R_\pi(Y,T)\cap U_r$ is a single embedded half-open arc $\{\Gamma_\varepsilon(y):0\le y<y_0\}$ with $\Gamma_\varepsilon(0)=r_\pm(\varepsilon)$,
$\Pi$ maps $r_\pm(\varepsilon)$ to the corner $\Pi(r_\pm(0))$, and, with a suitable lift,
$\widetilde\Pi(\Gamma_\varepsilon(y))=y\,\mathbf c(\varepsilon)+y^3R(y,\varepsilon)$, where $\mathbf c(\varepsilon)=(c_1(\varepsilon),c_2(\varepsilon))\in\R^2$ and $R$ are real-analytic and $\mathbf c(0)$ is a nonzero vector along the image of $\Omega$. With constants independent of $\varepsilon$: the limiting slope of
$L_1$ at $r_+$ is the slope of the image of $\Omega$ plus $O(|\varepsilon|)$; on the lift through the lattice point,
$\varphi=y\,(c_2(\varepsilon)-c_1(\varepsilon))+O(y^3)$ with $c_2-c_1$ bounded away from $0$, so that $\varphi$ is strictly monotone
on $0<y<y_0$; and if $U_r$ is the neighborhood of $r_+(0)$, then for every sufficiently small earring parameter $L_0$
meets $L_1(R_\pi(Y,T)\cap U_r)$ in exactly one point, the generator $r_+$ of~\cite[Lemma~5.1]{HHK2}. The statements about the limiting slope and about $\varphi$ also hold at the other end of the arc, whose image is the corner $(\pi,0)$.
\item[(P8)] \emph{The arc and the remnant circle}~\cite[\S6.1, Lemmas~6.1 and~6.2]{WuebbenV}. Let $\tilde A$ be the lift
of the arc component of $L_1$ that starts at $(0,0)$ and enters $\{\gamma>0\}$. If the arc is
$\Omega_0\cdot H\cdot\Omega_1^{-1}\cdot H'\cdot\Omega_2$, then $\tilde A$ joins $(0,0)$ to $(\pi,8\pi)$, and if it is
$\Omega_0\cdot H\cdot\Omega_2$, to $(\pi,4\pi)$. The interior of $\tilde A$ lies in the open strip $0<\gamma<\pi$, $\tilde A$ is
embedded, and $\varphi$ along $\tilde A$ is piecewise monotone, taking the values
$0,-\frac\pi6,3\pi+\frac\pi6,4\pi-\frac\pi6,7\pi+\frac\pi6,7\pi$, resp.\ $0,-\frac\pi6,3\pi+\frac\pi6,3\pi$, at the ends of the
pieces of the unperturbed curve, up to $O(|\varepsilon|)$. When the arc is $\Omega_0\cdot H\cdot\Omega_2$, the further circle
of (P6), the \emph{remnant circle} $\Omega_1^{-1}\cdot H'$, misses the left and right edges, $\varphi$ is strictly monotone along it with
$d\varphi\neq0$, and $\varphi$ changes by $4\pi$ over one period; its vertical degree is $2$.
\item[(P9)] \emph{Bigons and admissibility}~\cite[Lemmas~6.4 and~6.6]{WuebbenV}. The boundary of a bigon from $p$ to $p'$
lifts to a closed curve in $\R^2\setminus(\pi\Z)^2$ consisting of a sub-path $\tilde\beta$ of one lift of a component of
$L_1$, from a lift $\tilde p'$ of $p'$ to a lift $\tilde p$ of $p$, and a sub-arc of one lift $S^k_\pm$ of $L_0$ from $\tilde p$
to $\tilde p'$. If $\varphi$ is strictly monotone along $\tilde\beta$, then no such bigon exists; in particular a component of
$L_1$ along which $\varphi$ is strictly monotone carries no bigon. Moreover $L_1$ is a restricted immersed
$1$-manifold, and $(L_0,L_1)$ is an admissible pair.
\end{enumerate}
By (V2) and (V4), $P_\Omega\parallel(1,0)$ and $V\parallel\pm(1,6)$. The first statement of (P9) holds with any immersed
curve in place of $L_1$, because $\R^2\setminus(\pi\Z)^2\to\Pst$ is a covering. Where we apply an argument
of~\cite{WuebbenV} to the sheared curve, rather than a statement of (P5)--(P9) to the curve of~\cite{WuebbenV}, we say so.

For $\varepsilon\in\mathcal C_a$ small, each of $c_\pm$ is resolved: the two arcs of $W_\varepsilon$ near $c$ each join a
half-branch of $\Omega$ to a half-branch of $C_0$. Following~\cite[\S6.1]{WuebbenV}, we name the two alternatives of (P6):
\begin{itemize}
\item the \emph{joined resolution}, in which $R_\pi(Y,T)$ is a single arc $\Omega_0\cdot H\cdot \Omega_1^{-1}\cdot H'\cdot \Omega_2$;
\item the \emph{split resolution}, in which $R_\pi(Y,T)$ is the arc $\Omega_0\cdot H\cdot \Omega_2$ together with the remnant
circle $\Omega_1^{-1}\cdot H'$.
\end{itemize}
Here each product denotes the resolved curve, as in (P6). The analysis of~\cite{WuebbenV} does not need to know which
type occurs. After the shear it matters, and we determine it.

\begin{lemma}\label{lem:resolvedlambda}
Let $c\in\{c_+,c_-\}$, $a>0$, and let $\lambda'$ be a linear function on $\R^2$ with $\lambda'(P_\Omega)\neq0$ and $\lambda'(V)\neq0$. Let
$U_c$ (the box of half-width $r_0$ in the Morse coordinates) and $\varepsilon_0$ be as in (P5). After shrinking $r_0$ and $\varepsilon_0$, the following
holds for every $\varepsilon\in\mathcal C_a$ with $0<|\varepsilon|<\varepsilon_0$ and each of the two arcs of $W_\varepsilon\cap U_c$, traversed from its
end on $\Omega$ to its end on $C_0$ or conversely.
\begin{enumerate}
\item Call \emph{arrival direction} the direction of traversal along the first half-branch towards $c$, and
\emph{departure direction} that along the second half-branch away from $c$; each is a multiple of $\pm P_\Omega$ or $\pm V$.
For any prescribed $\vartheta>0$, the velocity of $\Pi$ along the arc stays within angle $\vartheta$ of the closed convex cone
spanned by the arrival and departure directions, which has opening less than $\pi$. Consequently, if $\ell'$ is a constant line field not parallel to the arrival or departure
direction, and $\vartheta$ is small, the net number of passages of the tangent line through $\ell'$ is $\pm1$ if one of the two
directions of $\ell'$ lies inside that cone, the sign being that of the rotation from the arrival to the departure
direction, and $0$ otherwise.
\item If the two joined half-branches move $\lambda'$ in opposite senses away from $c$, then $\lambda'\circ\Pi$ is strictly monotone
along the arc.
\item Along the arc, $|\lambda'\circ\Pi-\lambda'(\Pi(c))|\le C_{\lambda'}(r_0+|\varepsilon|)$, with $C_{\lambda'}$ independent of $\varepsilon$.
\end{enumerate}
\end{lemma}

\begin{proof}
By (P5) the arc is parametrized by $t'$, with $s'=-\kappa/t'$, $\kappa=\kappa_c(\varepsilon)$ and
$|\kappa|/r_0\le|t'|\le r_0$, and its velocity is $(\kappa/t'^2)P_\Omega+2t'V+e(t')$ with
$|e|\le C\bigl(r_0+|\varepsilon|^{2/3}a^{-1/3}\bigr)A(t')$, where $A(t')=|\kappa|/t'^2+|t'|$ and $C$ does not depend on $r_0$. Along the arc the signs of $\kappa$ and $t'$ are
fixed, so the main term is a positive combination of $\operatorname{sign}(\kappa)P_\Omega$ and $\operatorname{sign}(t')V$, of length at least $c_0A(t')$.
At the end where $|t'|\approx r_0$ the term $2t'V$ dominates, and at the end where $|t'|\approx|\kappa|/r_0$ the term
$(\kappa/t'^2)P_\Omega$ does; these are the directions of traversal along the two half-branches. Shrinking $r_0$ and $\varepsilon_0$
makes $|e|/A$ as small as we wish, which gives the first assertion of (1). The tangent direction moves continuously
in a sector of opening less than $\pi$, which gives the count of passages. For (2), $\lambda'$ applied to the main term is
$\kappa\lambda'(P_\Omega)/t'^2+2\lambda'(V)t'$. In the case of opposite senses both summands have the same sign, so its absolute value is
at least a fixed multiple of $A$, and the error does not change the sign. For (3), $\Pi_\varepsilon$ is uniformly Lipschitz on
$U_c$ and $\Pi_\varepsilon(c)=\Pi(c)+O(|\varepsilon|)$, by the smooth dependence in (P2).
\end{proof}

For $\lambda'=\lambda$ we have $\lambda(1,0)=q-6\neq0$ and $\lambda(1,6)=q\neq0$, so the lemma applies at $c_+$ and $c_-$.

\begin{proposition}\label{prop:restype}
Let $a>0$ and $\varepsilon\in\mathcal C_a$ with $|\varepsilon|$ small. The critical values at $c_\pm$ are
\[
\kappa_\pm(\varepsilon)=\mp\tfrac{\sqrt3}2\,\varepsilon_A-\tfrac12\,\varepsilon_B+O(|\varepsilon|^2),
\]
for the covector $\ell=(-2\tau,-1)$ of (V5). The resolution is
\[
\text{joined}\iff|\varepsilon_B|<\sqrt3\,|\varepsilon_A|,\qquad\text{split}\iff|\varepsilon_B|>\sqrt3\,|\varepsilon_A| .
\]
For small $a$ both types occur in $\mathcal C_a$: the joined sectors contain the $\varepsilon_A$-axis, and the split sectors the
$\varepsilon_B$-axis. The set excluded from $\mathcal C_a$ is an open conical neighborhood of the two lines $\{\varepsilon_B=\mp\sqrt3\,\varepsilon_A\}$,
the boundaries between the sectors.
\end{proposition}

\begin{proof}
At $c=(\frac\pi2,\frac\pi2,\tau)$, $\tau=\pm\tau_1$, put $\Lambda=\ell\cdot D^2\Psi(T_\Omega,X)=-8\tau$ (V5). The tangent $T_C$ of $C_0$ at $c$ is the
second null direction of the reduced Hessian; it has the form $X+\alpha T_\Omega$, with $\partial_v$-component $1$ because $v$ is a
regular parameter of $C_0$ at $c$ (V4). Since $\ell\cdot D^2\Psi(T_\Omega,T_\Omega)=0$, we have $\ell\cdot D^2\Psi(T_\Omega,T_C)=\Lambda$. So in
coordinates $(s,t)$ along $(T_\Omega,T_C)$ the reduced function of (P3) is $f_c=\kappa_c(\varepsilon)+\Lambda st+(\text{higher order})$, and the arc of
$\{f_c=0\}$ in the quadrant with $\operatorname{sign}(s)\operatorname{sign}(t)=-\operatorname{sign}(\kappa_c\Lambda)$ joins the half-branch of $\Omega$ in direction
$\operatorname{sign}(s)T_\Omega$ to the half-branch of $C_0$ in direction $\operatorname{sign}(t)T_C$. Now $+T_C$ points into $C_0^{\mathrm R}$ at both double
points, and $+T_\Omega=\partial_\tau$ points into $\Omega_0$ at $c_+$ and into $\Omega_1$ at $c_-$. At $c_+$, $\Lambda<0$, so $\Omega_0$ is joined to the
half with $\operatorname{sign}(t)=\operatorname{sign}(\kappa_+)$; at $c_-$, $\Lambda>0$, so $\Omega_2$ (where $s<0$) is joined to the half with
$\operatorname{sign}(t)=\operatorname{sign}(\kappa_-)$. The resolution is split exactly when $\Omega_0$ and $\Omega_2$ are joined to the same half, that
is, when $\kappa_+\kappa_->0$.

For the critical values we use the phases $\Phi_i,\Theta_i$ of (P2); since $\Psi=(\mathrm{E2},\mathrm{E1})$ at $\varepsilon=0$, the
covector $\ell$ of (V5) pairs with $\Psi$ as $\ell\cdot\Psi=-2\tau\Psi_2-\Psi_1$. At $c$ the phases give $\partial_\varepsilon\Psi_1=(-3\tau,2)$ and $\partial_\varepsilon\Psi_2=(2,-\tau)$, so
$d\kappa_c(0)=\ell\cdot\partial_\varepsilon\Psi=-2\tau(2,-\tau)-(-3\tau,2)=(-\tau,\,2\tau^2-2)=(-\tau,-\tfrac12)$. Hence
$\kappa_+\kappa_-=\frac14\varepsilon_B^2-\frac34\varepsilon_A^2+O(|\varepsilon|^3)$. On $\mathcal C_a$ both linear parts are at least $a|\varepsilon|$ in absolute
value, so for $|\varepsilon|$ small the signs of $\kappa_\pm$ are those of their linear parts.
\end{proof}

In particular the proposition decides which alternative of (P6) occurs for $n=5$.

Let $\tilde A$ be the lift of the arc of (P8), starting at $O$, and $\hat A=S_{-m}(\tilde A)$. In the split case let
$\tilde C$ be the lift of the remnant circle through the lifts of Corollary~\ref{cor:polygon}, and $\hat C=S_{-m}(\tilde C)$. The values of
$\varphi$ at the sheared vertices are
\begin{equation}\label{eq:phiJ}
\begin{gathered}
\varphi(S_{-m}O)=0,\quad \varphi(S_{-m}J_1)=\tfrac{(q-6)\pi}6,\quad \varphi(S_{-m}J_2)=4\pi+\tfrac{5(q-6)\pi}6,\quad \varphi(S_{-m}J_3)=4\pi+\tfrac{(q-6)\pi}6,\\
\varphi(S_{-m}J_4)=8\pi+\tfrac{5(q-6)\pi}6,\qquad \varphi(S_{-m}Z)=(q+2)\pi,\qquad \varphi(S_{-m}Z')=(q-2)\pi .
\end{gathered}
\end{equation}
We write $\varphi_i=\varphi(S_{-m}J_i)$. Since $q-6$ is odd, each $\varphi_i$ differs from $4\pi\Z$ by an odd multiple of $\pi/6$, and so
lies at distance at least $\pi/6$ from $2\pi\Z$.

\begin{proposition}\label{prop:shearedcurves}
Let $a>0$. There are neighborhoods $U_{J_i}$ of the junctions $S_{-m}J_i$ (the images of $U_{c_\pm}$) and of the corners such
that for every $\varepsilon\in\mathcal C_a$ with $0<|\varepsilon|$ sufficiently small, the threshold depending on $q$:
\begin{enumerate}
\item $\hat A$ is an embedded proper arc from $(0,0)$ to $(\pi,(q+3)\pi)$ in the joined case and to $(\pi,(q-1)\pi)$ in
the split case, with interior in the open strip. It leaves $(0,0)$ with limiting slope $q-5+O(|\varepsilon|)$, and $\varphi>0$ on
its initial segment.
\item Outside $\bigcup_iU_{J_i}$ and the corner neighborhoods, $\hat A$ is $C^1$-close to the image under $S_{-m}$ of the polygon
of Corollary~\ref{cor:polygon}, and $\varphi$ is strictly monotone with $|d\varphi|$ bounded below: increasing near
$S_{-m}(\Omega_0)$, $S_{-m}(H)$, $S_{-m}(H')$, $S_{-m}(\Omega_2)$, and decreasing near $S_{-m}(\Omega_1^{-1})$. In $U_{J_1}$, in $U_{J_4}$, and at the
junction of $H$ with $\Omega_2$ in the split case, $\varphi$ is strictly monotone. In $U_{J_2}$ and $U_{J_3}$,
$|\varphi-\varphi_i|<\pi/12$. Near the corners $|d\varphi|$ is bounded below.
\item For a sub-arc $\alpha$ of $\hat A$ with endpoints outside $\bigcup_iU_{J_i}$, $\mu(\alpha,\ell_1)$ is the number of
neighborhoods $U_{J_2}$ traversed by $\alpha$ minus the number of $U_{J_3}$ traversed. All critical points of $\varphi$ on $\hat A$
lie in $U_{J_2}\cup U_{J_3}$.
\end{enumerate}
In the split case $\hat C$ is a restricted immersed circle that misses the edges and has vertical degree $2$. Its lift
is $(0,4\pi)$-periodic and alternates between pieces $C^1$-close to $S_{-m}(\Omega_1^{-1})+(0,4\pi j)$, on which $\varphi$ decreases from
$\varphi_2+4\pi j$ to $\varphi_3+4\pi j$, and pieces $C^1$-close to $S_{-m}(H')+(0,4\pi j)$, on which $\varphi$ increases from
$\varphi_3+4\pi j$ to $\varphi_2+4\pi(j+1)$; the junctions are of the types $J_2$ and $J_3$, and (2), (3) hold on it verbatim.
\end{proposition}

\begin{proof}
(1) By (P8), $\tilde A$ is embedded with interior in the strip and ends at $Z$ or $Z'$, and $S_{-m}$
preserves $\gamma$ and maps $(\pi,2\pi j)$ to $(\pi,2\pi j+(q-5)\pi)$. The limiting slope is that of $\Pi(\Omega)$, which is $0$, plus
$q-5$, up to $O(|\varepsilon|)$, by (P7) and Lemma~\ref{lem:shear}(c). In the corner chart
of (P7) the lift is $y\,(c_1(\varepsilon),c_2(\varepsilon))+O(y^3)$ with $(c_1,c_2)(0)$ a positive multiple of $(1,0)$, so
on $\hat A$ we get $\varphi=y\,\bigl(c_2+(q-6)c_1\bigr)+O(y^3)$, which is positive for small $y>0$.

(2) Away from the junctions and corners, $\tilde A$ is $C^1$-close to the polygon of Corollary~\ref{cor:polygon},
by (P6). The image under $S_{-m}$ of the polygon has velocities that are positive multiples of
$(1,q-5)$ on $S_{-m}(\Omega_0)$ and $S_{-m}(\Omega_2)$, of $(1,q+1)$ on $S_{-m}(H)$ and $S_{-m}(H')$, and of $(-1,5-q)$ on $S_{-m}(\Omega_1^{-1})$, along which
$d\varphi$ is $q-6$, $q$ and $-(q-6)$. At the junctions we apply Lemma~\ref{lem:resolvedlambda} with $\lambda'=\lambda$:
\begin{center}\small
\begin{tabular}{llll}
\toprule
junction & first half-branch, away from $c$ & second half-branch & conclusion\\
\midrule
$J_1$ at $c_+$, $\Omega_0\to H$ & $\Omega_0$ decreases $\lambda$ & $H$ increases $\lambda$ & (2): monotone\\
$J_2$ at $c_-$, $H\to \Omega_1^{-1}$ & $H$ decreases $\lambda$ & $\Omega_1^{-1}$ decreases $\lambda$ & (3)\\
$J_3$ at $c_+$, $\Omega_1^{-1}\to H'$ & $\Omega_1$ increases $\lambda$ & $H'$ increases $\lambda$ & (3)\\
$J_4$ at $c_-$, $H'\to \Omega_2$ & $H'$ decreases $\lambda$ & $\Omega_2$ increases $\lambda$ & (2): monotone\\
split, at $c_-$, $H\to \Omega_2$ & $H$ decreases $\lambda$ & $\Omega_2$ increases $\lambda$ & (2): monotone\\
\bottomrule
\end{tabular}
\end{center}
Choose $r_0$ with $C_\lambda r_0<\pi/24$ and then $|\varepsilon|$ with $C_\lambda|\varepsilon|<\pi/24$. Near the corners we apply the argument of (P7) to the sheared curve, using Lemma~\ref{lem:shear}(c): it uses only that
the limiting slope differs from $1$, and here it is $q-5\ge2$.

(3) By Lemma~\ref{lem:shear}(f), $\mu(S_{-m}\circ\tilde\alpha,\ell_1)=\mu(\tilde\alpha,\ell')$, where $\ell'$ is spanned by $(1,6-q)$, at angle
$-\arctan(q-6)$, equivalently $\pi-\arctan(q-6)\in(\frac\pi2,\frac{3\pi}4]$. We compute on the unsheared curve. On the
$C^1$-close pieces the directions are near $(\pm1,0)$ or $(1,6)$, and near the corners the slope is $O(|\varepsilon|)$; none of these
is parallel to $(1,6-q)$, so there are no passages there. In $U_{J_2}$ the velocity moves, by Lemma~\ref{lem:resolvedlambda}(1),
from near $(1,6)$ to near $(-1,0)$ inside the cone of angles $[\arctan6,\pi]$. This cone contains the direction at angle
$\pi-\arctan(q-6)$ and not its opposite, and the rotation is counterclockwise; so the count is $+1$. In $U_{J_3}$ the
motion is reversed and the count is $-1$. In $U_{J_1}$, $U_{J_4}$ and the split junction the cone is $[0,\arctan6]$, which
contains neither direction. A critical point of $\varphi$ on $\hat A$ is a point where the tangent of $\hat A$ is parallel to
$(1,1)$, that is, where the tangent of $\tilde A$ is parallel to $(1,6-q)$; by the above this happens only in
$U_{J_2}\cup U_{J_3}$.

For the remnant circle, the properties of being restricted, avoiding the edges and having vertical degree $2$ follow from (P8), (P9) and
Lemma~\ref{lem:shear}(b)--(d). Its pieces are those of Corollary~\ref{cor:polygon}, with junctions of type $J_2$ (from $H'$
to $\Omega_1^{-1}$, at $c_-$) and $J_3$ (from $\Omega_1^{-1}$ to $H'$, at $c_+$), and the estimates above apply to it by compactness and
periodicity.
\end{proof}

\begin{remark}\label{rem:threshold}
The smallness threshold for $\varepsilon$ in Proposition~\ref{prop:shearedcurves} depends on $q$ through the constant $C_\lambda$, which
grows linearly in $q$. The margins that matter in the next subsection are the distances from $\varphi_3$ to the nearest
value in $2\pi\Z$ above it and from $\varphi_2$ to the nearest value in $2\pi\Z$ below it. Both are at least $\pi/6$; the first
equals $\pi/6$ exactly when $q\equiv5\pmod{12}$, and the second exactly when $q\equiv11\pmod{12}$. The order of the
crossings in Proposition~\ref{prop:bigonsP} also requires the perturbation to move the crossings by less than half the gaps
between their abscissae at $\varepsilon=0$, which are at least $\pi/(q(q-6))$ and so shrink quadratically in $q$. Above the
threshold a canceling pair of generators can be missing (a line $\mathcal D_k$ crossed once instead of three times); in our
computations (Remark~\ref{rem:numericsP}) the homology is the same.
\end{remark}

\subsection{Generators}\label{ss:gens-P}

Put
\[
N_q=\#\Bigl\{j\in\Z:\ \tfrac{q-6}{12}<j<\tfrac{5(q-6)}{12}\Bigr\}=\Bigl\lfloor\tfrac{5(q-6)}{12}\Bigr\rfloor-\Bigl\lfloor\tfrac{q-6}{12}\Bigr\rfloor ;
\]
both bounds are non-integers because $q-6$ is odd. Thus $N_7=0$, $N_9=1$, $N_{11}=N_{13}=2$, $N_{15}=3$ and
$N_{17}=N_{19}=4$. We call a transverse crossing of a lift of $L_1$ with a line $\mathcal D_k$ an \emph{up-crossing} or a
\emph{down-crossing} according as $\varphi$ increases or decreases through $2\pi k$ there.

\begin{proposition}\label{prop:gens}
For $\varepsilon\in\mathcal C_a\setminus\{0\}$ sufficiently small, and then $\epsilon$ sufficiently small:
\begin{enumerate}
\item $L_0$ and $L_1$ meet transversally. The arc carries $r_+$, which lies on $S^0_+$, and every component carries a pair
$x_P^\pm\in S^k_\pm$ over each transverse crossing $P$ of (one period of) its lift with a line $\mathcal D_k$, with $k\ge1$ for the
arc. No generator lies over the double point of $L_0$.
\item In the joined case, $\hat A$ crosses $\mathcal D_k$ three times (up, down, up) if $\varphi_3<2\pi k<\varphi_2$, that is for $k=2+j$
with $\frac{q-6}{12}<j<\frac{5(q-6)}{12}$, and once (up) for every other $k\in\{1,\dots,\frac{q+1}2\}$; there are no other
crossings. The arc carries $q+2+4N_q$ generators.
\item In the split case, $\hat A$ crosses $\mathcal D_1,\dots,\mathcal D_{(q-3)/2}$ once each (up) and carries $q-2$ generators; one period
of $\hat C$ crosses $N_q$ lines downward and $N_q+2$ upward, and $\hat C$ carries $4+4N_q$ generators.
\end{enumerate}
In both cases the complex has $q+2+4N_q$ generators.
\end{proposition}

\begin{proof}
The argument of~\cite[Prop.~6.3]{WuebbenV} applies. By Proposition~\ref{prop:shearedcurves}(2), outside small neighborhoods
of its transverse crossings with the lines $\mathcal D_k$ and of the corners, each curve stays a positive distance from
$\bigcup_k\mathcal D_k$: in $U_{J_2}$ and $U_{J_3}$ the values of $\varphi$ stay within $\pi/12$ of $\varphi_i$, and the same holds in
$U_{J_1}$, $U_{J_4}$ by Lemma~\ref{lem:resolvedlambda}(3). Near a crossing, $|d\varphi|$ is bounded below and the slope is
bounded away from $1$, so the curve meets each of $S^k_+$ and $S^k_-$ once, transversally. The corner $(0,0)$ contributes
exactly $r_+$, by~\cite[Lemma~5.1]{HHK2} and the argument of (P7) applied to $\hat A$. Since $\hat A$ leaves $(0,0)$ with
slope $q-5>1$, it meets the lift $S^0_+$ of $L_0$, which passes $(0,0)$ through $(-\epsilon,\epsilon)$, near $\gamma=2\epsilon/(q-6)>0$, and it does not meet $S^0_-$ near
the corner; so $r_+\in S^0_+$. The far endpoint lies on $\{\varphi\in(2\Z+1)\pi\}$, away from every $\mathcal D_k$.

In the joined case $\varphi$ is monotone along the pieces over $(0,\varphi_2)$ (near $S_{-m}(\Omega_0)$ and $S_{-m}(H)$), over $(\varphi_3,\varphi_2)$
(near $S_{-m}(\Omega_1^{-1})$) and over $(\varphi_3,(q+2)\pi)$ (near $S_{-m}(H')$ and $S_{-m}(\Omega_2)$). Here $\varphi_3>0$ and $(q+2)\pi-\varphi_2=(q+18)\pi/6>0$.
The lines with $2\pi k\in(\varphi_3,\varphi_2)$ are those with $k=2+j$ and $\frac{q-6}{12}<j<\frac{5(q-6)}{12}$, and the lines with
$0<2\pi k<(q+2)\pi$ are $k=1,\dots,\frac{q+1}2$. This gives $1+2\bigl(\frac{q+1}2+2N_q\bigr)$ generators. In the split case
$\varphi$ increases monotonically along the arc over $(0,(q-2)\pi)$; on one period of the circle it decreases over
$(\varphi_3,\varphi_2)$, crossing $N_q$ lines, and then increases over $(\varphi_3,\varphi_2+4\pi)$, crossing $N_q+2$ lines.

At $\varepsilon=0$ every crossing of the sheared polygons with a line $\mathcal D_k$ lies on a piece of $\Omega$, where $\varphi=(q-6)\gamma+4\pi j$,
or on a half of $C_0$, where $\varphi=q\gamma+(2j+1)\pi$; so its abscissa is $2\pi i/(q-6)$ or $(2i+1)\pi/q$ with $i\in\Z$. Neither
equals $\pi/2$, because $q-6$ and $q$ are odd. Hence for small perturbations no crossing lies near the double point
$(\frac\pi2,\frac\pi2)$ of $L_0$.
\end{proof}

For $q=7$ there are $9=1+|\sigma(K_7)|$ generators. In general the count exceeds $1+|\sigma(K_q)|=q+2$ by $4N_q$; we compare it
with the traceless characters in Remark~\ref{rem:FKPcount}.

\subsection{Gradings}\label{ss:grad-P}

\begin{lemma}\label{lem:zstepgen}
Let $p,p'$ be generators with lifts $\tilde p\in S^k_\pm$ and $\tilde p'\in S^{k'}_\pm$ on lifts of $L_0$ of the same sign, joined
by a sub-path $\tilde\alpha$ of a lift of a component of $L_1$ that lies in the open strip and has slope bounded away from $1$
at $\tilde p$ and $\tilde p'$. Then
\[
\gr(p,p')\equiv\mu(\alpha_1,\ell_1)+2(k-k')\pmod4,
\]
where $\alpha_1$ is the projection of $\tilde\alpha$, run from $p'$ to $p$.
\end{lemma}

\begin{proof}
This is the proof of~\cite[Lemma~6.8]{WuebbenV}, without its hypothesis that $\tilde\alpha$ contain no critical point of $\varphi$.
Apply~\eqref{eq:grading15} with $\alpha_0$ the branch of $L_0$ on the front face from $p$ to the point over $p'$; its lift runs
along $S^k_\pm$ from $\tilde p$ to $\tilde p'-(0,2\pi(k'-k))$ inside the open strip. Then $\mu(\alpha_0,\ell_1)=0$, because the slope of
$S^k_\pm$ stays on one side of $1$. The triple indices at $p$ and $p'$ agree: both are $0$ for the sign $+$ and $1$ for the sign $-$, since
the tangent line of $L_1$ is bounded away from slope $1$. Finally the lift of $\alpha_0*\alpha_1$ runs in the open strip from $\tilde p$
to $\tilde p+(0,2\pi(k-k'))$, so $\alpha_0*\alpha_1$ is homotopic in $\Pst$ to $k-k'$ times the core of the annulus obtained from the
strip, which encircles two corners; hence $z(\alpha_0*\alpha_1)\equiv2(k-k')$.
\end{proof}

\begin{proposition}\label{prop:degrees}
Put $\deg=\gr-\sigma(K_q)$, normalized by $\gr(r_+)=\sigma(K_q)$ as in~\eqref{eq:grading19}. For a crossing $P$ of $\hat A$ with $\mathcal D_k$,
\[
\deg x_P^+=\begin{cases}2k,&P\text{ an up-crossing},\\ 2k+1,&P\text{ a down-crossing},\end{cases}\qquad \deg x_P^-=\deg x_P^+-1 .
\]
On the remnant circle the same formula holds for the relative grading: for crossings $P,P'$ of $\hat C$ with $\mathcal D_k,\mathcal D_{k'}$,
$\deg x_P^+-\deg x_{P'}^+\equiv2(k-k')+\delta_P-\delta_{P'}$, with $\delta_P=1$ for down-crossings and $0$ for up-crossings.
\end{proposition}

\begin{proof}
On the arc, apply Lemma~\ref{lem:zstepgen} to $p=x_P^+\in S^k_+$ and $p'=r_+\in S^0_+$ along $\hat A$. By
Proposition~\ref{prop:shearedcurves}(3), $\mu$ counts $+1$ for each $U_{J_2}$ and $-1$ for each $U_{J_3}$ passed. Up-crossings lie
before $U_{J_2}$ or after $U_{J_3}$, so $\mu=0$; down-crossings lie between them, so $\mu=+1$. Then
$\gr(x_P^+,x_P^-)=1$~\cite[\S5.1]{HHK2}. On the circle the same argument applies along its lift: between two up-crossings the
neighborhoods $U_{J_2}$ and $U_{J_3}$ alternate, so their net count is $0$, and from an up-crossing to the next down-crossing
it is $+1$. Over one period $\mu=0$, $z\equiv0$ and $k$ changes by $2$, which is consistent.
\end{proof}

\subsection{Bigons}\label{ss:bigons-P}

A bigon lifts to an immersed disc in $\R^2\setminus(\pi\Z)^2$, since $\R^2\setminus(\pi\Z)^2\to\Pst$ is a covering, and by the first
statement of (P9) its boundary is a sub-path of one lift of a component of $L_1$ together with an arc of one lift of $L_0$.
The following lemma counts such discs in the configuration of Figure~\ref{fig:sheared}(c). Each lift $S^k_\pm$ of $L_0$ is a
graph over $\gamma$ and leaves the open strip at both ends; a transverse crossing with it is an up- or down-crossing
according as the curve passes from below $S^k_\pm$ to above it or conversely.

\begin{lemma}\label{lem:zigzag}
Let $\tilde L$ be an embedded lift of a component of $L_1$ that lies in the open strip $\{0<\gamma<\pi\}$, apart from endpoints at
lattice points, and is either an arc or a properly embedded line. Let $\Sigma$ be one of the lifts $S^k_\pm$ of $L_0$, and suppose that
$\tilde L\cap\Sigma=\{p_0,p_1,\dots,p_{2d}\}$, all transverse and in this order along $\tilde L$, with $p_{2i}$ up-crossings and $p_{2i+1}$
down-crossings, and with $\gamma(p_0)>\gamma(p_1)>\dots>\gamma(p_{2d})$. Then the bigons whose boundary lifts to a sub-path of $\tilde L$
and an arc of $\Sigma$ are exactly $2d$ bigons, one from $p_{2i-1}$ to $p_{2i-2}$ and one from $p_{2i-1}$ to $p_{2i}$ for each
$1\le i\le d$; each of them lifts to an embedded disc in $\R^2$.
\end{lemma}

\begin{proof}
For $0\le s<2d$ let $\Gamma_s$ be the union of the sub-path $\tilde L[p_s,p_{s+1}]$ and the arc $\Sigma[p_{s+1},p_s]$. The first meets $\Sigma$
only at its ends, because the points are consecutive along $\tilde L$, and the second contains no point of $\tilde L$ in its interior,
because they are consecutive in the order of $\gamma$. So $\Gamma_s$ is a Jordan curve; let $D_s$ be its bounded complementary
component. The two closed half-planes complementary to the strip are connected, unbounded and disjoint from $\Gamma_s$, so
$D_s$ lies in the strip and contains no lattice point. Between consecutive crossings the curve stays on one side of $\Sigma$,
so $D_s$ lies above $\Sigma$ if $p_s$ is an up-crossing and below $\Sigma$ if it is a down-crossing: the domains alternate. They are
pairwise disjoint, since their boundaries meet only at shared endpoints and $\overline{D_s}\cap\Sigma=\Sigma[p_{s+1},p_s]$.

\emph{Existence.} At the corners of $D_s$ the four continuations leave $D_s$ without entering it: the rays of $\tilde L$ beyond
$p_s$ and $p_{s+1}$ meet $\Gamma_s$ nowhere else, by embeddedness and the order of $\gamma$, and run to lattice points outside the strip
or to infinity; the rays of $\Sigma$ leave the strip. So both corners are convex, and $D_s$ projects to a bigon. If $D_s$
lies above $\Sigma$, its counterclockwise boundary runs along $\Sigma$ in the direction of increasing $\gamma$, from $p_{s+1}$ to $p_s$; if
below, from $p_s$ to $p_{s+1}$. In both cases the bigon runs from the down-crossing to the up-crossing, in agreement with
Proposition~\ref{prop:degrees} and~\cite[Prop.~3.8]{HHK2}.

\emph{Uniqueness.} The $L_1$-side of a bigon as in the statement is $\tilde L[p_i,p_j]$ for some $i\neq j$, and its $L_0$-side is
$\Sigma[p_j,p_i]$, since $\Sigma$ is an embedded line. If $|i-j|\ge2$, its boundary is, as a $1$-cycle, the sum of the $|i-j|$
consecutive curves $\Gamma_s$ between $p_i$ and $p_j$, each traversed with its sub-path of $\tilde L$ in the same direction. Since the
domains alternate above and below $\Sigma$, consecutive curves $\Gamma_s$ are traversed with opposite orientations, so the
boundary winds $-1$ times around the points of one of the domains. This is impossible for the boundary of an
orientation-preserving immersed disc, whose winding number around a point is its number of preimages. If $|i-j|=1$
the boundary is $\Gamma_s$, and only the direction from the down-crossing is counterclockwise. There is at most one bigon
from a given generator to another~\cite[Cor.~3.13]{HHK2}.
\end{proof}

\begin{proposition}\label{prop:bigonsP}
\begin{enumerate}
\item In the joined case the complex is the direct sum, over the lifts $S^k_\pm$ of $L_0$, of subcomplexes $C^\pm_k$. If $\mathcal D_k$ is
crossed once, $C^\pm_k$ has one generator. If $\mathcal D_k$ is crossed three times, at $P_1$ (up), $P_2$ (down) and $P_3$ (up),
then $\partial x^\pm_{P_2}=x^\pm_{P_1}+x^\pm_{P_3}$, and $H(C^\pm_k)=\F_2$, generated by $[x^\pm_{P_1}]$ in degree $2k$ for the sign $+$
and $2k-1$ for the sign $-$. There are no other bigons; in particular none involves $r_+$. The differential has rank $2N_q$.
\item In the split case the arc carries no bigon. The complex of the remnant circle is the direct sum of four chains,
one for each sign and each residue of $k$ modulo $2$; each has generators $p_0,\dots,p_{2d}$ with $\partial p_{2i-1}=p_{2i-2}+p_{2i}$, and homology
$\F_2$ in the degree of its up-crossings. The circle contributes rank $1$ in each of the four degrees, for every choice of
absolute grading on it; its differential has rank $2N_q$.
\end{enumerate}
\end{proposition}

\begin{proof}
By the first statement of (P9) the boundary of a bigon lifts to a sub-path of one lift of a component of $L_1$ and an arc of
one lift of $L_0$. For the arc there is one lift through each generator, $\hat A$. The only elements of $G$ preserving the open strip
are the vertical translations, so for the circle a sub-path of a lift is unique up to the translation $T=(0,4\pi)$, which
carries $S^k_\pm$ to $S^{k+2}_\pm$.

(1) On a lift of $L_0$ met three times, the order along $\hat A$ is $P_1$ (near $S_{-m}(H)$), $P_2$ (near $S_{-m}(\Omega_1^{-1})$), $P_3$ (near $S_{-m}(H')$),
because $\varphi_1<2\pi k<\varphi_2<\varphi_4$. At $\varepsilon=0$ the abscissae of the three crossings are
\[
\gamma_1=\frac{(2k+1)\pi}q,\qquad \gamma_2=\frac{2\pi(k-2)}{q-6},\qquad \gamma_3=\frac{(2k-3)\pi}q ,
\]
from $q\gamma-\pi=2\pi k$ on $S_{-m}(H)$, $4\pi+(q-6)\gamma=2\pi k$ on $S_{-m}(\Omega_1^{-1})$ and $q\gamma+3\pi=2\pi k$ on $S_{-m}(H')$. With $x=2\pi(k-2)$,
\[
\gamma_2-\gamma_3=\frac{6\bigl(x-\frac{(q-6)\pi}6\bigr)}{q(q-6)},\qquad \gamma_1-\gamma_2=\frac{6\bigl(\frac{5(q-6)\pi}6-x\bigr)}{q(q-6)} .
\]
Since $\varphi_3<2\pi k<\varphi_2$ and both $\varphi_2-2\pi k$ and $2\pi k-\varphi_3$ are at least $\pi/6$, both differences are at least
$\pi/(q(q-6))$. So for small perturbations $\gamma(P_1)>\gamma(P_2)>\gamma(P_3)$, and Lemma~\ref{lem:zigzag} applies with $d=1$. A lift of $L_0$
met once carries one generator, and $S^0_+$ carries only $r_+$, since $\varphi>0$ after the corner. The generators on $S^k_+$ and
on $S^k_-$ have degrees $(u,u+1,u)$ with $u=2k$, resp.\ $u=2k-1$ (Proposition~\ref{prop:degrees}). There are $N_q$ such lines,
each with two lifts of $L_0$ near it, so the differential has rank $2N_q$.

(2) The arc is $\varphi$-monotone, so it carries no bigon, by the argument of the second statement of (P9) applied to $\hat A$, using Proposition~\ref{prop:shearedcurves}(2) and, as in the proof of Proposition~\ref{prop:gens}, the fact that no crossing with a line $\mathcal D_k$ lies in a junction neighborhood. Along the lift of the circle the pieces
alternate, $\mathrm{dec}_j$ near $S_{-m}(\Omega_1^{-1})+(0,4\pi j)$ and $\mathrm{inc}_j$ near $S_{-m}(H')+(0,4\pi j)$, and each crosses a given line
$\mathcal D_k$ at most once. By Proposition~\ref{prop:shearedcurves}, the range of $\varphi$ on $\mathrm{dec}_j$ is the intersection of its
ranges on $\mathrm{inc}_{j-1}$ and $\mathrm{inc}_j$. Hence the crossings of $\mathcal D_k$ along the lift alternate between up and down,
beginning and ending with an up-crossing (as $\varphi\to\pm\infty$ along the lift), and a down-crossing on $\mathrm{dec}_j$ has the
up-crossings on $\mathrm{inc}_{j-1}$ and $\mathrm{inc}_j$ as neighbors. The triple $(\mathrm{inc}_{j-1},\mathrm{dec}_j,\mathrm{inc}_j)$ is the triple
$(S_{-m}(H),S_{-m}(\Omega_1^{-1}),S_{-m}(H'))$ of the arc translated by $(0,4\pi j)$, so the abscissa computation above shows that $\gamma$
decreases strictly along the sequence, and Lemma~\ref{lem:zigzag} applies to every lift of $L_0$. The complex of the circle in
$\Pst$ is the quotient by $\langle T\rangle$: its generators are the $T$-orbits of crossing points, represented exactly once on the
four lifts $S^{k_0}_\pm$ with $k_0\in\{0,1\}$, and its bigons are the $T$-orbits of the domains of Lemma~\ref{lem:zigzag}.
Each chain has homology $\F_2$ in the degree of its up-crossings. By Proposition~\ref{prop:degrees} these degrees are
$2k_0$ and $2k_0-1$ up to a common shift, that is, $0$, $2$, $3$, $1$. Finally $4+4N_q=\sum(2d+1)$ over the four chains,
so $\sum d=2N_q$ is the rank of the differential.
\end{proof}

\subsection{Proof of Theorem~\ref{thm:hhk}}\label{ss:proof-hhk}

\begin{theorem}\label{thm:pillowP}
Let $q\ge7$ be odd and $a>0$. For the decomposition $K_q=(Y,T)\cup_{t_e^{-(q-5)/2}}(D,U)$, every sufficiently small
$\varepsilon\in\mathcal C_a\setminus\{0\}$, the threshold depending on $q$, and every sufficiently small earring parameter:
\begin{enumerate}
\item $L_1$ is a restricted immersed $1$-manifold, $(L_0,L_1)$ is admissible, and $C(L_0,L_1)$ has $q+2+4N_q$ generators;
\item the differential has rank $2N_q$ and is described by Proposition~\ref{prop:bigonsP}; it is nonzero for $q\ge9$, and no
bigon involves $r_+$;
\item with $\gr(r_+)=\sigma(K_q)$, $\Hnat(Y,T,\pi)$ has rank $q+2$ and graded dimensions
\[
\bigl(1+N_3,\ N_1,\ N_1,\ N_3\bigr)\quad\text{in degrees }\gr-\sigma(K_q)=0,1,2,3,\qquad N_1=\bigl\lceil\tfrac{q+1}4\bigr\rceil,\ N_3=\bigl\lfloor\tfrac{q+1}4\bigr\rfloor .
\]
\end{enumerate}
This holds in both resolutions, and for every choice of absolute grading on the remnant circle.
\end{theorem}

\begin{proof}
(1) By Proposition~\ref{prop:shearedcurves}(1) and Lemma~\ref{lem:shear}(c), the arc is a proper immersion with embedded
lift, so it has no fishtail and is a restricted immersed arc; the remnant circle is restricted by Proposition~\ref{prop:shearedcurves}. Transversality is
Proposition~\ref{prop:gens}. Admissibility follows as in the proof of~\cite[Lemma~6.6]{WuebbenV} (see (P9)): loops on $L_0$ have classes in the span of
$e_{(0,0)}$ and $e_{(\pi,\pi)}$ in $H_1(\Pst)$, while a loop running $j$ times around the remnant circle, which is homotopic in the
strip to twice the core, has class $\pm2j(e_{(0,0)}+e_{(0,\pi)})$, which lies in that span only if $j=0$. The generator count is
Proposition~\ref{prop:gens}.

(2) This is Proposition~\ref{prop:bigonsP}, with $N_q\ge1$ for $q\ge9$.

(3) In the joined case, sum Proposition~\ref{prop:bigonsP}(1) over $k=1,\dots,k_{\max}$ with $k_{\max}=\frac{q+1}2$, and add $r_+$ in degree
$0$. The lift $S^k_+$ contributes degree $2k$ and $S^k_-$ degree $2k-1$; modulo $4$, odd $k$ gives degrees $2$ and $1$, and
even $k$ gives $0$ and $3$. So the homology is
$\bigl(1+\lfloor k_{\max}/2\rfloor,\lceil k_{\max}/2\rceil,\lceil k_{\max}/2\rceil,\lfloor k_{\max}/2\rfloor\bigr)$, and
$\lceil k_{\max}/2\rceil=N_1$, $\lfloor k_{\max}/2\rfloor=N_3$. In the split case the arc crosses $\mathcal D_k$ once for
$k=1,\dots,k_{\max}-2$, which gives
$\bigl(\lfloor k_{\max}/2\rfloor,\lceil k_{\max}/2\rceil-1,\lceil k_{\max}/2\rceil-1,\lfloor k_{\max}/2\rfloor-1\bigr)$, and the circle adds $(1,1,1,1)$ by
Proposition~\ref{prop:bigonsP}(2). The sum is the same vector.
\end{proof}

\begin{proof}[Proof of Theorem~\ref{thm:hhk}]
By Theorem~\ref{thm:topology}, $K_q=K^{(-m)}$ with $m=(q-5)/2$, so $K_q$ is the union of the tangle of~\cite[\S10]{HHK2} for
$T(3,5)$ with $(r,s)=(2,-1)$ and the trivial tangle, glued by $t_e^{-m}$; this is a decomposition as
in~\cite[(17)]{HHK2}, and by Lemma~\ref{lem:shear} its restriction map is $S_{-m}\circ L_1$. Theorem~\ref{thm:pillowP} gives
restrictedness, admissibility, the generators and the differential, and the graded dimensions of $\Hnat(Y,T,\pi)$, with
$\gr(r_+)=\sigma(K_q)=-(q+1)$ (Proposition~\ref{prop:sig}). By Proposition~\ref{prop:inat-graded}, $\Inat(K_q;\F_2)$ has the same
graded dimensions in Kronheimer and Mrowka's grading, so $\Hnat(Y,T,\pi)\cong\Inat(K_q;\F_2)$ as $\Z/4$-graded vector spaces.
This holds for every sufficiently small perturbation in the cone $\mathcal C_a$, in either of its kinds of sector
(Proposition~\ref{prop:restype}), with zero bounding cochain; in particular Conjecture~\ref{conj:hhk65} holds for $K_q$.
\end{proof}

For $q=3$ and $q=5$ the knots are $T(3,4)$ and $T(3,5)$, and the corresponding statement over $\Q$ is~\cite[Thm.~1.1,
Cor.~1.2]{WuebbenV}. The knots $K_q$ with $q\ge7$ are hyperbolic and not two-bridge (Proposition~\ref{prop:hyperbolic}).

\begin{remark}\label{rem:FKPcount}
At $\varepsilon=0$ the intersection points of $S_{-m}(\Pi(W))$ with $\Delta$ are the restrictions of the traceless characters of $K_q$.
The two halves $H$ and $H'$ of $C_0$ have the same image, their lifts differing by $(0,4\pi)$, and the restriction map is
$2$-to-$1$ over the points of $\Delta$ in that image. So the irreducible traceless characters correspond to the crossings of
the lifted sheared polygon of Corollary~\ref{cor:polygon}, rather than of its image, with the lines $\mathcal D_k$, and the
generator count of Proposition~\ref{prop:gens} is one more than twice their number. It can be written
$q+2+4N_q=q-6+4\bigl(\lfloor\frac{5q+6}{12}\rfloor-\lfloor\frac{q+6}{12}\rfloor\bigr)$, which is the count of~\cite[\S5, p.~29]{FKP}
read consistently: the rank formula printed there gives $13$ at $q=7$ (and the printed summands give $5$), whereas the
explicit count on p.~28 is $9$. For $q\ge9$ the traceless character variety therefore has more points than are needed to generate $\Inat(K_q)$, and
the complex is not minimal. This agrees with the remark of
Daemi and Scaduto that a direct computation from the character variety shows that $P(-2,3,7)$ is $I$-basic, while
such a computation does not work for $P(-2,3,q)$ with $q>7$~\cite[p.~161]{DS24}.
\end{remark}

\begin{remark}\label{rem:scomplexP}
Put $\chi(x_P^-)=x_P^+$ and $\chi(x_P^+)=\chi(r_+)=0$. By Proposition~\ref{prop:bigonsP} every bigon pairs $x^\pm_{P_2}$ with
$x^\pm_{P_1}$ and $x^\pm_{P_3}$, for each sign, so the differential commutes with $\chi$. The pillowcase complex is thus an
S-complex over $\F_2$, in the sense of~\eqref{eq:scomplex}, with $\delta_1=\delta_2=0$ and $v=0$, and its irreducible part has a
differential of rank $N_q$. The homology of the span of the generators $x_P^-$ has one class for each line $\mathcal D_k$, $1\le k\le\frac{q+1}2$, in
degree $2k-1$ (in the split case, with the remnant circle graded so that its generators $x^-$ have odd degree, as
in~\cite[Def.~6.13]{WuebbenV}): it has rank $N_1$ in degree $1$ and $N_3$ in degree $3$, the ranks of the irreducible instanton homology $I(K_q)$
of~\cite[Prop.~8.30]{DS24}. A chain-level comparison with $\tC(K_q;\F_2)$ is open. For torus knots the
corresponding conjecture is formulated in~\cite[Conj.~1.5]{WuebbenIV} and proved for $T(3,n)$
in~\cite[Thm.~1.3]{WuebbenV}; there the irreducible part has zero differential and the pillowcase differential is a
component of type $\delta_1$, whereas for $K_q$ with $q\ge9$ it lies in the irreducible part.
\end{remark}

\begin{remark}\label{rem:numericsP}
The proofs above are by hand. Numerical computations of the perturbed curves, of their Floer complexes in both kinds
of sector, and of the resolution type agree with Proposition~\ref{prop:restype}, Proposition~\ref{prop:degrees} and
Theorem~\ref{thm:pillowP}, apart from canceling pairs of generators that are missing at the larger perturbations, as in
Remark~\ref{rem:threshold}; they are described in Appendix~\ref{app:data}, item~(v), and are not used in the proofs.
\end{remark}

\section{Twisted torus knots}\label{sec:twisted}

For $n\equiv1,2\pmod 6$ the traceless character variety of the tangle of \cite[\S10]{HHK2} for $T(3,n)$ is,
before and after a small holonomy perturbation, a disjoint union of an arc and $2\lfloor n/6\rfloor$ circles,
and the image of the unperturbed arc in the pillowcase is a straight segment
\cite[Props.~3.6, 3.9 and~5.7]{WuebbenV}. There is no central circle,
so the argument of Section~\ref{sec:sheared} goes through without an analogue of
Lemma~\ref{lem:straight}. Regluing the trivial tangle by powers of the twist $t_e$ produces the twisted
torus knots $T(3,n;2,m)$, and Theorem~\ref{thm:H} below computes their pillowcase homology, identifies the knots under a hypothesis on the chirality of the torus-knot decomposition, and compares the result with their instanton homology. On the pillowcase side, the one input
beyond \cite{WuebbenV} and Section~\ref{sec:sheared} is Hedden, Herald and Kirk's computation of the Floer
homology of a circle that misses the edges \cite[Thm.~5.4]{HHK2}, which we use for the closed components.

\subsection{Statement}\label{ss:twisted-statement}
For $n\ge1$ coprime to $3$ and $m\ge0$, the \emph{twisted torus knot} $T(3,n;2,m)$ is the closure of the
positive braid
\[
\beta_{n,m}=(\sigma_1\sigma_2)^n\sigma_1^{2m}\in B_3 ,
\]
that is, $T(3,n)$ with $m$ positive full twists added on two adjacent strands; $T(3,n;2,0)=T(3,n)$. Clay
and Watson write $T^m_{3,n}$ for this knot when $n\equiv2\pmod3$ \cite[\S3.2]{ClayWatson}, and Daemi and Scaduto write
$T(3,n;2,m)$ \cite[\S8.5]{DS24}. The knots of Section~\ref{sec:sheared} are the case $n=5$:
$T(3,5;2,m)=P(-2,3,5+2m)$. Throughout we put
\[
\nu=\lfloor n/3\rfloor,\qquad M=\begin{cases}m,&n\equiv1\pmod3,\\ m+1,&n\equiv2\pmod3,\end{cases}
\]
and for $n\equiv1,2\pmod6$ we write $n=6j+1$ or $n=6j+2$, so that $\nu=2j$.

For $n\ge4$ coprime to $3$ let $(Y_n,T_n)$ be the tangle of \cite[\S10]{HHK2} for $T(3,n)$ with the
parameters $(r,s)=((2n+1)/3,-2)$ if $n\equiv1$ and $(r,s)=((n+1)/3,-1)$ if $n\equiv2\pmod3$, as
in \cite{WuebbenV}, and let $K^{(h)}(3,n)=(Y_n,T_n)\cup_{t_e^h}(D,U)$ be the regluing of
Section~\ref{sec:sheared}; for $n=5$ this is the knot $K^{(h)}$ of Section~\ref{sec:sheared}. Lemma~\ref{lem:group} and its proof apply to every $n$, with the boundary words
$a,b,c$ of \cite[(3.2)]{WuebbenV}: $\pi_1(S^3\setminus K^{(h)}(3,n))=\langle A,B\mid t_e^h(b)=c\rangle$ with
$t_e^h(b)=(ba)^hb(ba)^{-h}$, and $K^{(h-1)}(3,n)$ differs from $K^{(h)}(3,n)$ by one full twist of two strands,
hence by one crossing change. The knot group does not determine whether $K^{(0)}(3,n)$ is the positive torus knot $T(3,n)$ or its
mirror image. For $n=5$ we verified that it is $T(3,5)$ (Remark~\ref{rem:chirality}); for the knots of
Theorem~\ref{thm:H} it is a hypothesis of the theorem, shared with \cite{HHK2} and \cite{WuebbenV}.

We write $\varepsilon=(\varepsilon_A,\varepsilon_B)$ for the parameters of the holonomy perturbation $\pi$
of \cite{WuebbenV} (see (P1) in \S\ref{ss:straight}) and $L_1^{\mathrm{tor}}$ for the resulting curve of $(Y_n,T_n)$. By Lemma~\ref{lem:shear}, the
curve of $(Y_n,T_n)$ in the decomposition of $K^{(-m)}(3,n)$ is $L_1=S_{-m}\circ L_1^{\mathrm{tor}}$, where
$S_{-m}(\gamma,\theta)=(\gamma,\theta+2m\gamma)$. For the trivial side we use the earring curve
$L_0=L_0^{\epsilon,g}$ of \cite[Def.~3.4]{HHK2} with $\epsilon>0$ and $g=\delta\sin$; in \cite{WuebbenV},
$\delta=0$. Recall that $\ell(K)$ denotes the sum of the absolute values of the coefficients of $\Delta_K$.

\begin{theorem}\label{thm:H}
Let $n\ge7$ with $n\equiv1$ or $2\pmod6$, let $m\ge1$, and define $j,\nu,M$ as above. Assume, as in~\cite{HHK2}
and~\cite{WuebbenV}, that the tangle $(Y_n,T_n)$ glued to the trivial tangle, $K^{(0)}(3,n)$, is the positive torus knot
$T(3,n)$. Put $K=K^{(-m)}(3,n)$ and $(M_0,M_1)=(\lfloor M/2\rfloor,\lceil M/2\rceil)$.
\begin{enumerate}
\item[(i)] $K=T(3,n;2,m)$.
\item[(ii)] $\sigma(K)=\sigma(T(3,n))-2m=-(4\nu+2M)$, and $\ell(K)=|\sigma(K)|+1$.
\item[(iii)] For every sufficiently small $\varepsilon\in\R^2\setminus\{0\}$, $L_1$ is a restricted immersed
$1$-manifold consisting of an arc and $2j$ circles, and the circles miss the left and right edges and have
vertical degree $2$. There is $\epsilon_0>0$, depending on $\varepsilon$, such that for $0<\epsilon<\epsilon_0$ and $|\delta|<\epsilon^2/2$
the arc meets $L_0$ transversely in $1+2M$ points and carries no bigon. For each such $\epsilon$ and almost every
such $\delta$ the pair $(L_0,L_1)$ is admissible, and with $\gr(r_+)=\sigma(K)$ the graded dimensions of
$\Hnat(Y_n,T_n,\pi)$ are
\[
\bigl(1+2j+M_0,\ 2j+M_1,\ 2j+M_1,\ 2j+M_0\bigr)\qquad\text{in degrees }\gr-\sigma(K)=0,1,2,3,
\]
whatever absolute gradings are chosen on the circles.
\item[(iv)] $\Inat(K;\Z)$ is free of rank $|\sigma(K)|+1$, and $\Hnat(Y_n,T_n,\pi)\cong\Inat(K;\F_2)$ as
$\Z/4$-graded vector spaces.
\end{enumerate}
In particular, this decomposition of $T(3,n;2,m)$ satisfies the conclusion of \cite[Conj.~6.5]{HHK2}.
\end{theorem}

\begin{remark}\label{rem:twisted-overlap}
By Lemma~\ref{lem:conj} below, $T(3,6j+2;2,m)=T(3,6j+1;2,m+1)$. The two families in Theorem~\ref{thm:H}
therefore consist of the knots $T(3,6j+1;2,M)$ with $M\ge1$, and the vector in (iii) depends only on
$(j,M)$; for $M\ge2$ the knot arises from both tangles $(Y_{6j+1},T_{6j+1})$ and $(Y_{6j+2},T_{6j+2})$. For $M=1$
the knot $T(3,6j+1;2,1)=T(3,6j+2)$ is a torus knot, and Theorem~\ref{thm:H} gives a second decomposition of
it, with the homology found in \cite{WuebbenV}. For $M\ge2$ the knot is neither a torus knot, nor alternating,
nor a pretzel knot $P(-2,3,q)$ (Lemma~\ref{lem:twisted-not-torus}).
\end{remark}

\subsection{Alexander polynomials}\label{ss:twisted-alex}
Let $\psi$ be the reduced Burau representation of $B_3$,
\[
\psi(\sigma_1)=\begin{pmatrix}-t&1\\0&1\end{pmatrix},\qquad
\psi(\sigma_2)=\begin{pmatrix}1&0\\t&-t\end{pmatrix}.
\]
For $\beta\in B_3$ whose closure $\hat\beta$ is a knot,
$\Delta_{\hat\beta}(t)\doteq\frac{1-t}{1-t^3}\det\bigl(I-\psi(\beta)\bigr)$ by Burau's formula~\cite[Thm.~3.11]{Birman}; here $\doteq$ denotes equality up
to a unit $\pm t^k$. Daemi and Scaduto record a formula for the Alexander polynomials of the twisted torus
links $T(p,q;2,k)$ \cite[Prop.~8.32]{DS24}; we need the coefficient norm, which the closed form below displays directly.

\begin{lemma}\label{lem:conj}
For $\nu,m\ge0$, $\sigma_1^{-1}(\sigma_1\sigma_2)^{3\nu+2}\sigma_1^{2m}\sigma_1=(\sigma_1\sigma_2)^{3\nu+1}\sigma_1^{2m+2}$ in
$B_3$. Hence $T(3,3\nu+2;2,m)=T(3,3\nu+1;2,m+1)$ as oriented knots, and for every $n$ coprime to $3$ the knot
$T(3,n;2,m)$ is the closure of $(\sigma_1\sigma_2)^{3\nu+1}\sigma_1^{2M}$.
\end{lemma}

\begin{proof}
The full twist $\Delta^2=(\sigma_1\sigma_2)^3$ is central, and
$(\sigma_1\sigma_2)^2=\sigma_1(\sigma_2\sigma_1\sigma_2)=\sigma_1^2\sigma_2\sigma_1$ by the braid relation. Hence
$(\sigma_1\sigma_2)^{3\nu+2}\sigma_1^{2m}=\Delta^{2\nu}\sigma_1^2\sigma_2\sigma_1^{2m+1}$, and conjugation by $\sigma_1$ gives
$\Delta^{2\nu}\sigma_1\sigma_2\sigma_1^{2m+2}=(\sigma_1\sigma_2)^{3\nu+1}\sigma_1^{2m+2}$. Conjugate braids have isotopic closures.
\end{proof}

For example, $P(-2,3,5+2m)=T(3,5;2,m)=T(3,4;2,m+1)$ and $T(3,6j+1;2,1)=T(3,6j+2)$.

\begin{proposition}\label{prop:twisted-alex}
Let $\nu,M\ge0$ and let $K_{\nu,M}$ be the closure of $(\sigma_1\sigma_2)^{3\nu+1}\sigma_1^{2M}$. Then
\[
\Delta_{K_{\nu,M}}(t)\doteq\underbrace{\sum_{i=0}^{\nu-1}\bigl(t^{3i}-t^{3i+1}\bigr)}_{L(t)}
+\underbrace{\sum_{i=0}^{2M}(-1)^it^{3\nu+i}}_{L_{\mathrm{mid}}(t)}
+\underbrace{\sum_{i=0}^{\nu-1}\bigl(t^{6\nu+2M-3i}-t^{6\nu+2M-3i-1}\bigr)}_{L^*(t)} .
\]
The three sums have disjoint exponent ranges, every coefficient is $0$ or $\pm1$, the nonzero coefficients
alternate in sign, and $\ell(K_{\nu,M})=4\nu+2M+1$.
\end{proposition}

\begin{proof}
Put $P=\psi(\sigma_1\sigma_2)=\begin{pmatrix}0&-t\\t&-t\end{pmatrix}$ and $T=t^{2M}$. Since $\tr P=-t$ and $\det P=t^2$,
the Cayley--Hamilton theorem gives $P^3=t^3I$, and induction on $M$ gives
$\psi(\sigma_1)^{2M}=\begin{pmatrix}T&(1-T)/(1+t)\\0&1\end{pmatrix}$. The matrix $X=t^{3\nu}P\psi(\sigma_1)^{2M}$ has
$\tr X=-t^{3\nu+1}(t+T)/(1+t)$ and $\det X=t^{6\nu+2}T$, so $\det(I-X)=1-\tr X+\det X$ is
\begin{equation}\label{eq:twisted-F}
F_{\nu,M}(t):=\det\bigl(I-\psi((\sigma_1\sigma_2)^{3\nu+1}\sigma_1^{2M})\bigr)
=1+t^{3\nu+2}\,\frac{1+t^{2M-1}}{1+t}+t^{6\nu+2M+2}
\end{equation}
(for $M=0$ the middle term is $t^{3\nu+1}$). Since $(1-t)/(1-t^3)=1/(1+t+t^2)$, it suffices to show that
$(1+t+t^2)(L+L_{\mathrm{mid}}+L^*)=F_{\nu,M}$. From $(1+t+t^2)(1-t)=1-t^3$ and $L^*(t)=t^{6\nu+2M}L(t^{-1})$,
\[
(1+t+t^2)L=1-t^{3\nu},\qquad(1+t+t^2)L^*=t^{6\nu+2M+2}-t^{3\nu+2M+2}.
\]
Next, $L_{\mathrm{mid}}=t^{3\nu}(1+t^{2M+1})/(1+t)$ and
$(1+t+t^2)(1+t^{2M+1})-(1+t)(1+t^{2M+2})=t^2+t^{2M+1}$, so
\[
(1+t+t^2)L_{\mathrm{mid}}=t^{3\nu}+t^{3\nu+2M+2}+t^{3\nu+2}\,\frac{1+t^{2M-1}}{1+t}.
\]
The three products add up to $F_{\nu,M}$. The sums $L$, $L_{\mathrm{mid}}$ and $L^*$ occupy the exponents
$0,\dots,3\nu-2$, then $3\nu,\dots,3\nu+2M$, then $3\nu+2M+2,\dots,6\nu+2M$, so no cancellation occurs, and they
contribute $2\nu$, $2M+1$ and $2\nu$ nonzero coefficients. Inside each sum the signs alternate; the last
coefficient of $L$ and the first of $L^*$ are $-1$, and $L_{\mathrm{mid}}$ begins and ends with $+1$.
\end{proof}

The degree $6\nu+2M=2(n+m-1)$ is twice the genus of the positive braid closure. For $T(3,7;2,2)$, where
$\nu=M=2$, the coefficients of $t^0,\dots,t^{16}$ are
\[
1,\,-1,\,0,\,1,\,-1,\,0,\,1,\,-1,\,1,\,-1,\,1,\,0,\,-1,\,1,\,0,\,-1,\,1,
\]
and $\ell=13=|\sigma|+1$ (Proposition~\ref{prop:twisted-sig}).

\subsection{Signature}\label{ss:twisted-sig}

\begin{proposition}\label{prop:twisted-sig}
Let $n\equiv1,2\pmod6$ and $m\ge0$. Then $\det T(3,n;2,m)=2m+1$ for $n\equiv1$ and $2m+3$ for
$n\equiv2\pmod 6$, and
\[
\sigma(T(3,n;2,m))=\sigma(T(3,n))-2m=-(4\nu+2M).
\]
Consequently $\ell(T(3,n;2,m))=|\sigma(T(3,n;2,m))|+1$.
\end{proposition}

\begin{proof}
At $t=-1$, $\psi(\sigma_1)=\begin{pmatrix}1&1\\0&1\end{pmatrix}$ and $\psi(\sigma_2)=\begin{pmatrix}1&0\\-1&1\end{pmatrix}$, so
$Q=\psi(\sigma_1\sigma_2)|_{t=-1}=\begin{pmatrix}0&1\\-1&1\end{pmatrix}$, which has order $6$. Since $1+t+t^2=1$ at $t=-1$,
$\det T(3,n;2,m)=|\det(I-Q^n\psi(\sigma_1)^{2m}|_{t=-1})|$. For $n\equiv1\pmod6$ the matrix
$Q^n\psi(\sigma_1)^{2m}|_{t=-1}$ is $\begin{pmatrix}0&1\\-1&1-2m\end{pmatrix}$, and
$\det\begin{pmatrix}1&-1\\1&2m\end{pmatrix}=2m+1$; for $n\equiv2$ it is $\begin{pmatrix}-1&1-2m\\-1&-2m\end{pmatrix}$, and
$\det\begin{pmatrix}2&2m-1\\1&2m+1\end{pmatrix}=2m+3$.

By Murasugi's theorem \cite[Thm.~5.6]{Murasugi}, if $|\sigma(K)|\equiv2\mu\pmod4$ then
$|\Delta_K(-1)|\equiv(-1)^\mu\pmod4$. As $m$ increases by one, the determinant alternates between $1$ and
$3$ modulo $4$, so the parity of $\sigma/2$ changes and $\sigma(T(3,n;2,m+1))\ne\sigma(T(3,n;2,m))$. Changing
one crossing of $\sigma_1^{2m+2}$ in $\beta_{n,m+1}$ from positive to negative gives $\beta_{n,m}$. For knots
$K_\pm$ related by such a change, $\sigma(K_+)-\sigma(K_-)\in\{0,-2\}$ in our convention (the standard
Seifert-matrix computation; see \cite[p.~163]{DS24}). Hence $\sigma(T(3,n;2,m+1))=\sigma(T(3,n;2,m))-2$.
Finally $|\sigma(T(3,n))|=2(n-1)-4\lfloor n/6\rfloor$ \cite{Litherland} (see \cite[proof of Prop.~6.3]{WuebbenV}),
which is $4\nu$ for $n=6j+1$ and $4\nu+2$ for $n=6j+2$. The last assertion follows from
Proposition~\ref{prop:twisted-alex} and Lemma~\ref{lem:conj}.
\end{proof}

For $n\equiv2\pmod6$ this agrees with the formula $\sigma(T(3,3n'+2;2,k))=-4n'-2-2k$, $k>0$, of
\cite[p.~163]{DS24}. The argument needs the determinant to alternate modulo $4$; for $n=5$ it does not at the
first step, and indeed $\sigma(P(-2,3,5))=\sigma(P(-2,3,7))=-8$.

\begin{lemma}\label{lem:twisted-not-torus}
Let $\nu\ge2$ and $M\ge2$. Then $K_{\nu,M}$ is not a torus knot, not alternating, and not a pretzel knot
$P(-2,3,q)$.
\end{lemma}

\begin{proof}
By Proposition~\ref{prop:twisted-alex}, $\Delta_{K_{\nu,M}}=1-t+t^3-\cdots$, with zero coefficient of $t^2$. As in
the proof of Lemma~\ref{lem:not-torus}, a torus knot with this Alexander polynomial would be $T(3,b)$ with
$2(b-1)=6\nu+2M$. By Lemma~\ref{lem:conj}, $T(3,b)$ is $K_{\nu',0}$ or $K_{\nu',1}$ for some $\nu'$, so by
Proposition~\ref{prop:twisted-alex} its Alexander polynomial has no more than three consecutive nonzero
coefficients, whereas $\Delta_{K_{\nu,M}}$ has $2M+1\ge5$. Next,
$\det K_{\nu,M}=|F_{\nu,M}(-1)|=|2+(-1)^\nu(2M-1)|<4\nu+2M+1=\ell(K_{\nu,M})$, while for an alternating knot the
determinant equals $\ell$ (see the proof of Lemma~\ref{lem:not-torus}). Finally, the coefficient of $t^5$ in
$\Delta_{K_{\nu,M}}$ is $0$, because $5=3\cdot1+2$ is not an exponent of $L$ and $L_{\mathrm{mid}}$ begins at
$3\nu\ge6$, whereas in $\Delta_{P(-2,3,q)}$ it is $1$ for $q\ge5$ (Lemma~\ref{lem:alex}); and $P(-2,3,q)$ is a torus
knot for $q\le5$.
\end{proof}

\subsection{Identification of the knots}\label{ss:twisted-top}
Clay and Watson prove that for $\nu\ge0$ and $m\ge0$ the group of $T(3,3\nu+2;2,m)=T^m_{3,3\nu+2}$ is
\[
\pi^m_{3\nu+2}=\bigl\langle\alpha,\beta\ \big|\ \alpha^2(\beta^{-\nu}\alpha)^m\alpha=\beta^{2\nu+1}(\beta^{-\nu}\alpha)^m\beta^{\nu+1}\bigr\rangle
\]
\cite[\S3.2 and Prop.~22]{ClayWatson}.

\begin{lemma}\label{lem:twisted-group}
\begin{enumerate}
\item[(a)] If $n=3\nu+2\ge5$ and $m\ge0$, then $\alpha\mapsto A$, $\beta\mapsto B^{-1}$ induces an isomorphism
$\pi^m_{3\nu+2}\cong\pi_1(S^3\setminus K^{(-m)}(3,n))$.
\item[(b)] If $n=3\nu+1\ge4$ and $m\ge1$, then $\alpha\mapsto AB^{2\nu+1}$, $\beta\mapsto B$ induces an isomorphism
$\pi^{m-1}_{3\nu+2}\cong\pi_1(S^3\setminus K^{(-m)}(3,n))$.
\end{enumerate}
\end{lemma}

\begin{proof}
Both maps are automorphisms of the free group on two letters. The group of a one-relator presentation is
unchanged, up to isomorphism, if the relator is replaced by its image under an automorphism of the free
group, by a cyclic permutation, or by its inverse. So it suffices to show that each map carries
Clay and Watson's relator $\alpha^2Y^m\alpha\,\beta^{-(\nu+1)}Y^{-m}\beta^{-(2\nu+1)}$, where $Y=\beta^{-\nu}\alpha$ (with
$m-1$ in place of $m$ in (b)), to a cyclic permutation of the relator, or of its inverse, of the presentation
$\langle A,B\mid(ba)^{-m}b(ba)^m=c\rangle$.

(a) Here $a=A^2B^{2\nu+1}$, $b=B^{-(\nu+1)}A^{-1}$ and $c=B^{-(\nu+1)}aB^{\nu+1}$ \cite[(3.2)]{WuebbenV}. With $C=AB^\nu$ we
have $ba=B^{-(\nu+1)}CB^{\nu+1}$, so the relation reads $C^{-m}A^{-1}B^{-(\nu+1)}C^m=A^2B^{2\nu+1}$; a cyclic permutation
of the inverse of the corresponding relator is $R=AC^mA^2B^{2\nu+1}C^{-m}B^{\nu+1}$. Under $\alpha\mapsto A$,
$\beta\mapsto B^{-1}$ we have $Y\mapsto B^\nu A$, $A(B^\nu A)^m=C^mA$ and $(B^\nu A)^{-m}=B^\nu C^{-m}B^{-\nu}$, so Clay and
Watson's relator goes to $AC^mA^2\cdot B^{\nu+1}\cdot B^\nu C^{-m}B^{-\nu}\cdot B^{2\nu+1}=R$.

(b) Here $a=AB^\nu$, $b=B^{-(2\nu+1)}A^{-2}$ and $c=B^{-(2\nu+1)}AB^{3\nu+1}$ \cite[(3.2)]{WuebbenV}. With $\Xi=B^{\nu+1}A$ we
have $ba=B^{-(2\nu+1)}\Xi^{-1}B^{2\nu+1}$, so the relation reads $\Xi^mA^{-2}B^{-(2\nu+1)}\Xi^{-m}=AB^\nu$, whose relator has
inverse $R_H=AB^\nu \Xi^mB^{2\nu+1}A^2\Xi^{-m}$. Under $\alpha\mapsto AB^{2\nu+1}$, $\beta\mapsto B$ we have
$Y\mapsto B^{-\nu}AB^{2\nu+1}=B^{-(2\nu+1)}\Xi B^{2\nu+1}$. Using $B^{2\nu+1}A=B^\nu \Xi$, the left side
$\alpha^2Y^{m-1}\alpha$ goes to $AB^\nu \Xi^mB^\nu \Xi B^{2\nu+1}$ and the right side $\beta^{2\nu+1}Y^{m-1}\beta^{\nu+1}$ to
$\Xi^{m-1}B^{3\nu+2}$. The relator therefore goes to $AB^\nu \Xi^mB^\nu \Xi B^{-(\nu+1)}\Xi^{1-m}$, which equals $R_H$ because
$B^\nu \Xi=B^{2\nu+1}A$ and $AB^{-(\nu+1)}\Xi=A^2$.
\end{proof}

At $m=0$ in (a), and at $m=1$ in (b), Clay and Watson's relation is $\alpha^3=\beta^{3\nu+2}$, the group of the torus
knot $T(3,3\nu+2)$; part (a) at $\nu=1$ is Lemma~\ref{lem:relator}.

\begin{proof}[Proof of Theorem~\ref{thm:H}\,(i) and (ii)]
Let $K=K^{(-m)}(3,n)$. For $n=3\nu+2$, Lemma~\ref{lem:twisted-group}(a) identifies $\pi_1(S^3\setminus K)$ with the
group of $T(3,n;2,m)$. For $n=3\nu+1$, part (b) identifies it with the group of $T(3,3\nu+2;2,m-1)$, which is
$T(3,n;2,m)$ by Lemma~\ref{lem:conj}. Both groups are generated by two elements, and both knots are
nontrivial by Proposition~\ref{prop:twisted-alex}, so both knots are prime by Norwood's
theorem \cite{Norwood}. Prime knots with isomorphic groups are equivalent \cite[Cor.~2.1]{GordonLuecke},
by a homeomorphism of $S^3$ that may reverse orientation. Hence $K$ is $T(3,n;2,m)$ or its mirror image.

The chirality follows by induction on $m$, as in the proof of Theorem~\ref{thm:topology}. By hypothesis
$K^{(0)}(3,n)=T(3,n)$. Suppose that $\sigma(K^{(-m)}(3,n))=\sigma(T(3,n))-2m$ for some $m\ge0$. The knot $K^{(-m-1)}(3,n)$
differs from $K^{(-m)}(3,n)$ by one crossing change, so their signatures differ by at most
$2$ \cite[(10.2)]{Murasugi}. By Proposition~\ref{prop:twisted-sig}, $T(3,n;2,m+1)$ and its mirror image have
signatures $\mp s$ with $s=|\sigma(T(3,n))|+2m+2$, and only $-s$ lies within $2$ of $\sigma(T(3,n))-2m=-(s-2)$,
since $s+(s-2)>2$. Hence $K^{(-m-1)}(3,n)=T(3,n;2,m+1)$. This proves (i), and (ii) is
Proposition~\ref{prop:twisted-sig}.
\end{proof}

\subsection{The pillowcase homology}\label{ss:twisted-pillow}

\begin{proof}[Proof of Theorem~\ref{thm:H}\,(iii)]
Let $\omega=2$ for $n\equiv1$ and $\omega=4$ for $n\equiv2\pmod6$, so that $\omega+2m-1=2M+1$.

\emph{The arc.} By \cite[Prop.~5.7(3) and Lemma~6.1]{WuebbenV}, for small $\varepsilon\neq0$ the lift $\tilde A^{\mathrm{tor}}$
of the arc of $L_1^{\mathrm{tor}}$ starting at $(0,0)$ is, away from its ends, $C^1$-close to the segment from $(0,0)$ to
$(\pi,\omega\pi)$, along which $\theta=\omega\gamma$ \cite[Prop.~3.9]{WuebbenV}. Near $(0,0)$ it has the form
$y\,\mathbf c(\varepsilon)+y^3R(y,\varepsilon)$, where $\mathbf c=(c_1,c_2)$ and $R$ are real-analytic and $\mathbf c(0)$ is a positive multiple of $(1,\omega)$
\cite[Lemma~5.1]{WuebbenV}, and the same holds at its other end. The shear $S_{-m}$ is linear. Hence the lift
$\tilde A=S_{-m}\tilde A^{\mathrm{tor}}$ is $C^1$-close, away from its ends, to the segment from $(0,0)$ to
$(\pi,(\omega+2m)\pi)$, on which $\varphi=(2M+1)\gamma$; near $(0,0)$ it equals
$y\,S_{-m}\mathbf c(\varepsilon)+y^3S_{-m}R(y,\varepsilon)$, and the $\varphi$-component of $S_{-m}\mathbf c(0)$ is $(2M+1)\,c_1(0)>0$. So, as
in \cite[Cor.~5.3]{WuebbenV}, $\varphi$ is strictly monotone along $\tilde A$ with $|d\varphi|$ bounded below in
arclength, and the limiting slopes $\omega+2m+O(|\varepsilon|)$ differ from $1$. The interior of $\tilde A$ lies in the open
strip $0<\gamma<\pi$, and $\tilde A$ ends on the line $\varphi=(2M+1)\pi$. It therefore crosses exactly the lines
$\mathcal D_1,\dots,\mathcal D_M$, each once and transversally, and meets $\mathcal D_0$ only at its initial point. Since $\tilde A$ is
embedded, the arc has no fishtail and is a restricted immersed arc.

\emph{The circles.} By \cite[Prop.~5.7(3)]{WuebbenV}, the other components of $L_1^{\mathrm{tor}}$ are circles
$C^1$-close to the $2\lfloor\nu/2\rfloor=2j$ non-central circles \cite[Prop.~3.6]{WuebbenV}. They are restricted,
miss the left and right edges, and have vertical degree $2$ \cite[Cor.~4.8 and Lemma~6.6]{WuebbenV}. By
Lemma~\ref{lem:shear} their images under $S_{-m}$ have the same three properties. Along these images
$\varphi$ need not be monotone: $S_{-m}$ adds $2m\gamma$ to $\varphi$, and $\gamma$ is not monotone along a closed curve.

\emph{The earring curve.} The lifts $S^k_\pm$ of $L_0=L_0^{\epsilon,g}$ are the translates of $\pm\tilde L_0^{\epsilon,g}$
by $(0,2\pi k)$, where $\tilde L_0^{\epsilon,g}$ is given by \cite[(13)]{HHK2}. At the point with parameter $t$ the slope
of either is $(1-\epsilon\cos t)/(1+\epsilon\cos t)+\delta\cos\gamma$, and, for small $\epsilon$, one has
$\cos t>\epsilon/2$ on $S^k_+$ over the front strip $0<\gamma<\pi$, and $\cos t<-\epsilon/2$ on $S^k_-$ over
$2\epsilon\le\gamma\le\pi-2\epsilon$. Hence for $|\delta|<\epsilon^2/2$ the lifts $S^k_+$ have slope less than $1$ over
the front strip and the lifts $S^k_-$ slope greater than $1$ over $2\epsilon\le\gamma\le\pi-2\epsilon$, as for
$\delta=0$ \cite[\S2.2]{WuebbenV}, and they lie within $\sqrt2\,\epsilon+|\delta|$ of $\mathcal D_k$. The proofs of
\cite[Cor.~5.3(4), Prop.~6.3 and Lemma~6.8]{WuebbenV} use only these properties of the lifts of $L_0$, the slope
of $S^k_-$ only at points at distance much larger than $\epsilon$ from the edges, and the
properties of the arc established above. Therefore, for small $\epsilon$ and $|\delta|<\epsilon^2/2$, the arc meets
$L_0$ transversely in exactly $1+2M$ points: the generator $r_+$ of \cite[Lemma~5.1]{HHK2}, which lies on $S^0_+$
because $\tilde A$ leaves $(0,0)$ above $\mathcal D_0$, and a pair $x_k^\pm\in S^k_\pm$ near the crossing of $\tilde A$ with
$\mathcal D_k$ for $k=1,\dots,M$.

\emph{Transversality on the circles.} Fix $\epsilon$ and let $C$ be a circle of $L_1$. Since
$L_0^{\epsilon,\delta\sin}=c_{\delta\sin}\circ L_0^{\epsilon,0}$, where $c_g(\gamma,\theta)=(\gamma,\theta+g(\gamma))$ \cite[(11), p.~18]{HHK2},
$L_0^{\epsilon,\delta\sin}$ is transverse to $C$ if and only if $L_0^{\epsilon,0}$ is transverse to $c_{-\delta\sin}(C)$.
Consider $\Phi(t,x,\delta)=\tilde L_0(t)-c_{-\delta\sin}(\tilde C(x))\in\R^2$, where $\tilde L_0$ is a lift of
$L_0^{\epsilon,0}$ and $\tilde C$ a lift of $C$. Its derivative in $t$ is the tangent vector of $\tilde L_0$, which is
never vertical, and its derivative in $\delta$ is $(0,\sin\gamma(\tilde C(x)))$, which is vertical and nonzero because
$C$ misses the edges. So $\Phi$ is a submersion, $\Phi^{-1}(0)$ is a $1$-manifold, and $\delta$ is a regular value of
the projection $\Phi^{-1}(0)\to\R$ exactly when $0$ is a regular value of $\Phi(\cdot,\cdot,\delta)$. By Sard's theorem,
applied to the countably many pairs of lifts, $L_0^{\epsilon,\delta\sin}$ is transverse to every circle of $L_1$ for
almost every $\delta$. For such $\delta$ the pair $(L_0,L_1)$ is admissible: $L_0$ is a circle, the intersections
are transverse, and the argument of \cite[Lemma~6.6]{WuebbenV} applies, since $L_0$ lies near $\Delta$ and each
circle of $L_1$ lies in the open strip with vertical degree $2$.

\emph{Bigons and gradings on the arc.} By \cite[Lemma~6.4]{WuebbenV} ((P9) in \S\ref{ss:perturbed-P}), the boundary of a
bigon from $p$ to $p'$ lifts to a closed curve made of a sub-path of one lift of a component of $L_1$ and an arc of one
lift of $L_0$, which contains lifts $\tilde p$, $\tilde p'$ of both corners. Every lift of the arc is a translate of $\tilde A$
under the deck group $G$, which permutes the lifts of $L_0$, and $\tilde A$ meets each of them at most once. So a bigon
between generators on the arc would have $\tilde p=\tilde p'$, hence $p=p'$, which is impossible since $\gr(p,p')=1$
\cite[Prop.~3.8]{HHK2}. Thus the arc carries no bigon. For the degrees, $\deg=\gr-\sigma(K)$ satisfies $\deg r_+=0$
by \cite[(19)]{HHK2} and $\deg x_k^-=\deg x_k^+-1$ \cite[\S5.1]{HHK2}. Since $\varphi$ has no critical point on $\tilde A$,
\cite[Lemma~6.8]{WuebbenV} applies to $r_+\in S^0_+$ and $x_1^+\in S^1_+$, and to $x_{k-1}^+$ and $x_k^+$; each step
changes the degree by $2$. Hence $\deg x_k^+=2k$ and $\deg x_k^-=2k-1$ modulo $4$, and the homology of the arc,
which equals its complex, has dimensions $(1+M_0,M_1,M_1,M_0)$, since $M_0$ and $M_1$ count the even and the odd
$k\in\{1,\dots,M\}$.

\emph{The total.} By \cite[Thm.~5.4]{HHK2}, each circle contributes $(1,1,1,1)$, whatever its absolute grading; this
applies to $L_0^{\epsilon,\delta\sin}=c_{\delta\sin}\circ L_0^{\epsilon,0}$ because the pair $(L_0^{\epsilon,\delta\sin},C)$ is carried by
$c_{-\delta\sin}$ to $(L_0^{\epsilon,0},c_{-\delta\sin}(C))$, and $c_{-\delta\sin}(C)$ again misses the edges and has vertical degree~$2$.
There is no differential between different components \cite[p.~37]{HHK2}, and $\Hnat(Y_n,T_n,\pi)$ is the direct
sum of the contributions of the components \cite[p.~34]{HHK2}. Adding $2j$ copies of $(1,1,1,1)$ to
$(1+M_0,M_1,M_1,M_0)$ gives the stated vector.
\end{proof}

\begin{remark}\label{rem:twisted-numerics}
We have computed the sheared curves and their complexes numerically, by the method of
\cite[Rem.~6.12]{WuebbenV} with $\delta=0$, for $n=7,8,13$, $m=1,2,3$, and the two perturbations
$(n\varepsilon_A,\varepsilon_B,\epsilon)=(0.3,0.05,0.01)$ and $(-0.3,-0.03,0.02)$. In all $18$ cases the arc carries $1+2M$
generators and no bigon, the homology of each circle, computed from its complex rather than from
\cite[Thm.~5.4]{HHK2}, is $(1,1,1,1)$, and the total is the vector of Theorem~\ref{thm:H}(iii). These
computations are not used in the proofs.
\end{remark}

\subsection{Comparison with instanton homology}\label{ss:twisted-inat}

\begin{proposition}\label{prop:twisted-inat}
Let $n\ge7$ with $n\equiv1,2\pmod6$, let $m\ge0$, and let $K=T(3,n;2,m)$. Then $\Inat(K;\Z)$ is free of rank
$|\sigma(K)|+1$, and in Kronheimer and Mrowka's absolute grading its graded ranks are
\[
\bigl(1+\lfloor|\sigma(K)|/4\rfloor,\ \lceil|\sigma(K)|/4\rceil,\ \lceil|\sigma(K)|/4\rceil,\ \lfloor|\sigma(K)|/4\rfloor\bigr)
\qquad\text{in degrees }\gr-\sigma(K)=0,1,2,3 .
\]
The same holds with coefficients in $\Q$ and in $\F_2$.
\end{proposition}

\begin{proof}
Since $n=3(2j+1)-2$ or $n=3\cdot2j+2$, and $\det K\ne0$ by Proposition~\ref{prop:twisted-sig}, $K$ is
$I$-basic by \cite[Prop.~8.31]{DS24}, applied with $p=3$ and with $m$ full twists: its irreducible instanton homology $I(K)$ has rank $|\sigma(K)|/2$. The
proof of \cite[Prop.~8.31]{DS24} applies \cite[Lemma~8.23]{DS24} starting from $T(3,n)$, whose irreducible
complex has rank $|\sigma(T(3,n))|/2$ equal to the absolute value of its Euler characteristic
\cite[Ex.~8.11]{DS24}, so that $I(T(3,n))$ is free abelian; by \cite[Rem.~8.27]{DS24}, $I(K)$ is free abelian too.
By \cite[Prop.~8.28]{DS24}, $I(K)\cong\Z_{(1)}^{\lceil|\sigma|/4\rceil}\oplus\Z_{(3)}^{\lfloor|\sigma|/4\rfloor}$ in the grading
of~\eqref{eq:scomplex}. For $n\equiv2\pmod6$ and $m\ge1$ this is \cite[Cor.~8.36]{DS24}, with their $n$ equal
to $2j$ and their $k$ equal to $m$.

The rest is the proof of Proposition~\ref{prop:inat-graded}. The initial page of the spectral
sequence~\eqref{eq:ds-ss} is free, with ranks $1+\lfloor|\sigma|/4\rfloor$, $\lceil|\sigma|/4\rceil$,
$\lceil|\sigma|/4\rceil$ and $\lfloor|\sigma|/4\rfloor$ in degrees $0$, $1$, $2$ and $3$, for a total of $1+|\sigma(K)|$. By
\eqref{eq:ell} and Proposition~\ref{prop:twisted-sig}, $\dim_\Q\Inat(K;\Q)\ge\ell(K)=1+|\sigma(K)|$. Hence every
differential vanishes rationally, hence integrally, and $\Inat(K;\Z)$ is free with these graded ranks in the
grading of~\eqref{eq:scomplex}. Kronheimer and Mrowka's grading is this grading shifted by
$\sigma(K)$ \cite[Thm.~8.9]{DS19}. The assertions over $\Q$ and $\F_2$ follow by base change and the universal
coefficient theorem.
\end{proof}

\begin{proof}[Proof of Theorem~\ref{thm:H}\,(iv)]
By (i), (ii) and Proposition~\ref{prop:twisted-inat}, $\Inat(K;\Z)$ is free of rank $|\sigma(K)|+1$, with
$|\sigma(K)|/4=2j+M/2$. Hence $\lfloor|\sigma(K)|/4\rfloor=2j+M_0$ and $\lceil|\sigma(K)|/4\rceil=2j+M_1$, and the graded
dimensions of $\Inat(K;\F_2)$ are the vector of (iii).
\end{proof}

The restriction to $n\equiv1,2\pmod6$ comes from the pillowcase side. For $n\equiv4,5\pmod6$ the curve
of \cite{WuebbenV} contains the central circle, and the monotonicity of $\varphi$ along its halves after the
shear requires control of $\gamma$ along them; Section~\ref{sec:sheared} obtains it for $n=5$ from the
straightness of the image of the central circle (Lemma~\ref{lem:straight}). We have not treated $n=4$, whose twisted knots are again the pretzel knots by
Lemma~\ref{lem:conj}, nor $n\equiv4,5\pmod6$ with $n\ge10$.

\section{Discussion}\label{sec:disc}

\subsection{Relation to earlier work}\label{ss:disc-prior}
Hedden, Herald and Kirk verify their Conjecture~6.5 (Conjecture~\ref{conj:hhk65}) for two-bridge knots \cite[\S7]{HHK2} and for a
list of torus knots \cite[\S11]{HHK2}; the torus knots $T(3,n)$ are treated in \cite{WuebbenV}.  For
a decomposition of their own, Fukumoto, Kirk and Pinz\'on-Caicedo describe the unperturbed traceless
character variety of the pretzel knots and count nine generators for $P(-2,3,7)$ \cite[\S5]{FKP}; they do not perturb or compute a Floer complex.  Theorems~\ref{thm:hhk}
and~\ref{thm:twisted} supply the perturbed curves, their Floer complexes and the gradings.  For
$q\ge9$ the differential in Theorem~\ref{thm:hhk} is nonzero, so the agreement with $\Inat$ is not a
count of generators.  The twists act on the pillowcase by the shear of \cite[(28)--(30)]{HHK2}.

For Smith's decomposition of $K_q$ as a sum of rational tangles, Smith proves that for $q=5$ the pairing with
zero bounding cochains has rank $5$ rather than $\dim\Inat=7$~\cite[proof of Thm.~5.9]{Smith}, so that under the
conjecture of Cazassus, Herald, Kirk and Kotelskiy a nonzero bounding cochain is required
there~\cite[Thm.~5.9]{Smith}. The decomposition used here needs none. Maurer--Cartan elements for Smith's decomposition at
$q=5,7$, and their relation to the Bar-Natan invariant of Kotelskiy, Watson and Zibrowius~\cite{KWZ} of the tangle formed
by the twist regions with $3$ and $q$ half-twists, are the subject of~\cite{WuebbenIIb}.

\subsection{Open problems}\label{ss:disc-open}
\emph{Chain level.} The pillowcase complex of Theorem~\ref{thm:hhk} is an S-complex over $\F_2$ whose only nonzero
component is the differential of its irreducible part (Remark~\ref{rem:scomplexP}). Whether it is S-chain homotopy
equivalent to $\tC(K_q;\F_2)$ is open. For the torus knots $T(3,n)$ the corresponding statement
holds~\cite[Thm.~1.3]{WuebbenV}, and there the pillowcase differential is a component of type $\delta_1$; for $K_q$
with $q\ge9$ it lies in the irreducible part instead.

\emph{Other families.}  Theorem~\ref{thm:twisted} covers $T(3,n;2,m)$ for $n\equiv1,2\pmod6$.  For
$n\equiv4,5\pmod6$ and $n\ge10$ the curve of \cite{WuebbenV} has a central circle, and the shear
argument of Section~\ref{sec:sheared} needs that circle to be a straight line, which we prove only for
$n=5$; the signature argument also fails for $n\equiv5\pmod6$.

\appendix
\section{Computations}\label{app:data}

\subsection{The sheared curves}\label{ss:methods}
Apart from the base case of Remark~\ref{rem:chirality}, a verified computation described at the end of this
subsection, the proofs in Sections~\ref{sec:sheared} and~\ref{sec:twisted} do not use computation.  The following
computations corroborate them.
\begin{enumerate}
\item[(i)] \emph{Lemma~\ref{lem:straight}.}  Each displayed identity in Steps 1--7 of its proof holds
 identically in the indeterminates $v$ and $\zeta$, as verified by symbolic computation.  Without the
 assumption $E=0$ one has $\mathbf{Vec}\,\rho(w_0)\cdot\mathbf{Vec}\,\rho(a)=\frac12\zeta\sin2v\,E^2$ and
 $\operatorname{tr}\bigl(\rho(w_0)\rho(a)\bigr)=-2\zeta\sin2v\,E^2$.  In the Fricke coordinates
 $x=\operatorname{tr}\rho(A)$, $y=\operatorname{tr}\rho(B)$, $z=\operatorname{tr}\rho(AB)$, the condition
 $\operatorname{tr}\rho(b)=0$ reads $x=yz$; then $E=y^2+z^2-3$,
 \[
  \operatorname{tr}\rho(a)=-yE,\qquad \operatorname{tr}\rho(w_0)=-zE,\qquad \operatorname{tr}\rho(w_0a)=yzE^2,
 \]
 and $\operatorname{tr}\rho(w)-2$ lies in the ideal generated by $x-yz$ and $E$.  For the word
 $(ba)^{-2}c\,(ba)^2a$ the same reduction leaves $-y^2-1$, so that word is not trivial on $C_0$.  At three
 points of $C_0$ with algebraic coordinates, exact quaternion arithmetic gives $\rho(w)=1$, and a
 numerical trace of $\Pi(C_0)$ at $50$ significant digits satisfies $|\theta-6\gamma+\pi|\le1.4\cdot10^{-47}$.
\item[(ii)] \emph{Combinatorics.}  For every odd $q\le401$ we verified in exact rational arithmetic: the
 crossings of the sheared polygons of Corollary~\ref{cor:polygon} with the lifts of $L_0$, for both
 resolutions; the value of $N_q$; the order of the crossing abscissae in $(\frac\pi6,\frac{5\pi}6)$ and
 the absence of a crossing at $\gamma=\frac\pi2$; the bookkeeping of the graded homology in
 Theorem~\ref{thm:pillowP}; the count of Remark~\ref{rem:FKPcount}; and the number of generators on the
 remnant circle.
\item[(iii)] \emph{Groups.}  The identity of Lemma~\ref{lem:relator} holds as an identity of reduced free-group
 words for $n\in\{5,8,\dots,29\}$ and $m\le15$; with $t_e^{+m}$ in place of $t_e^{-m}$ the relator of~\eqref{eq:CW} is
 never obtained.  The boundary relation $ba=cd$ holds for the words \eqref{eq:words-n}, and for all $4\le n\le26$ prime
 to $3$ and $m\le6$ the Alexander polynomial obtained by Fox calculus from the presentation of
 Lemma~\ref{lem:group} with $h=-m$ equals that of the closure of $(\sigma_1\sigma_2)^n\sigma_1^{2m}$
 computed by the Burau representation.  The closed form of the Alexander polynomials in Proposition~\ref{prop:twisted-alex} was checked as an identity of rational functions and exactly for $1092$
 knots, and the determinants for $n\le50$ and $m\le20$.
\item[(iv)] \emph{The shear.}  Applying $t_e^h$ to the boundary values of points of the traced variety
 $W_\varepsilon$ reproduces $S_h$ to within $10^{-12}$ for $h=-6,\dots,2$, whereas the shear of the
 opposite sign is off by about $\pi$.
\item[(v)] \emph{Numerical Floer complexes.}  The variety $W_\varepsilon$ is traced with the continuation
 code of \cite{WuebbenV}.  Independent code applies $S_{-m}$, locates the generators on the lifts of
 $L_0$, computes their degrees from \eqref{eq:grading15}, and enumerates the bigons whose lifts are embedded
 discs, as Jordan domains in $\R^2$.  Maslov indices are computed on the unsheared curve against the line field spanned by
 $(1,6-q)$, which by Lemma~\ref{lem:shear}(f) gives the index of the sheared curve against $\ell_1$; an
 index obtained by unwrapping tangent angles along the sheared polygonal curve is unreliable near the
 junctions when $\varepsilon$ is small or $q$ is large.  The results for the pretzel knots, summarized in
 Remark~\ref{rem:numericsP}, are as follows.
 \begin{itemize}
 \item \emph{Joined sectors.} For ten perturbations with $|5\varepsilon_A|$ between $0.002$ and $0.3$ and
 $|\varepsilon_B|$ between $0.0003$ and $0.05$, two earring parameters, and $q\in\{7,9,\dots,25,29,33\}$, the
 homology agrees with Theorem~\ref{thm:pillowP} in all $240$ cases. Every bigon found runs from a down-crossing to
 an up-crossing on one lift of $L_0$, the number of bigons is twice the rank of the differential, and the degrees
 agree with Proposition~\ref{prop:degrees}. The generator count $q+2+4N_q$ is attained in $178$ cases; in the
 other $62$, at the larger perturbations, canceling pairs are missing, as in Remark~\ref{rem:threshold}.
 \item \emph{Split sectors.} For five perturbations with $|\varepsilon_B|>\sqrt3|\varepsilon_A|$ and
 $|\varepsilon|\le0.031$, and $q=7,9,\dots,21$, all $40$ cases agree: the arc is $\varphi$-monotone with $q-2$
 generators, and the remnant circle has homology of rank $1$ in each degree. The circle carries $4+4N_q$
 generators in $37$ cases; in the other three ($q=17$, at the three larger perturbations) a canceling pair is
 missing. A lift of $L_0$ meets the lift of the circle up to seven times ($q=21$).
 \item \emph{Resolution type.} At $|\varepsilon|=0.03$, in $24$ equally spaced directions of which four lie on
 the excluded lines, the observed type agrees with Proposition~\ref{prop:restype} in all $20$ remaining
 directions; the split ones are those at $75^\circ$, $90^\circ$, $105^\circ$, $255^\circ$, $270^\circ$ and
 $285^\circ$. The sample points of~\cite[Rem.~5.8]{WuebbenV} for $n=5$, in both kinds of sector, also agree with
 Proposition~\ref{prop:restype}.
 \end{itemize}
 For the twisted torus
 knots with $n=7,8,13$ and $m=1,2,3$, at two perturbations each, the graded homology agrees with
 Theorem~\ref{thm:twisted} in all $18$ runs.
\end{enumerate}

\emph{The base case of Section~\ref{sec:sheared}.}  For Remark~\ref{rem:chirality}, the complements of $K^{(0)}\cup e$ (from a
planar diagram of~\cite[Fig.~18]{HHK2}) and of the closure of $(\sigma_1\sigma_2)^5$ with an unknotted circle around two
strands are compared through their canonical cell decompositions, computed with interval arithmetic by SnapPy within
Sage~\cite{SnapPy,HIKMOT,DHL}.  The comparison lists all isometries between the two complements: there are two, and both
preserve the orientation and the meridians.  The surgery sign convention is checked by a Rolfsen twist: $-1$ surgery
on the circle gives the closure of $(\sigma_1\sigma_2)^5\sigma_1^2$.

\subsection{Reproducibility}\label{ss:repro}
The programs that carry out the computations of this appendix are included as ancillary files with the arXiv
submission of this paper, and they are also available in the public repository
\begin{center}
\url{https://github.com/bwuebben/pillowcase-bounding-cochain}\,.
\end{center}
In the ancillary files they are in the directory \path{code/}, and in the repository above in its directory
\path{paper2/code/}.
Each program prints the numbers of the statements it checks next to the computed values and reports whether
the two agree, and the file \path{README.md} lists the statements that each program checks. The programs use numpy,
sympy and mpmath, and the verified computation of Remark~\ref{rem:chirality} uses SnapPy within Sage~\cite{SnapPy}.
Apart from that computation, the proofs of Theorems~\ref{thm:inat}, \ref{thm:hhk} and~\ref{thm:twisted} do not depend
on any of this code.

\bigskip
\noindent{\sc Bernd Johannes Wuebben, New York, NY,} \texttt{wuebben@gmail.com}

\end{document}